\documentclass[12pt]{amsart}    
\usepackage{graphicx}
\usepackage{layout}
\usepackage{amsfonts}
\usepackage{amssymb}
\usepackage{hyperref}
\usepackage{amsthm}
\usepackage{amsmath}
\usepackage{verbatim}
\usepackage[nameinlink]{cleveref}
\numberwithin{equation}{section} 
\crefformat{equation}{#2(#1)#3}
\Crefformat{equation}{#2(#1)#3} 
\usepackage{mathrsfs}
\usepackage{faktor}
\usepackage{epsfig}
\usepackage{color}
\usepackage{stmaryrd}
\usepackage{cite}
\usepackage{soul} 

\usepackage[normalem]{ulem}
\usepackage{footmisc}
\usepackage{mathtools}

\newcommand{\BP}{\mathbf{P}}
\newcommand{\Div}{\mathrm{Div}}
\newcommand{\dv}{\mathrm{div}} 
\newcommand{\disc}{\mathsf{d}}
\newcommand{\PL}{\mathbf{PL}}
\newcommand{\CDiv}{\mathbf{CartierDiv}}
\newcommand{\gd}{\mathrm{Gd}}
\newcommand{\Tot}{\mathrm{Tot}}
\newcommand{\ugamma}{\underline{\gamma}}
\newcommand{\Z}{\underline{Z}}
\newcommand{\BC}{\mathbf{C}} 
\newcommand{\BQ}{\mathbf{Q}}

\newcommand{\BZ}{\mathbf{Z}} 
\newcommand{\BA}{\mathbf{A}}
\newcommand{\BR}{\mathbf{R}}

\newcommand{\Pt}{\widetilde{P}}
\newcommand{\Gt}{\widetilde{G}}
\newcommand{\Lt}{\widetilde{\mathscr{L}}}
\newcommand{\Nt}{\widetilde{\mathscr{N}}}
\newcommand{\uy}{\underline{y}}

\newcommand{\laurent}{\BZ(\!(q^{1/w})\!)}
\newcommand{\laurentp}{\BZ'(\!(q^{1/w})\!)}  
\newcommand{\pwr}{\BZ \llbracket q^{1/w} \rrbracket }   
  
\newcommand{\Sym}{\mathrm{Sym}}   
\newcommand{\Tate}{\mathrm{Tate}}
\newcommand{\Disc}{\mathrm{Disc}}
\newcommand{\spli}{\mathrm{split}}
\newcommand{\amplegg}{\mathrm{ample}, g, g}
\newcommand{\ISgg}{\mathrm{IS}, g, g}
\newcommand{\Frac}{\mathrm{Frac}}
\newcommand{\DD}{\mathrm{DD}}
\newcommand{\DEG}{\mathrm{DEG}}
\newcommand{\Pic}{\mathrm{Pic}}
\newcommand{\Proj}{\mathrm{Proj}}
\newcommand{\LL}{\mathscr{L}}
\newcommand{\ua}{\underline{a}}

\newcommand{\Mample}{\mathrm{Mumford-ample}}
\newcommand{\ad}{\mathrm{ad}} 
\newcommand{\adictorus}{\BG_{m}^{g, \mathrm{an}}}

\newcommand{\Spf}{\mathrm{Spf}}

\newcommand{\Spec}{\mathrm{Spec}}

\newcommand{\Spa}{\mathrm{Spa}}
\newcommand{\an}{\mathrm{an}}

\newcommand{\N}{\mathrm{N}}

\newcommand{\Cl}{\mathrm{Cl}}

\newcommand{\BG}{\mathbf{G}}

\newcommand{\GL}{\mathrm{GL}}

\newcommand{\Hom}{\mathrm{Hom}}

\hypersetup{colorlinks = true} 
\usepackage[margin = 1.06in]{geometry} 
\newtheorem{lemma}{Lemma}
\newtheorem{proposition}[lemma]{Proposition}
\newtheorem{theorem}[lemma]{Theorem}
\newtheorem{corollary}[lemma]{Corollary}

\theoremstyle{definition}
\newtheorem{construction}[lemma]{Construction}

\newtheorem{definition}[lemma]{Definition}

\numberwithin{lemma}{section}  
\theoremstyle{remark}
\newtheorem*{remark}{Remark}
\usepackage[msc-links, alphabetic, y2k]{amsrefs}

\usepackage[toc]{appendix} 
\newsavebox\ideabox  
\newenvironment{idea}
  {\begin{equation}
   \begin{lrbox}{\ideabox}
   \begin{minipage}{\dimexpr\columnwidth-2\leftmargini}
   \setlength{\leftmargini}{0pt}%
   \begin{quote}}
  {\end{quote}
   \end{minipage}
   \end{lrbox}\makebox[0pt]{\usebox{\ideabox}}
   \end{equation}}
\usepackage{tikz-cd} 
\usetikzlibrary{fadings}   
\makeatletter
\usepackage{enumitem}

\usepackage{etoolbox}
\AtBeginEnvironment{itemize}{\setlength\parskip{0pt}}
\AfterEndEnvironment{itemize}{\vspace*{-\dimexpr\parskip\relax}}
\AtBeginEnvironment{enumerate}{\setlength\parskip{0pt}}
\AfterEndEnvironment{enumerate}{\vspace*{-\dimexpr\parskip\relax}}
\AfterEndEnvironment{ex}{\vspace*{-2\dimexpr\parskip\relax}}

\tikzfading[name=soft edge,
  inner color=transparent!0,
  outer color=transparent!100]

\makeatletter
\DeclareRobustCommand*{\mfaktor}[3][]
{
   { \mathpalette{\mfaktor@impl@}{{#1}{#2}{#3}} }
}
\newcommand*{\mfaktor@impl@}[2]{\mfaktor@impl#1#2}
\newcommand*{\mfaktor@impl}[4]{
   \settoheight{\faktor@zaehlerhoehe}{\ensuremath{#1#2{#3}}}%
   \settoheight{\faktor@nennerhoehe}{\ensuremath{#1#2{#4}}}%
      \raisebox{-0.5\faktor@zaehlerhoehe}{\ensuremath{#1#2{#3}}}%
      \mkern-4mu\diagdown\mkern-5mu%
      \raisebox{0.5\faktor@nennerhoehe}{\ensuremath{#1#2{#4}}}%
}

\begin{document}
\linespread{1.1}
\setcounter{tocdepth}{2}  
\title[Fourier expansions for Kuga--Sato degenerations]{A Fourier expansion principle for coherent cohomology of line bundles on powers of Tate curves}
\author{Kenz Kallal}
\date{20 October 2026}  
\maketitle 
\begin{abstract}
In the case of the totally degenerate split degenerations given by fibered powers of Tate curves over $\pwr$, this article proves a Fourier expansion principle that applies to higher coherent (equivariant) cohomology classes of rigidified line bundles on the $\laurent$-fiber. We do not require any ampleness hypotheses. This generalizes a previous Fourier expansion principle due to Faltings--Chai that applies only to sections of ample line bundles; that version was used as a key step in their proof of their classification of polarized degenerations and thus in the construction of toroidal compactifications of moduli spaces of polarized abelian varieties. This article can be considered as a theory of Fourier expansions for algebraic theta cocycles associated to possibly indefinite quadratic forms, and has application to arithmetic: in joint work with Venkatesh, we have used it to produce an algebraic theory of Kudla--Millson indefinite modular theta functions. 
  \end{abstract}
\tableofcontents

\section{Introduction}\label{intro}

The purpose of this article is to prove a Fourier expansion principle for higher coherent cohomology classes of line bundles on certain  arithmetic (meaning: over $\BZ$ or other number ring or localization thereof) degenerating families of abelian varieties. 

\subsection{An illustration in a simple analytic case}
To illustrate what we mean by a \emph{Fourier expansion principle} in this context, we begin with the simplest possible degenerating family of abelian varieties --- the familiar analytic family of elliptic curves occurring near a cusp of a modular curve. Over a small punctured disc $\Delta^*$ around $0$, it is given over $q \in \Delta^*$ by
\[E_q := \BC^\times / q^{\BZ} \simeq \BC / (\BZ + \tau \BZ), \qquad q = \exp(2 \pi i \tau)\]
(In this article we will always take the point of view of multiplicative uniformization of abelian varieties, and use $Z$ to denote the variable in $\BC^\times$). The family $E_q$ degenerates as $q \to 0$ to a copy of $\BC^\times$, and it admits a family of principal polarizations \cite[p. 36]{mumfordAV}, that is, isomorphisms $E_q \simeq E_q^\vee$ which may be induced by the ample line bundle\footnote{To be precise one defines $\LL$ by defining its total space to be the quotient of $\BC^\times \times \BC$ by a suitable action of $q^{\BZ}$. In order to define it, one has to access $q^{1/2}$ in a family over all of $\Delta^*$ --- so we really define $\Delta^*$ to at least be a double cover parametrized by $q^{1/2}$ of the punctured disc parametrized by $q$. For a modular curve, this would be a neighborhood of a cusp of even width.} $\mathscr{L}_q$ on $E_q$ whose global sections are
\begin{equation}\label{analytic-ample}H^0(E_q, \mathscr{L}_q) = \{\text{holomorphic } \vartheta: \BC^\times \to \BC \text{ such that } \vartheta(qZ) = q^{-1/2}Z^{-1}\vartheta(Z)\}.\end{equation}
We consider \eqref{analytic-ample} to be a \emph{Fourier expansion principle} for global sections of the ample line bundle $\mathscr{L}_q$, since it allows one to calculate an explicit Fourier series basis for the global sections of $\LL_q$. In particular, \eqref{analytic-ample} implies that $H^0(E_q, \LL_q)$ is spanned by Jacobi's theta series \[\vartheta_3(q, Z) := \sum_{n \in \BZ}q^{n^2/2}Z^n.\] This expansion around $Z = 0$ can be viewed as a family parametrized by $q^{1/2}$ of Fourier series, because $Z$ is already defined to be \[Z = \exp\left(2 \pi i \left(\text{the additive coordinate on } \BC\right)\right).\]

The main result of this paper is the generalization of \eqref{analytic-ample} in multiple simultaneous directions:
\begin{itemize}
\item The degenerating family $E_q$ is replaced with its arithmetic analog, the Tate curve over $\laurent$; and the higher rank version $E_q^g$ is also considered.
\item The choice of line bundle $\LL$ on $E_q^g$ can be \emph{arbitrary}, so long as it is rigidified: in particular, there is no longer any ampleness restriction.
\item The degree of cohomology can be \emph{arbitrary}: our Fourier expansion principle applies simultaneously to all $H^i(E_q^g, \LL)$.
\end{itemize}

The last two points are well-known in the analytic context --- it is our transfer of them over to powers of the Tate curve in the arithmetic context that is novel. We refer the reader more interested in the full detail of the statement to \Cref{shortresult} and \Cref{veryconvergent}. We also provide translations of two other canonical pieces of structure into the explicit world of the Fourier expansion principle: the theta group action on the higher cohomology of a line bundle (\Cref{heis}), and evaluation morphisms on \emph{equivariant} coherent cohomology (see \Cref{intro2} for the statement, and \Cref{equivariant} for the proof). These other results, and their utility to the algebraic theory of indefinite theta functions \cite{KV} are also briefly outlined later in this introduction, in \Cref{app}. 

Let us now provide simple cases of the second and third points in the analytic setting, in order to show the interest and basic form of such a statement --- the rest of this article is devoted to making them apply in the arithmetic setting described in the first point. The family of elliptic curves $E_q$ admits families of line bundles that are \emph{not ample}. For instance, one can simply take the family of bundles $\LL_q^{-1}$ on $E_q$. Its coherent cohomology can be studied abstractly via \cite[\S 16]{mumfordAV}, but this article is more concerned with explicit computations generalizing \eqref{analytic-ample}, which we now perform in this case. If $\LL_q^{-1}$ had global sections, its global sections would consist of holomorphic $\vartheta: \BC^\times \to \BC$ with the property that $\vartheta(qZ) = q^{1/2}Z\vartheta(Z)$. Considering the effect of $\vartheta \mapsto q^{-1/2}Z^{-1}\vartheta(qZ)$ on formal power series $\vartheta$ in $Z^{\pm 1}$, the only possible such $\vartheta$ would be the na\"ive theta series for the \emph{negative-definite} quadratic form $-X^2$, that is, the series
\[\sum_{n \in \BZ} q^{-n^2/2}Z^{-n}.\] 
For any fixed value of $q^{1/2} \in \Delta^*$, this series does not converge, so it cannot define an actual global section of $\LL_q^{-1}$ --- hence $\LL_q^{-1}$ has no global sections. Instead, it has nonvanishing cohomology in degree $1$. A standard spectral sequence argument based on the vanishing of the coherent cohomology of the structure sheaf of $\BC^\times$ (as $\BC^\times$ is Stein) and the fact that $\LL_q^{-1}$ pulls back under the covering map $\BC^\times \to E_q$ to the holomorphic structure sheaf of $\BC^\times$ endowed with a $q^\BZ$-equivariant structure which acts on sections $f: \BC^\times \to \BC$ by $(q \cdot f)(Z) := q^{-1/2}Z^{-1}f(qZ)$
shows that
\begin{equation}\label{analytic-nonample}H^1(E_q, \LL_q^{-1}) \simeq H^1(q^\BZ, \mathcal{O}(\BC^\times)), \text{ where } (q \cdot f)(Z) := q^{-1/2}Z^{-1}f(qZ).\end{equation}
This generalizes the Fourier expansion principle \eqref{analytic-ample}, because the right hand side of \eqref{analytic-ample} can equally be interpreted as $H^0(q^{\BZ}, \mathcal{O}(\BC^\times))$, where $q$ acts on $\mathcal{O}(\BC^\times)$ via $(q \cdot f)(Z) := q^{1/2}Z^1f(qZ)$. The higher-cohomology version equally deserves the name \emph{Fourier expansion principle}, because the group cohomology $H^1(q^\BZ, \mathcal{O}(\BC^\times))$ admits an explicit basis of elements represented by explicit cocycles whose values are Fourier series. In fact, for this particular action, $H^1(q^\BZ, \mathcal{O}(\BC^\times))$ is spanned by the class of the explicit homogeneous $1$-cocycle
\[(q^{y_0}, q^{y_1}) \mapsto \begin{cases}\sum_{y_0 + \frac{1}{2} < n < y_1 + \frac{1}{2}} q^{-n^2/2}Z^{-n} \text{  if } y_0 \leq y_1 \\
-\sum_{y_1 + \frac{1}{2} < n < y_0 + \frac{1}{2}} q^{-n^2/2}Z^{-n} \text{  if } y_1 \leq y_0 
\end{cases},\] 
where the finiteness of the oriented sums (enabled by the fact that this is an $H^1$ class rather than a global section) salvages the non-convergence of the na\"ive theta series. This explicit construction of cohomology classes is generalized in \cite[\S I]{KV}.  
 
There is an obvious generalization of \eqref{analytic-nonample} (with exactly the same proof as indicated above) for an arbitrary line bundle on an abelian variety. In fact, Zharkov \cite{zharkov} has studied explicit bases of the resulting group cohomology. 

\subsection{Statement of main result}\label{statements}
Fix positive integers $g$ and $w$. 
This article provides a generalization of \eqref{analytic-nonample} in a different direction: as mentioned in the first sentence of this introduction, we prove an analog of it in the case of certain \emph{arithmetic} degenerating families. Although there is a theory of such degenerating families in a large level of generality \cite{FC, lan}, we will content ourselves with studying the arithmetic analog of an analytic situation only slightly more general than that of \eqref{analytic-nonample}. Namely, we consider the analytic family parametrized by $q^{1/w} \in \Delta^*$ whose fiber over $q^{1/w}$ is $E_q^g$. The arithmetic analog of that family is just the $g$-th fibered power of the Tate curve $E_{w, \Tate}$ over $\laurent$. 

The core of this article consists of the following two steps:
\begin{enumerate}
\item[(I)] Following \cite{FC} and \cite{lan}, we will classify the rigidified line bundles $\mathscr{L}$ on $E_{w, \Tate}^g$ in terms of linear-algebraic data analogous to the descent datum for the periods $q^\BZ$ acting on $\BC^\times$ that defines $\LL_q$. 
\item[(II)] We will calculate the coherent cohomology $H^i(E_{w, \Tate}^g, \LL)$ in terms of group cohomology of $\BZ^g = (q^\BZ)^g$ on a certain module of ``very convergent'' power series, where (as in \eqref{analytic-nonample}) the group action of  $\BZ^g$ will depend on the linear algebraic data from (I) that defines $\LL$. In symbols, we prove :
  \begin{equation}\label{symbols} 
H^\bullet\left(E_{w,\mathrm{Tate}}^g,\LL\right)
\simeq
H^\bullet\left(
  \mathbf{Z}^g,
  \begin{gathered}
    \text{a module of convergent power series}\\
    \text{in } Z_1^{\pm 1},\ldots,Z_g^{\pm 1}
    \text{ over } \laurent
  \end{gathered}
\right).\end{equation}
  where the group action on the right hand side depends on $\LL$.  
  This is the claimed Fourier expansion principle for higher cohomology of line bundles on the $g$-th power of the Tate curve, and constitutes the main theorem of this paper. We delay the exact definition for the module of convergent power series in $\underline{Z}^{\pm 1}$ for \Cref{veryconvergent} and \Cref{shortresult}, by which point the definition will be better motivated.  
\end{enumerate} 

The explicit Fourier expansion principle \Cref{symbols} has more utility than the knowledge of the $\laurent$-rank of $H^\bullet(E_{w, \Tate}^g, \LL)$ (which is standard --- see \cite[Lemma 3.2]{KV}). Indeed, it allows one to construct explicit \emph{elements} of $H^\bullet(E_{w, \Tate}^g, \LL)$, which are explicit algebraic theta cocycles near a cusp of a modular curve. This is the point of view taken in \cite{KV}, where the Fourier expansion principle \Cref{symbols} (along with some somewhat more refined information --- see \Cref{app} for details) is used to prove an integral version of Kudla--Millson theory, as well as a version of the theory of algebraic theta functions \cite{mumfordeq, tataiii} that applies to higher cohomology of non-ample line bundles and is more refined than the original theory due to the ability to prove the modularity of specific $q$-expansions (including Kudla--Millson $q$-expansions valued in group cohomology).

We now further explicate both steps in the case $g = 1$, in order to give a clear overall picture of what will take place in the body of this article. The arithmetic replacement in \cite{FC, lan} for the multiplicative uniformization of $E_q$ by $\BC^\times$ is a \emph{relatively complete model} $\Pt$, which is a kind of compactification of $\BG_m$ for which a quotient by the periods ``$q^\BZ$'' may actually be defined. In fact, since we are taking $g = 1$, the standard choice of relatively complete model $\Pt$ and the procedure for quotienting by $q^\BZ$ date back at least to \cite[\S VII]{DR}. For the purposes of this introduction, we omit some features of $\Pt$, but it is relevant for now that $\Pt$ is $\pwr$-scheme containing a Zariski-dense copy of $\BG_{m, \pwr}$ such that the $\laurent$-fiber of $\Pt$ is identified with $\BG_{m, \laurent}$. On the other hand, the $q^{1/w} = 0$ fiber of $\Pt$ is a N\'eron $\infty$-gon: an infinite chain of copies of $\BP^1_{\BZ}$ glued together along $0$ and $\infty$. The scheme $\Pt$ also admits an action of the abstract group of periods $Y := \BZ$ restricting on $\BG_{m, \laurent}$ to the action of $Y$ given by the inclusion $y \mapsto q^{y}$ and the group law of $\BG_m$. A relatively complete model $\Pt$ uniformizes the Tate curve over  $\laurent$ in the following sense: 
\begin{idea}\label{qt}
The formal completion $\Pt^\wedge$ of $\Pt$ along $q^{1/w}$ has a quotient by the action of $Y$, $\Pt^\wedge/Y$, which algebrizes uniquely to a $\pwr$-scheme $P_w$ whose $\laurent$-fiber is precisely the Tate curve. The $q^{1/w} = 0$-fiber of $P_w$ is a N\'eron $w$-gon. 
\end{idea}

The process of taking a relatively complete model, formal completion and algebrization are the workaround for the fact that one cannot literally take the quotient of the scheme $\BG_m$ by the action of $Y$. The algebrization step is provided by Grothendieck's algebrization theorem \cite[Th\'eor\`eme (5.4.5)]{EGAIII}, which requires the auxiliary data of an ample line bundle --- this is discussed in detail in the body of this article. We refer the reader unfamiliar with this story to the extremely satisfying diagrams of \cite[\S 5]{mumford_compositio}. Indeed, the uniformization construction of \Cref{qt} was originally worked out in generality by Mumford in \cite{mumford_compositio}, which was the basis for much of \cite{FC, lan}. 

As we will make explicit in \Cref{degens}, the classification of line bundles on degenerations in \cite{FC, lan} assures us that for all rigidified line bundles $\LL$ on $E_{w, \Tate}$, there is a $Y$-equivariant line bundle $\Lt$ on $\Pt$ inducing (by algebrization following the quotient procedure of \Cref{qt}) $\LL$. That $Y$-equivariant structure on $\Lt$, restricted to $(\BG_m)_{\laurent}$ (where the restriction of $\Lt$ is just the structure sheaf) is the linear-algebraic datum alluded to in (I) --- indeed, it adapts the procedure of uniformizing a line bundle on $E_q$ by an action of the periods $q^\BZ$ on the trivial line bundle on $\BC^\times$ to the quotient procedure of \Cref{qt}. 

It remains to explain (II), where we will for now put the emphasis on an explanation that allows us to arrive at the correct definition of the module of formal power series required in the right hand side of \Cref{symbols} rather than on the exact method of proof that we will end up using. For the purpose of convincingly arriving at the correct answer for that definition, there is a more convenient viewpoint than that of \Cref{qt}:  $\Pt^\wedge$ is a formal model for $\BG_m$ the adic space over $\Spa(\laurent, \pwr)$ --- to avoid confusion, we denote this adic space by\footnote{As opposed to $\BG_m^{\ad}$, to emphasize that we are taking the analytic points of $\BG_m$, which is an adic space over $\Spa(\laurent, \pwr)$ rather than the larger $\Spa(\pwr, \pwr)$.} $\BG_m^{\an}$, and remark that one of Huber's first intended applications of his theory of adic spaces (see \cite[\S 5]{hubergen}) was to create an analytic theory of uniformization of degenerations of abelian varieties over rings like $\pwr$. Indeed, Tate's rigid spaces \cite{tate, BGR} were invented to provide a nonarchimedean analytic theory over the \emph{field} $k((q^{1/w}))$ over which the Tate curve could literally be considered as a quotient of rigid-analytic $\BG_m$ by the periods $q^\BZ$. Then Huber \cite[\S 5]{hubergen} showed (as a special case of his more general results) that the Tate curve produced by means of \eqref{qt} also algebrizes the adic space $\BG_m^\an/Y$ where again the abstract group of periods $Y = q^\BZ$ acts on $\BG_m^{an}$ by multiplication by $q^y$.

Our technique in (II), restricted for this introduction to the case $g = 1$, is to emulate the analytic strategy of using a spectral sequence and a cohomology vanishing result for the uniformizing cover --- namely, a version of Cartan's Theorem B that applies to the adic space $\BG_m^\an$. That version has been proved by Kedlaya--Liu \cite[Theorem 2.6.5]{KL2} using ideas closely related to the proof of Cartan's Theorem B for rigid analytic spaces \cite{kiehl}. One plausible way to deduce a Fourier expansion principle from this cohomology vanishing result is to apply a GAGA result for adic analytifications of proper morphisms of schemes: then the cohomology of a line bundle $\LL$ on $E_{w, \Tate}$ is the same as the cohomology of $\LL^\an$ on $\BG_m^{\an}/Y$, which fits into a spectral sequence
\[E_2^{p, q} = H^p(Y, H^q(\BG_m^{\an}, \mathcal{O})) \Longrightarrow H^{p + q}(\BG_m^\an/Y, \LL^{\an})\]
assuming $\LL^{\an}$ is obtained from the structure sheaf of $\BG_m^\an$ by descent under the $Y$-action.  
The vanishing of the higher cohomology of the structure sheaf of $\BG_m^\an$ (that is, Kedlaya--Liu's Cartan Theorem B) implies the desired arithmetic Fourier expansion isomorphism 
\[H^p(Y, \mathcal{O}(\BG_m^{\an})) \simeq H^p(\text{Tate curve}, \LL)\]
the left hand side of which tells us exactly how to specify the full detail of the Fourier expansion principle for the Tate curve ---  what the module of power series should be on the right hand side of \Cref{symbols}. It should be the global sections of the structure sheaf of $\BG_m^\an$. In concrete terms, the ring of global functions on $\BG_m^\an$ is the $\laurent$-algebra of power series $\sum_{n \in \BZ} a_n Z^n$ with coefficients $a_n \in \laurent$, which are required to converge $q^{1/w}$-adically over \emph{all} of $\BG_m^\an$ considered as an increasing union of affinoid annuli. That requirement is interpreted concretely as requiring
\begin{idea}
  The $q^{1/w}$-adic valuation of $a_n$ goes to $\infty$ as $n \to \pm \infty$, faster than any linear function $n \mapsto xn$, $x \in \BR$ --- in the sense that $v_{q}(a_n) + xn \to \infty$ as $n \to \pm \infty$ for all $x \in \BR$. 
\end{idea}
In general, for the rank-$g$ version, the required module $\mathcal{O}(\BG_m^{\an, g})$ of global sections is defined by
\begin{definition}\label{veryconvergent}
  $\mathcal{O}(\BG_m^{\an, g})$ be the $\laurent$-module of power series in $Z_1^{\pm 1}, \ldots, Z_g^{\pm 1}$ consisting of
  \[\sum_{\chi \in \BZ^g} a_\chi \underline{Z}^\chi\]
  such that for all $x \in \Hom(\BZ^g, \BR)$,
  \[v_q(a_\chi) + \langle x, \chi \rangle \to \infty\]
  as $|\chi| \to \infty$, in the sense that all but finitely many $v_q(a_\chi) + \langle x, \chi \rangle$ are bounded below by any given constant. We call such power series \emph{very convergent}.
\end{definition}

And a choice of rigidified line bundle $\LL$ on the $g$-th power of the Tate curve on $\laurent$ defines a specific action of $\BZ^g$ on $\mathcal{O}(\BG_m^{\an, g})$, after which (by exactly the same argument as above, applied to $\BG_m^{\an, g}$ instead of $\BG_m^\an$) our main result may be concisely stated as the Fourier expansion principle
\begin{equation}\label{shortresult}
H^p(\BZ^g, \mathcal{O}(\BG_m^{\an, g})) \simeq H^p(g\text{-th power of Tate curve}, \LL).
\end{equation}

Here the action of $\BZ^g$ on $\mathcal{O}(\adictorus)$ depends on $\LL$ and is of the form
\begin{equation}\label{finalaction}(\underline{y} \cdot f)(\underline{Z}) :=  f(q^{\underline{y}}\underline{Z}) q^{Q'(\underline{y})}\underline{Z}^{B\underline{y}} \cdot \prod_{j = 1}^g a_j(q^{1/w})^{y_j},\end{equation}
where $Q(x) = \frac{1}{2}x^\top B x$ is a quadratic form on $\BZ^g$ with $B$ a symmetric $g \times g$ $\BZ$-entried matrix; $a_j(q^{1/w}) \in \laurent^\times$ are data that determine the choice of $\LL$; and $Q'$ denotes the integer-valued quadratic function $Q(x) - \frac{1}{2}\sum_i b^\top x$, where $b := (B_{11}, \ldots, B_{gg})^\top$.

Since it requires more input from \cite{FC, lan} or otherwise some theory of descent of coherent sheaves on adic spaces, we delay the full detail of the classification of possible $\LL$ in terms of $Q$ and $a_1, \ldots, a_g$ until \Cref{degens}.  

\subsection{Some ideas of our proof}\label{ideas}

Wanting to minimize the amount of theory of adic spaces that we must use, we do not literally use the technique of proof indicated in \Cref{statements}. Let $g = w = 1$ for simplification again.

Without using adic spaces, one can consider $\Pt^\wedge$ to be an object in the category of ``rigid spaces over $\pwr$'', that is, a category of certain formal $\pwr$-schemes localized by certain blow-ups \cite{BL, bosch, abbes}. Huber \cite[(1.1.12)]{huberbook} \cite[\S 3.9]{regensburg} provided a functor from such a category to a category of adic spaces over $\Spa(\laurent, \pwr)$ which is an equivalence. For that equivalence as literally proved in those works to apply in this setting, one would have to weaken the ``finite type'' hypothesis in \cite{huberbook, hubergen} to ``quasi-paracompact,'' following an argument that already exists in the rigid analytic setting \cite{bosch}. This is because $\Pt$ is not of finite type over $\pwr$, being an infinite union of affine schemes that are individually formal models for annuli in $\BG_m^{\an}$. Rather than establish such an equivalence, we simply use the fact that it is likely true to motivate the idea that we may work directly with the formal models and the procedure \Cref{qt} following \cite{FC, lan} to prove the result \Cref{shortresult}. This is in contrast to the fact that the most expedient way to motivate the statement of \Cref{shortresult} was purely in terms of adic spaces.

Working directly with a relatively complete model $\Pt$, we will apply formal GAGA from \cite[Th\'eor\`eme (4.1.5)]{EGAIII} to equate the base-change-to-$\laurent$ of the cohomology of $\LL$ with the same base change of the cohomology of $\LL^\wedge$ on $\Pt^\wedge/Y$. We will show that this cohomology is converged to by a spectral sequence with $E_0$-page
\begin{equation}\label{ssintro}C^p(Y, \laurent \otimes (\text{completed \v Cech complex for  }\Pt))\end{equation}
and vertical differentials. 
We stress that the informal notation \[\laurent \otimes (\text{completed \v Cech complex for  }\Pt)\] that we use for now to shorten equations should be understood to mean the complex whose terms take the form $\prod_{\underline{i}} \laurent \otimes \Gamma(U_{\underline{i}}, \Lt)^\wedge$. Contrary to what is easy to express in natural language, base-change to $\laurent$ should occur on the inside, not the outside, of the infinite product.

We will explicitly identify the inner complex $\laurent \otimes (\text{completed \v Cech complex for  }\Pt)$ with a \v Cech complex that computes the cohomology of the structure sheaf on $\BG_m^{\an, g}$ --- which allows us to employ the higher cohomology vanishing result \cite[Theorem 2.6.5]{KL2} as above to conclude \eqref{shortresult}. The advantage of this version of the method is that the argument can be made to work with minimal input from the theory of adic spaces --- the only time we use them is when invoking the cohomology vanishing result from \cite{KL2}, but this requires no additional machinery besides explicitly identifying \v Cech complexes on the formal and adic sides, and could easily be replaced at the cost of additional notation by copying the proof of \cite{KL2} to directly prove the vanishing of the higher cohomology of the inner complex.

\subsection{Why, and how to invert $q^{1/w}$ ?}\label{invertq}
In the previous paragraphs, all of the calculations are made \emph{after passing from } $\pwr$ \emph{to} $\laurent$, that is, away from the locus $q^{1/w} = 0$ in the proper $\pwr$-scheme $P_w$ algebrizing $\Pt^\wedge/Y$. Recall that the $q^{1/w}$-fiber of $P_w$ is a N\'eron $w$-gon, identified as the quotient by $Y$ of the N\'eron $\infty$-gon that is the $q^{1/w}$-fiber of $\Pt$. In the logic of the proof described above, this restriction is important because the nonarchimedean function theory used to prove Cartan's Theorem B in \cite{KL2} does not go through without first inverting $q^{1/w}$. 

If $q$ is not inverted, it is not just that our \emph{proof} fails. The proper $\pwr$-scheme $P_w$ whose cohomology we would like to calculate would now depend not only on the choice of $w, g$, but also on the choice of relatively complete model $\Pt$. Once $\Pt$ is chosen, there is no reason for the spectral sequence calculating the cohomology (the one from \Cref{ssintro} but without the base change to $\laurent$) to have second page concentrated in the first row and give a result resembling \Cref{shortresult}.  Indeed, let us take the standard choice of $\Pt$ described in \Cref{statements}, and let $\mathscr{L}$ be an ample line bundle on $E_{w, \Tate}$ inducing a principal polarization, with model $\Lt$ on $\Pt$, so that $\Lt^{-1}$ is a model for the non-ample line bundle $\LL^{-1}$.  The problem is that the higher cohomology of $\Lt^{-1}$ on the $q = 0$ fiber $\widetilde{P}_0$ does not vanish, as we now prove. The restriction of $\Lt^{-1}$ to $\Pt_0$ is obtained by gluing together copies of $\mathscr{O}(-1)$ on each copy of $\BP^1_\BZ$ making up $\widetilde{P}_0$. Reflecting this, there is an exact sequence
\[0 \to H^0(\widetilde{P}_0, \Lt^{-1}) \to \prod_{n \in \BZ} H^0(\BP^1, \mathcal{O}(-1)) \to \prod_{n \in \BZ} \BZ \to H^1(\widetilde{P}_0, \Lt^{-1}) \to \prod_{n \in \BZ} H^1(\BP^1, \mathcal{O}(-1)) \to 0,\]
from which one obtains the isomorphism
\[0 \to \prod_{n \in \BZ} \BZ \overset{\sim}{\to} H^1(\widetilde{P}_0, \Lt^{-1}) \to 0\]
and the fact that $H^1(\Pt_0, \Lt^{-1}) \neq 0$ (in fact it is infinite-rank). By vanishing of $H^2$, this implies that the \v Cech complex on the inside of \Cref{ssintro} does not have vanishing cohomology in degree $q = 1$. 

It is interesting to note that this phenomenon does \emph{not} occur when dealing with $\LL$ instead of $\LL^{-1}$: the relevant exact sequence is then
\[0 \to H^0(\widetilde{P}_0, \widetilde{L}) \to \prod_{n \in \BZ} H^0(\BP^1, \mathcal{O}(1)) \to \prod_{n \in \BZ} \BZ \to H^1(\widetilde{P}_0, \widetilde{L}) \to \prod_{n \in \BZ} H^1(\BP^1, \mathcal{O}(1)) \to 0,\]
which is
\[0 \to H^0(\widetilde{P}_0, \widetilde{L}) \to \prod_{n \in \BZ} (\BZ x_0 + \BZ x_1) \to \prod_{n \in \BZ} \BZ \to H^1(\widetilde{P}_0, \widetilde{L})\to 0,\]
where the map $\prod_{n \in \BZ} (\BZ x_0 + \BZ x_1) \to \prod_{n \in \BZ} \BZ$ is by sending the tuple of linear forms $(a_n x_0 + b_n x_1)_n$ to the tuple $(b_n - a_{n+1})_n$. This is surjective, so $H^1(\widetilde{P}_0, \widetilde{L}) = 0$. It is true that the ample version of the Fourier expansion principle established by Faltings--Chai \cite[Ch. III, Theorem 6.2]{FC} (see also \cite[Corollary 4.5.4.24]{lan}) still requires one to pass to the generic fiber of a degeneration before taking cohomology; it would be interesting to see if versions of their result or ours could be established without inverting $q^{1/w}$ in some cases.

The need to invert $q^{1/w}$ introduces another technical issue related to ``unbounded denominators'': tensoring a \v{C}ech complex of infinite products of $\pwr$-modules by $\laurent$ does not commute with the infinite products. We bypass this issue by carefully using a fundamental domain for the action of $Y$ where the \v{C}ech complex consists only of finite products; much of this is conveniently absorbed into the use of Shapiro's lemma. It is for this reason that we must be so careful in the definition of the double complex \Cref{ssintro} to make sure that the base change to $\laurent$ occurs inside all of the infinite products. 

\subsection{Relationship to Faltings--Chai's Fourier expansion principle}\label{periods-recover}

In the general context of degenerations of abelian varieties, Faltings--Chai proved the special case $i = 0$ of our main result \eqref{shortresult} --- we have already referenced this appearing in \cite[Ch. III, Theorem 6.2]{FC} and in \cite[Corollary 4.5.4.24]{lan} in a form more similar to our statement. In order to make clear the relationship between their proof and ours, we now briefly explain their technique. Let $G/\pwr$ be the connected component of the smooth locus of the Tate curve over $\pwr$, and let $\LL$ be a rigidified ample line bundle on $G$. In particular, the $q^{1/w}=0$ fiber of $G$ is just a copy of $\BG_m/\BZ$. Then Faltings--Chai prove the $i = 0$ case of \eqref{shortresult} by taking the formal completion of $G$ along $q^{1/w} = 0$. Being a formal deformation of $\BG_m$, that formal completion is automatically $\BG_m / \Spf(\pwr)$. Therefore, global sections of $\LL$ automatically induce sections of the structure sheaf of $\BG_m / \Spf(\pwr)$, that is, elements of the $q^{1/w}$-adic formal completion of $\pwr[Z, Z^{-1}]$. So any global section $s$ of $\LL$ is identified via this procedure with its Fourier expansion
\[\widehat{s} = \sum_{\chi} a_\chi Z^\chi, \qquad a_\chi \to 0 \,\,\,\,\,\,q\text{-adically as } \chi \to \infty.\]
Similarly, global sections of $\LL$ over the base change to $\laurent$ induce formal power series in $Z, Z^{-1}$ with coefficients in $\laurent$ whose coefficients go to zero $q$-adically. Faltings--Chai \cite[Ch. II]{FC} use the addition formulae for algebraic theta functions to prove that the Fourier expansions associated to global sections over the base change to $\laurent$ are exactly the power series as above satisfying a certain quasiperiodicity condition amounting to being fixed by a certain action of $Y = q^\BZ$. That is, they prove exactly the statement \Cref{shortresult} in the case where $\LL$ is ample and $i = 0$. This method is highly sensitive to the fact that one is considering \emph{global sections} and not higher cohomology, which is why the method described in \Cref{ideas} at least seemed necessary despite the previous result of Faltings--Chai.

It is worth noting that their Fourier expansion principle for global sections was a key step in establishing their classification of line bundles on degenerations with ample generic fiber, which we use in step (I) described in \Cref{statements}. But their strategy for extending that equivalence to arbitrary line bundles was just to extend linearly from the ample case, so they had no need for any analog in the case of higher cohomology. And in any event, the main purpose of their classification of degenerations at the time of \cite{FC, lan} was to construct arithmetic compactifications of PEL Shimura varieties --- so they only needed to classify degenerations of ample line bundles. Due to arithmetic applications that are surfacing (such as the algebraic theory of indefinite theta functions \cite{KV} and the general principle appearing in \cite{KS, BCG, kopp2} that higher equivariant cocycles are relevant to problems in arithmetic), the higher cohomology version has now become relevant.

\subsection{Further results also present in this article, and applications to arithmetic}\label{app}
We also provide a number of auxiliary results related to how various other canonically-defined structures interact with our Fourier expansion principle.
\begin{enumerate}
\item Line bundles $\LL$ on the abelian schemes $E_{w, \Tate}^g/\laurent$ acquire additional structure (as they do for general abelian schemes), namely a \emph{theta group scheme} $\mathcal{G}(\LL)$ consisting of automorphisms of $\LL$ covering translations on $E_{w, \Tate}^g$. The group $\mathcal{G}(\LL)$ naturally acts on the cohomology of $\LL$. In the case where $\LL$ is symmetric and $Q$ is nondegenerate, we calculate an explicit model for this action.
\item When $\LL$ is defined by twisting a given symmetric rigidified line bundle on $E_{w, \Tate}$ by a quadratic lattice $Q$ as in \cite[\S 3]{KV}, it admits a canonical equivariant structure for any subgroup $\Gamma$ of the orthogonal group of $(Q, \BZ^g)$, covering the natural action of $\Gamma$ on $E_{w, \Tate}^g$. In that case, we extend our Fourier expansion principle to also apply to the $\Gamma$-equivariant coherent cohomology of $\LL$.
\item One advantage of considering the $\Gamma$-equivariant cohomology of a line bundle is that it offers a way to evaluate coherent cohomology classes for $\LL$: even in degree $i > 0$, such an evaluation produces a group cohomology class for $\Gamma$, as opposed to zero. In symbols, there is an evaluation map
  \[H^i_\Gamma(E_w^g, \LL) \to H^i(\Gamma, \laurent),\]
  as $\Gamma$ acts trivially on the rigidification. 
  Under the $\Gamma$-equivariant Fourier expansion principle that we will prove, we provide the exact explicit interpretation of the evaluation map. 
\end{enumerate}

All three of these extra points are used in \cite{KV} to prove an algebraic version of the Kudla--Millson theory of theta functions for signature-$(r, s)$ indefinite quadratic forms $Q$ \cite{kudlamillson}, with coefficients in $H^s(\Gamma, \BZ[1/(2\text{disc}(Q))])$, thus algebrizing what Kudla--Millson constructed with coefficients in $H^s(\Gamma, \BC)$. (1) is used there to study the monodromy of a deprojectivized algebraic version of the heat equation; (2) is used to produce explicit $\Gamma$-equivariant cohomology classes which are globalized by means of trivializing the monodromy of the heat equation, and (3) is used in order to evaluate the aforementioned cocycles along the identity section and calculate explicitly the resulting modular form valued in the cohomology of $\Gamma$, so that it may be compared with the ones defined by \cite{kudlamillson}.
\subsection{Outline of this article} \Cref{notation} contains a summary of notation to be used in this article, for convenience.

In \Cref{ana}, we review the analytic theory of Fourier expansions in higher cohomology, to make explicit the analytic analog of the general situation considered in this article, including the classification of line bundles. This section contains well-known material and is not logically necessary for what follows, but might be helpful for orienting the reader. 

In \Cref{intro2}, we give the definition of the Kuga--Sato degenerations over $\pwr$ including the line bundles $\LL$ on them; and we restate our Fourier expansion principle \Cref{shortresult} for them, in the notation that we set up in that section. 

In \Cref{degens}, we recall the general theory of the arithmetic degenerating families considered by Faltings--Chai, including the $i=0$ ample case of \Cref{shortresult} that they proved, and how it was used to classify the degenerating families that they considered. We also show how the Kuga--Sato degenerations and line bundles on them fit into that classification, thus certifying that the Fourier expansion principle stated in \Cref{intro2} exhausts all the rigidified line bundles on Kuga--Sato degenerations. 

\Cref{nonequiv} contains the proof of the promised Fourier expansion principle \Cref{shortresult}, and \Cref{equivariant} contains an augmented version of the proof that upgrades the result to a Fourier expansion principle for equivariant cohomology as promised in \Cref{app}.

Finally, \Cref{heis} contains the explicit detail of how the cohomology provides a model of the Stone--von Neumann representation of the relevant finite Heisenberg group, as also promised in \Cref{app}.

\section*{Acknowledgements} I thank Akshay Venkatesh, Ahmed Abbes, Brian Conrad, and Kai-Wen Lan for various useful conversations about the ideas in this work.
\section*{AI use disclosure} AI was used to check for mistakes and to generate all of the diagrams in this article. 
\section{Notation}\label{notation} 
We have already introduced much of this notation in \Cref{intro}, but we store all of our conventions here for the convenience of the reader.
\begin{itemize}
\item $\adictorus$ denotes the base-change to $\Spa(\laurent, \pwr)$ of the $\pwr$-scheme $\BG_m^g$, in the sense of \cite{hubergen}.
\item $E_{w, \Tate}$ denotes the Tate curve over $\laurent$
\item $\Gt$ is a rank-$g$ split torus over $\pwr$ with character lattice $X \simeq \BZ^g$. 
\item $Y \simeq \BZ^g$ is an abstract lattice of periods admitting an inclusion $i: Y \to \Gt_{\laurent}$ given by $(y_1, \ldots, y_g) \mapsto (q^{y_1}, \ldots, q^{y_g})$.
\item $\widetilde{P}$ is a $\pwr$-scheme defining a relatively complete model for $(\Gt, i)$ as in \cite{FC, lan}. In later sections, it will be constructed as a toric scheme and it will depend on a choice of fan $\Sigma$. So it will be denoted by $\widetilde{P}_{\Sigma}$ or with other subscripts if $\Sigma$ depends on something else.
\item $\widetilde{\LL}$ and $\widetilde{\mathscr{N}}$ denote line bundles on $\widetilde{P}$. We always require that $\widetilde{\mathscr{N}}$ be ample in the sense of \cite{mumford_compositio}.
\item $P_{g, w}$ is the projective $\pwr$-scheme algebrizing $\widetilde{P}^\wedge/Y$, which may also depend on $\widetilde{P}$. Its $\laurent$-fiber is precisely $E_{w, \Tate}^g$. 
\item $B$ is a symmetric $g \times g$ $\BZ$-entried matrix, and $Q(\underline{x}) = \frac{1}{2}\underline{x}^\top B \underline{x} $ is the corresponding quadratic form on $\BZ^g$. The notation $Q'$ denotes the integer-valued quadratic function $Q(x) - \frac{1}{2}b^\top x$, where $b := (B_{11}, \ldots, B_{gg})^\top$. 
\item $\underline{a} = (a_1, \ldots, a_g) \in (\laurent^\times)^g$ is a tuple of Laurent series. 
\item $\LL^{Q, \underline{a}}$ is a line bundle on $E_{w, \Tate}^g$ that depends on the choice of $B$ and $\underline{a}$.
\item $\mathcal{O}(\adictorus)$ is the ring of very convergent power series in $Z_1^{\pm 1}, \ldots, Z_g^{\pm 1}$ over $\laurent$ as defined in \Cref{veryconvergent}. 
\end{itemize}
\section{Recap of the analytic theory}\label{ana}

A general polarized abelian variety over $\BC$ of dimension $g$ is of the form $V_\Omega = \BC^g/(\BZ^g + \Omega \BZ^g) \cong (\BC^\times)^g/Y$, where $\Omega$ is a symmetric $g \times g$ matrix with positive-definite imaginary part, and $Y$ is the relevant lattice of periods in $(\BC^\times)^g$, which is generated by the $g$ periods \[(\exp(2 \pi i \Omega_{1, i}), \ldots, \exp(2 \pi i \Omega_{g, i})), \qquad i = 1, \ldots, g.\] The Appell--Humbert theorem classifies the line bundles $\mathscr{L}$ on $V_\Omega$ in terms of certain multilinear algebra data, where the strategy of the proof (at least in \cite[Ch. 1]{mumfordAV}) is to pull an arbitrary line bundle $\mathscr{L}$ on $V$ to the universal cover $\BC^g$, where it must be trivial because $H^1(\BC^g, \mathcal{O}^\times) = 1$, and then classify the possible $(\BZ^g + \Omega \BZ^g)$-equivariant structures on the holomorphic structure sheaf on $\BC^g$. There is in general an obvious degenerating family of polarization-inducing line bundles $\mathscr{L}$ given on $V_\Omega$ (as the period matrix $\Omega$ itself degenerates in some direction) via the Appell--Humbert theorem by the hermitian matrix $(\mathrm{Im} \Omega)^{-1}$. 

Let $w, g$ be positive integers. This article only considers the arithmetic analog of the degenerating family of dimension-$g$ abelian varieties given by $E_q^g := (\BC^\times/q^\BZ)^g = (\BC^\times)^g/(q^{\BZ})^g$, parametrized by $q \to 0$ in a small punctured disc $\Delta^*$ around zero. Away from $q = 0$, $E_q^g$ is the $g$-th power of the elliptic curve $\BC^\times/q^{\BZ}$, and it degenerates to the torus $(\BC^\times)^g$ at $q = 0$. In this complex-analytic setting, our basic object of study is a family of line bundles on the degenerating family $(\BC^\times)^g/(q^\BZ)^g$, $q \in \Delta^*$, with $q \to 0$; though we also permit studying families defined over the finite covers of $\Delta^*$ parametrized by $q^{1/w}$, which we abusively also call $\Delta^*$. In this section we will perform the steps (I) and (II) of \Cref{intro} in this analytic context, thus providing the full generality version of \eqref{analytic-ample} and \eqref{analytic-nonample}. None of this is new, but its presence here is intended to set the stage for the arithmetic version later on. 

As a consequence of the Appell--Humbert theorem \cite[Ch. 1]{mumfordAV}, for generic values of $q$, the line bundles on $E_q^g = (\BC^\times)^g/(q^\BZ)^g = \BC^g/(\BZ^g + \mathrm{diag}(\tau, \ldots, \tau)\BZ^g)$ are classified up to $\Pic^0$-twist by quadratic forms $Q \in \Sym^2(\BZ^g)^\vee = \mathrm{NS}(E_q^g)$. We take the convention that an element of $\Sym^2(\BZ^g)^\vee$ is represented equivalently by a symmetric $g \times g$ $\BZ$-entried matrix $B = (b_{ij})$, a symmetric pairing $\langle \underline{x}, \underline{y}\rangle_B = \sum_{i, j} b_{ij}x_iy_j$ on $\BZ^g$, and a quadratic form $Q(\underline{x}) = \langle \underline{x}, \underline{x} \rangle_B/2 = \frac{1}{2}\sum_{i, j} b_{ij}x_ix_j = \frac{1}{2}\sum_{i = 1}^g b_{ii} x_i^2 + \sum_{i < j} b_{ij} x_ix_j$ valued in $\frac{1}{2}\BZ$. 

We will modify $Q$ by a linear form so that it is guaranteed to be integer-valued regardless of $B$. Namely, we will take the quadratic function $Q': \BZ^g \to \BZ$ given by
\[Q'(\underline{y}) := Q(\underline{y}) - \frac{1}{2}\sum_{i = 1}^g b_{ii}y_i.\]

The specific form that the Appell--Humbert classification takes in this case is that for generic values of $q$, the line bundles on $E_q^g$ with Chern class $Q \in \mathrm{NS}(E_q^g)$ are $\Pic^0$-twists of 
\[\mathscr{L}_{q}^Q := ((\BC^\times)^g \times \BC)/(q^\BZ)^g,\]
where the $(q^\BZ)^g$-action on $(\BC^\times)^g \times \BC$ depends on $Q$ and is given by
\[(q^{y_1}, \ldots, q^{y_g}) \cdot (\underline{Z}, t ) := ((q^{y_1}Z_1, \ldots, q^{y_g}Z_g), q^{-Q'(\underline{y})}\underline{Z}^{-B(\underline{y})}t).\]
Here the notation $\underline{Z}$ and $\underline{y}$ stand for $(Z_1, \ldots, Z_g) \in (\BC^\times)^g$ and $(y_1, \ldots, y_g) \in \BZ^g$, respectively. We now impose the additional hypothesis that \emph{the line bundle pulls back to the structure sheaf} of $(\BC^\times)^g$, since this assumption will always hold in the arithmetic context and therefore allow us to produce formulas in the analytic setting that are most relevant. The analysis of this section could be done without this by pulling back to the universal cover $\BC^g$, as is done in \cite{zharkov}. 

Under the triviality restriction on the pullback, the Appell--Humbert theorem also provides the explicit detail of how to twist by a $\Pic^0$ bundle that is trivial upon pullback to $(\BC^\times)^g$, which corresponds to a choice of $\Hom((q^\BZ)^g, \BC^\times)$, that is, a collection $a_1, \ldots, a_g \in \BC^\times$. The resulting twisted line bundle is given by
\[\mathscr{L}^{Q, \underline{a}}_q := ((\BC^\times)^g \times \BC)/(q^\BZ)^g,\]
where the action is by
\begin{equation}\label{action-general}q^{\underline{y}} \cdot (\underline{Z}, t) := (q^{\underline{y}}\underline{Z}, a_1^{-y_1} \cdots a_g^{-y_g} q^{-Q'(y)}\underline{Z}^{-B(y)}t).\end{equation}
Note that since we used $Q'$ instead of $Q$, this line bundle gets to vary over the family $E_q^g$ where $q$ itself parametrizes the punctured disc $\Delta^*$, if $\ua$ varies over that disc. Or more generally, it varies over the family $E_q^g$ where $q^{1/w}$ parametrizes the disc, if $\ua$ varies over that disk. If all $a_i = 1$ then we just denote the line bundle by $\LL^Q_q$.

The space of global sections $H^0(E_{q}^g, \mathscr{L}_{q}^{Q, \underline{a}})$ is canonically isomorphic to the vector space of holomorphic functions $f \in \mathcal{O}((\BC^\times)^g)$ satisfying the quasiperiodicity condition
\[f(q^{y_1}Z_1, \ldots, q^{y_g}Z_g) = \langle \underline{a}, \underline{y} \rangle^{-1} q^{-Q'(\underline{y})}\underline{Z}^{-B(\underline{y})}f(\underline{Z}) \qquad \text{ for all } \underline{y} \in \BZ^g.\]
In other words, the $i = 0$ case of the Fourier expansion isomorphism (generalizing \eqref{analytic-ample} to a general line bundle on $E_q^g$ trivial on $(\BC^\times)^g$) is established:
\begin{equation}\label{fourierexp-definite} H^0(E_q^g, \mathscr{L}_{q}^{Q, \ua}) \cong H^0((q^\BZ)^g, \mathcal{O}((\BC^\times)^g)),\end{equation}
where the lattice of periods $(q^\BZ)^g$ acts on $\mathcal{O}((\BC^\times)^g)$ via 
\begin{equation}\label{period-action}((q^{y_1}, \ldots, q^{y_g}) \cdot f)(\underline{Z}) := \langle \underline{a}, \underline{y} \rangle q^{Q'(\underline{y})}\underline{Z}^{B(\underline{y})}f(q^{y_1}Z_1, \ldots, q^{y_g}Z_g).\end{equation}
Note that we have established \eqref{fourierexp-definite} regardless of whether $\mathscr{L}_q^{Q, \underline{a}}$ is ample. If $Q$ is positive-definite, then $\mathscr{L}^{Q, \underline{a}}_q$ is ample and induces a polarization $E_q^g \to (E_q^g)^\vee$. An explicit basis of size $\disc := |\det B|$ for $H^0(E_q^g, \mathscr{L}^Q_{q})$ is classically constructed using the \emph{theta functions with characteristic} associated to the quadratic form $Q$ --- see for instance \cite[\S 6]{tataiii}. For example
\[\vartheta(\underline{Z}) := \sum_{\underline{\lambda} \in \BZ^g} q^{Q'(\underline{\lambda})}\underline{Z}^{B(\underline{\lambda})} \in H^0(E_q^g, \mathscr{L}^Q)\] 
is a Fourier expansion in the $\underline{Z}$-variables with coefficients in $\BC(\!( q )\!)$, providing a formula for a valid \emph{family} of elements of $H^0(E_q^g, \mathscr{L}^{Q}_{q})$ for $q \in \Delta^*$. An important arithmetic application of this function is that (at least if $Q'$ is replaced with $Q$) if it is evaluated along the identity section $\underline{Z} = (1, \ldots, 1)$, one obtains a modular form of weight $g/2$ whose $q$-expansion is a generating function for the number of ways to represent integers by $Q'$ --- that is the idea behind the proof of modularity of the integral Kudla--Millson series in \cite{KV}.

The analytic analog of \Cref{periods-recover} --- that the period lattice underlying $E_q^{g}$ is $(q^\BZ)^g$ can be recovered from the abstract abelian variety $E_q^g$ --- is much more classical, as the period lattice can be viewed as the non-real part of the kernel of the exponential map from the universal cover $\BC^g$. But a key observation of Faltings--Chai that underscores the importance of their Fourier expansion principle in the ample case is that the periods can be recovered directly from the coefficients of $\vartheta$. The mechanism for this is as follows. The line bundle $\mathscr{L}_q^Q$ induces the polarization $E_q^g \to (E_q^g)^\vee$ which is given on the uniformizing tori by the map $(\BC^\times)^g \to (\BC^\times)^g$ corresponding to the map $B: Y \to X$, where $Y \simeq \BZ^g$ is the character lattice of the torus uniformizing $(E_q^g)^\vee$ and $X \simeq \BZ^g$ is the character lattice of the torus uniformizing $E_q^g$ and both have been identified with $\BZ^g$ using the standard coordinates on $(\BC^\times)^g$. More concretely (but equivalently), the polarization is uniformized by the map of tori $(\BC^\times)^g \to (\BC^\times)^g$ given by
\[(Z_1, \ldots, Z_g) \mapsto (Z_1^{b_{1i}} \cdots Z_g^{b_{gi}})_{1 \leq i \leq g}.\]
The exact lattice of periods $(q^{\BZ})^g$ that one quotients by to get $E_q^g$ and $(E_q^g)^\vee$ can be obtained \emph{just from looking at the coefficients} of $\vartheta$ and the other nonzero element of $H^0(E_q^g, \mathscr{L}^Q_q)$, in the following way. For the coefficient in front of the character $\underline{Z}^{B(\underline{\lambda})}$ (where now $\underline{\lambda}$ can be viewed as an element of $Y$ so that $B(\underline{\lambda}) \in X$, and $\underline{Z}^{B(\underline{\lambda})}$ is just notation for it being viewed as a character of the torus with character lattice $X$) in the definition of the series defining $\vartheta$ being denoted by
\[\sigma_{\chi}(\vartheta) := q^{Q'(\underline{\lambda})}, \qquad \chi := B(\underline{\lambda})\]
we have
\begin{align}\sigma_{\chi + B(\underline{y})}(\vartheta) &= q^{Q'(\underline{\lambda} + \underline{y})}\\
  &= \sigma_{\chi}(\vartheta) q^{Q'(\underline{y})} q^{\langle \underline{\lambda}, \underline{y} \rangle_B}\\
  &= \sigma_{\chi}(\vartheta) q^{Q'(\underline{y})} q^{\chi_1 y_1 + \cdots \chi_g y_g}, \label{fourierpairing}
\end{align}
from which one extracts the pairing $Y \times X \to \BC^\times$ given by $(y, \chi) \mapsto q^{\chi_1 y_1 + \cdots \chi_g y_g}$ thus defining the period lattice $(q^{\BZ})^g \subset (\BC^\times)^g$. This pairing becomes fully-defined once also taking into account the other thetas with characteristic.

The definition of $\mathscr{L}^{Q, \underline{a}}_{q}$ and the Fourier expansion principle \Cref{fourierexp-definite} are perfectly valid if $Q$ is not positive-definite --- it's just that $\mathscr{L}_{q}^{Q, a}$ will not necessarily have any nonzero global sections, as one can readily deduce from \Cref{fourierexp-definite} itself. In fact, if $Q$ is nondegenerate and of signature $(r, s)$, then $\mathscr{L}^{Q, \underline{a}}_{q}$ still induces an isogeny\footnote{which does not get to be called a ``polarization'' if $s \geq 1$, as it does \emph{not} come from an ample line bundle.} $E_q^g \to (E_q^g)^\vee$, and it has nonvanishing cohomology only in degree $s$ \cite[\S 16]{mumfordAV}. In general, the Grothendieck spectral sequence and Cartan's Theorem B, which yields the vanishing of the higher cohomology of the pullback of $\mathscr{L}^{Q, a}_{q}$ to $(\BC^\times)^g$, show that (\ref{fourierexp-definite}) generalizes to the Fourier expansion isomorphism for higher cohomology
\begin{equation}\label{fourierexp-indefinite}
H^i(E_q^g, \mathscr{L}_{q}^{Q, a}) \cong H^i((q^\BZ)^g, \mathcal{O}((\BC^\times)^g))
\end{equation}
for all $i \geq 0$, $q \in \Delta^*$, where the action of the periods is as already defined in \Cref{period-action}. Explicit bases for the right hand side of (\ref{fourierexp-indefinite}) were produced by Zharkov \cite{zharkov}, who also used the Poisson summation formula to prove modularity properties for them in global families. In \cite{KV}, we prove an arithmetic variant of that with Poisson summation replaced with an algebraic version of the heat equation, using the arithmetic variant of \Cref{fourierexp-indefinite} proved in this article.

\section{Kuga--Sato degenerations over $\pwr$, and statements of main results}\label{intro2}

Fix positive integers $w$ and $g$, a quadratic form $Q = \frac{1}{2}\sum_{i, j} b_{ij}x_ix_j$ on $\BZ^g$, and a collection $a_1, \ldots, a_g \in \laurent^\times$. We also fix the notation
\[Q'(x) := Q(x) - \frac{1}{2}\sum_i b_{ii} x_i\]
which is a quadratic function $\BZ^g \to \BZ$ differing from $Q$ by a linear functional. 

This article is about the computation of $H^i(G_{\laurent}, \mathscr{L}_{\laurent}^{Q, \ua})$, where $G$ is a $\pwr$-group scheme playing the role analogous to the analytic family $E_q^g$ of \Cref{intro} and \Cref{ana}, and $\mathscr{L}^{Q, \ua}$ is a line bundle on $G$ rigidified along the identity section, analogous to the family $\mathscr{L}^{Q, \ua}_q$ of \Cref{ana}. 

In this section we give a suitable definition of $G$ and $\mathscr{L}^{Q, a}$, and we state the main results of this article in full detail. This fills in the gap in \Cref{intro} where it was not explained what the relationship between the data $Q$ and $\ua$ and the line bundle $\LL$ were in the statement \Cref{shortresult} of our main result. 

\subsection{Definition of the Kuga--Sato degenerations}
The scheme $\Spec(\laurent)$ replaces the un-punctured disk $\Delta^*$, as $\pwr$ is the defining ring of a formal neighborhood of a width-$w$ cusp of a modular curve, at least before possible base-change to $\BZ[1/N]$ or $\BZ[1/N, \zeta_N]$. Let $\Tate_w$ be the Tate curve over $\pwr$ --- the one that restricts over $q^{1/w} = 0$ to a N\'eron $w$-gon over $\BZ$ \cite[\S VII, D\'efinition 1.16]{DR}. By taking the connected component of the smooth locus of $\Tate_w$, one obtains $G_w$, which agrees with $\Tate_w$ over $\laurent$ and degenerates to $\BG_{m, \BZ}$ over $q^{1/w} = 0$. We will consider the degeneration
\begin{equation}\label{KS}G := G_w^g\end{equation}
over $\pwr$. It plays the role of $E_q^g$ from \Cref{intro}, since over $\laurent$ viewed as a punctured formal neighborhood of a width-$w$ cusp of a modular curve, it is the fibered $g$-th power of the universal elliptic curve; and over $q^{1/w} = 0$ it degenerates to a rank-$g$ split torus $\BG_{m/\BZ}^g$.

\begin{definition}
The \emph{Kuga--Sato degeneration} of width $w$ and rank $g$ is the $\pwr$-scheme $G$ defined by \Cref{KS}. 
\end{definition}

The fibered power of the Tate curve $E_{w, \Tate}^g$ is nothing but the base-change of $G$ from $\pwr$ to $\laurent$.

\subsection{Definition of the line bundles}
Let $Q = \frac{1}{2}\sum b_{ij}x_ix_j$ be a quadratic form on $\BZ^g$, with $B = (b_{ij})$ a symmetric $\BZ$-entried matrix. The idea of the construction of the relevant line bundle on $G$ corresponding to $Q$ is that one ``twists by $Q$'' a basic symmetric ample line bundle $\LL$ on $G_w$. In particular, $\LL$ should be chosen so that it provides a principal polarization of $E_{w, \Tate}$. The twisting procedure is the same one relevant to \cite[\S 6]{tataiii}, and is also explained in \cite[\S 3]{KV}. We now explain both of these points. 

Define the symmetric rigidified line bundle $\mathscr{L}_{\laurent} = \mathcal{O}_{E_{w, \Tate}}([-1])$. Here the relative Cartier divisor\footnote{We refer to \cite{KM} for the notion of relative Cartier divisor. Later on this will be reconciled with the viewpoint from Mumford/Faltings--Chai.} $[-1]$ in $E_{w, \Tate}$ is defined to be the image of the copy ``$-1$'' of $\Spec(\laurent)$ in the group scheme $\mu_2$ under the inclusion 
\[\mu_2 \to E_{w, \Tate}[2]\]
from \cite[Ch. VII, (1.16.3)]{DR}.
The line bundle $\mathscr{L}_{\BZ((q^{1/w}))}$ on $E_{w, \Tate}$ is the analog of the line bundle on $\BC^\times/q^{\BZ}$ given previously in \Cref{ana} by
\[\mathscr{L}_q^{x^2} := (\BC^\times \times \BC)/q^{\BZ},\]
where we recall that the action is by
\[q^y \cdot (Z, t) := (q^yZ, q^{-\frac{y^2 - y}{2}}Z^{-y}t),\]
since $H^0(\BC^\times/q^\BZ, \mathscr{L}_{q})$ is spanned by the single global section 
\[\sum_{n \in \BZ} q^{\frac{n^2 - n}{2}}Z^n,\] 
whose only zero in $\BC^\times/q^{\BZ}$ is $Z = -1$ (hence why the family $\mathscr{L}_{q}^{x^2}$ over $\Delta^*$ is the complex analytic analog of $\LL_{\laurent}$; we had to choose $\mathscr{L}_{q}^{x^2}$ to be a $\Pic^0$-twist of the bundle $\mathscr{L}_{q}$ from \Cref{intro} because we needed only integral powers of $q$ to appear, in case $w$ is odd --- if $w$ is even then we could choose $\mathscr{L} = \mathcal{O}([-q^{1/2}])$ which would be analogous to the bundle $\mathcal{L}_{q}^{x^2, q^{1/2}} = \LL_q$ which has global section $\sum_n q^{n^2/2}Z^n$).  

Then $\mathscr{L}^Q_{\BZ((q^{1/w}))}$ can be defined following \cite[\S 3]{KV} by
\begin{align}\label{twistQ}\mathscr{L}^Q_{\laurent} :&= \mbox{``$Q$-twist of $\mathscr{L}_{\laurent}$''}\\
  &= \left(\bigotimes_{i = 1}^g \pi_i^*\mathscr{L}^{b_{ii}}\right) \otimes \left(\bigotimes_{1 \leq i < j \leq g} (m_{ij}^*\mathscr{L} \otimes \pi_i^* \mathscr{L}^{-1} \otimes \pi_j^*\mathscr{L}^{-1})^{b_{ij}}\right),\end{align} 
where:
\begin{enumerate}
\item $\pi_i: E_{w, \Tate}^g \to E_{w, \Tate}$ is the $i$-th natural projection,
\item $m_{ij}: E_{w, \Tate}^g \to E_{w, \Tate}$ is the multiplication map of the $i$-th and $j$-th projections.
\end{enumerate}
The additional data $\underline{a}$ defines a specific $\Pic^0$-twist of $\LL^Q_{\laurent}$. By \cite[Ch. 1]{KM}, an element of $\Pic^0(G_{\laurent})$ is always of the form $\otimes_{i = 1}^g \pi_i^*\mathcal{O}([p_i] - [-1])$, where $p_i: \Spec \laurent \to E_{w, \Tate}$ is a point of $E_{w, \Tate}$. By algebrizing the sections of \cite[Ch. VII, \S 1.2, (d)]{DR}, we obtain such a point $p_i$ from the choice of $a_i(q) \in \laurent^\times$. Thus, we may define a more general rigidified line bundle on $G_{\laurent}$, namely
\begin{equation}
\LL_{\laurent}^{Q, a} := \LL^Q_{\laurent} \otimes \bigotimes_{i = 1}^g \pi_i^*\mathcal{O}([-a_i(q)^{-1}] - [-1]),
\end{equation}
Note that if $a_i \in \{\pm q^{1/2}, \pm 1\}$ (the first two being possible only if $w$ is even), then the points $\Spec \laurent \to E_{w, \Tate}$ they define are exactly the ones from \cite[Ch. VII, (1.16.3), (1.16.4)]{DR}, and the $\Pic^0$-twists being used here are precisely the ones by $\Pic^0(G)[2]$, making the resulting bundles $\mathscr{L}^{Q, a}$ symmetric.

It isn't relevant to actually stating our theorems, but $\LL_{\laurent}^{Q, \underline{a}}$ extends uniquely to a rigidified line bundle on all of $G$, by \cite[Ch. II, Th\'eor\`eme 3.2.3]{MB}. In the next section we will prove that these account for all rigidified line bundles on $G$ (equivalently, on $G_{\laurent}$, by \cite[Ch. II, Th\'eor\`eme 3.2.3]{MB}). 

\subsection{Full statement of the non-equivariant and equivariant Fourier expansion principles}
Having now explained precisely how the line bundle $\LL^{Q, \underline{a}}$ depends on $Q$ and $\ua$, we note that the statement we already gave for the Fourier expansion principle on $G_{\laurent}$ in \Cref{shortresult} is now fully well-defined. Recalling the definition of $\mathcal{O}(\adictorus)$ from \Cref{veryconvergent}, and noting that its definition itself does not require any theory of adic spaces, we restate \Cref{shortresult} in full:

\begin{theorem}\label{mainthm}
  Let $i \geq 0$, $w \geq 1$ and $g \geq 1$ be integers, and let $Q = \frac{1}{2}\sum_{i, j} b_{ij} x_ix_j$ be a quadratic form on $\BZ^g$. Let $\ua = a_1, \ldots, a_g$ be a tuple of elements of $\laurent^\times$. Then there is a canonical isomorphism
  \[H^i(E_{w, \Tate}^g, \mathscr{L}^{Q, \ua}_{\laurent}) \cong H^i(\BZ^g, \mathcal{O}(\adictorus)),\]
  where the action of $\BZ^g$ on $\mathcal{O}(\adictorus)$ is as defined in \Cref{finalaction}.
\end{theorem}

In the case where $Q$ is positive-definite, Faltings--Chai \cite[Theorem 6.2]{FC} (see also \cite[Corollary 4.5.4.24]{lan}) proved the analog of (\ref{fourierexp-definite}), that is, the special case of \Cref{mainthm} where $i = 0$ and $Q$ is positive-definite. As discussed in \Cref{periods-recover}, their proof used the method of Raynaud extensions and thus \emph{only} applies to global sections of ample line bundles, because it relies on passing a global section to the formal completion and then studying the properties of its Fourier coefficients as a section of the structure sheaf of a torus (this will be described in more detail in \Cref{degens}).

As promised in \Cref{intro}, the second main theorem of this article is an equivariant version of \Cref{mainthm}. Let $\Gamma$ be a subgroup of the orthogonal group of $(\BZ^g, Q)$. Then $E_{w, \Tate}^g$ admits an action of $\Gamma$ for which $\mathscr{L}^Q$ admits an obvious $\Gamma$-equivariant structure. Therefore, in addition to the coherent cohomology $H^i(E_{w, \Tate}^g, \mathscr{L}^Q)$, one can also consider the $\Gamma$-\emph{equivariant} coherent cohomology $H^i(E_{w, \Tate}^g, \mathscr{L}^Q)$; our second main theorem is a Fourier expansion principle for the $\Gamma$-equivariant cohomology of $\mathscr{L}^Q$ and is stated here:
\begin{theorem}\label{mainthm-equiv} 
  Let $i \geq 0$ and $g, w \geq 1$ be integers. The isomorphism of \Cref{mainthm} fits into a commutative diagram
\begin{center}
  \begin{tikzcd}
{H^i(E_{w, \Tate}^g, \mathscr{L}^Q_{\laurent})} \arrow[rr, "\sim"]                                   &  & {H^i((q^\BZ)^g, \mathcal{O}(\adictorus))}                                                           \\
{H^i_\Gamma(E_{w, \Tate}^g, \mathscr{L}^Q_{\laurent})} \arrow[u] \arrow[d, "e^*"] \arrow[rr, "\sim"] &  & {H^i(\Gamma \ltimes (q^\BZ)^g, \mathcal{O}(\adictorus))} \arrow[u] \arrow[d, "\mathrm{res_\Gamma}"] \\
{H^i_\Gamma(\Spec(\laurent), \mathcal{O})} \arrow[rr, "\sim"]                               &  & {H^i(\Gamma, \laurent)}                                                                      
\end{tikzcd}
\end{center}
where:
\begin{itemize}
\item The vertical upwards-pointing maps are the natural restriction maps.
\item The vertical downwards-pointing map on the left side is induced by pulling back along $e: \Spec(\laurent) \to E_{w, \Tate}^g$ and using the rigidification $e^*\mathscr{L}^Q \cong \mathcal{O}_{\laurent}$ that is part of the data of $\mathscr{L}$.
\item The vertical downwards-pointing map on the right side is induced by restricting to $\Gamma$ and restricting the coefficients in $\mathcal{O}(\adictorus)$ along $\underline{Z} = (1, \ldots, 1)$ to obtain elements of $\laurent$.  
\end{itemize}
The statement is also valid as written with $\LL \otimes Q$ replaced with $\LL^{x^2, a} \otimes Q$ for any $a \in \{\pm 1, \pm q^{1/2}\}$ (the latter two possibilities allowed so long as $w$ is even).
\end{theorem}

\section{The Faltings--Chai theory of degenerations}\label{degens}
In this section, we recall the general framework that the Kuga--Sato degenerations fit into, especially the classification of more general degenerations established by Faltings--Chai \cite[Ch. II-III]{FC}. Although this is already present in the literature, there are four reasons for doing this here:

\begin{enumerate}
\item[(1)] The Faltings--Chai classification is precisely the fact that Mumford's relatively complete models can be used to uniformize degenerations and line bundles on them. Therefore, we will make explicit use of it in order to classify the line bundles on Kuga--Sato degenerations, and to actually prove our main results \Cref{mainthm} and \Cref{mainthm-equiv}.
\item[(2)] The \emph{proof} of the Faltings--Chai classification is in and of itself an illustration of the arithmetic importance of the Fourier expansion principle in the ample case. Indeed, the key idea of \cite[Ch. II]{FC} is that one can recover the lattice of periods that determines a degeneration from the Fourier expansions of global sections of an ample line bundle. The analytic version of this was already discussed in \Cref{ana}.   
\item[(3)] The proof of the Faltings--Chai classification in \cite{FC} is missing a few explicit statements when it comes to the uniformization of arbitrary line bundles, which are worth making precise.
\item[(4)] The statement of the classification in the case of $E_{w, \Tate}^g$ requires significantly less notation than the most general situation, so we hope that recalling it here will be a convenience for the reader. We also hope that this concrete explication of Faltings--Chai in the special case of powers of the Tate curve will be helpful for those first learning the subject. 
\end{enumerate}

The Kuga--Sato degenerations are a special case of the more general concept of \emph{degenerations of abelian varieties}, which are the algebraic model for degenerating families of abelian varieties that suffice for the construction by Faltings--Chai \cite{FC} and Lan \cite{lan} of toroidal compactifications of PEL Shimura varieties. 

In the general Faltings--Chai--Lan theory of arithmetic toroidal compactifications of PEL Shimura varieties, the universal families of polarized abelian varieties over the Shimura variety naturally extend to families of \emph{semiabelian varieties} (extensions of abelian varieties by algebraic tori) over the choice of compactification. That this is the natural way for quite general families of abelian varieties to degenerate at the boundary of their moduli space was \emph{a priori} known due to Grothendieck's semistable reduction theorem \cite[Th\'eor\`eme 3.6]{sga71} decades before any arithmetic compactifications of the global moduli were constructed.

It is also clearly visible in the analytic theory of moduli of polarized abelian varieties: a general family of abelian varieties over $\BC$ is given by \[V_\Omega := \BC^g/(\BZ^g + \Omega \BZ^g) \underset{\sim}{\overset{\exp(2 \pi i \cdot)}{\longrightarrow}} (\BC^\times)^g/\exp(2 \pi i \Omega \BZ^g)\] for a family of $g \times g$ period matrices $\Omega$ with positive-definite imaginary part. The abelian variety $V_\Omega$ naturally comes with a polarization coming from the Appell--Humbert theorem \cite[Ch. 1]{mumfordAV} via the Hermitian matrix $(\mathrm{Im} \Omega)^{-1}$. Such a family degenerates to a semiabelian family as $\Omega$ degenerates in different directions. The exact directions in which $\Omega$ can degenerate is part of the data that determines the choice of toroidal compactification --- see \cite{AMRT, namikawa1, namikawa2} for the full detail of how those degenerations are constructed. The analytic Kuga--Sato degenerations $E_q^g$ from \Cref{ana} are easy to understand from this point of view: they come from the choice $\Omega = \mathrm{diag}(\tau, \ldots, \tau)$, where the parameter $\tau$ degenerates to $i\infty$, and $q := \exp(2 \pi i \tau)$. In that case, the family degenerates to a full torus at $q = 0$, since the entire period lattice $(q^\BZ)^g$ degenerates to zero as $q \to 0$.

Over formal neighborhoods of the boundary components of the arithmetic toroidal compactifications, the aforementioned semiabelian universal families become \emph{degenerations} of polarized abelian varieties. The degenerations used by \cite{lan} can be over any affine base
\[S := \Spec R,\]
where $R$ is a normal Noetherian domain complete with respect to an ideal $I \subset R$. Since they apply to the only ring that we work over, $R = \pwr$, so long as we are working with a general ring $R$ (only in this section), we also add the hypotheses:
\begin{itemize}
\item $R$ is regular.
\item $R/I$ is a normal domain. 
\end{itemize}
These are not essential but allow for parts of the theory to be written down using fewer words and most importantly fewer steps involving gluing in the \'etale topology of $\Spf(R)$. For the reader carefully following along in \cite{lan}, we at least note that the regularity hypothesis on $R$ allows us to know \emph{a priori} that all $R$-tori are isotrivial, and that all cubical line bundles on $\BG_m^g$ over $R$ are trivial. 

Let $K$ be the field of fractions of $R$. 
\begin{definition}
  A \emph{degeneration} over $R$ is a semi-abelian scheme\footnote{A \emph{semiabelian scheme} over $S$ is a smooth separated commutative group scheme $G \to S$ such that for all points $s \in S$, the fiber $G_s$ is a semiabelian variety over $\kappa(s)$.} $G \to S$, such that the generic fiber $G_K$ is an abelian variety over $K$ and such that the fiber $G_0$ over $S_0 := \Spec(R/I)$ is an extension of an abelian scheme $A_0 \to S_0$ by a torus $T_0 \to S_0$. The category $\DEG$ is the one whose objects are the degenerations over $R$ and whose morphisms are the $S$-group homomorphisms. 
\end{definition}

In the general case, in either the analytic or arithmetic setting the exact description of the degenerations near each boundary component of the relevant compactified moduli of polarized abelian varieties requires fairly heavy notation to describe, as they depend on the choice of data determining the toroidal compactification in order to specify the directions in which the periods degenerate, and they might not be totally degenerate (meaning $G_0$ could be a nontrivial extension of an abelian scheme by an algebraic $S_0$-torus instead of just an $S_0$-torus). In the arithmetic situation one also has to contend with the possibility that the torus part of $G_0$ is not split. This article therefore enjoys simpler notational setup, since we need only consider the split totally degenerate situation of $E_{w, \Tate}^g$ ---  the Kuga--Sato degeneration $G = G_w^g$ over $S = \Spec \pwr$ is of relative dimension $g$ and specializes to split rank-$g$ tori over $S_0 = \Spec \BZ$. We therefore further specialize to the following category, following \cite[Ch. III, \S 2]{FC}:

\begin{definition}The category $\DEG_{g, g}^{\spli}$ is defined to be the full subcategory of $\DEG$ whose objects are the degenerations $G \to S$ of relative dimension $g$ with the property that $G_0$ is a split rank-$g$ torus over $S_0$. 
\end{definition} 

Obviously for our purposes one needs to consider degenerations $G$ along with the data of a rigidified line bundle $\mathcal{L}$ on $G$. The most important ones to the theory of PEL Shimura varieties are of course the ones inducing polarizations of the abelian varieties $G_K$.
\begin{definition}\label{ampledef}
  The category $\DEG_{\amplegg}^{\spli}$ is the category whose objects are pairs $(G, \mathcal{L})$, where $G$ is an object of $\DEG_{g, g}^{\spli}$, and $\mathcal{L}$ is a line bundle on $G$ rigidified along the identity section, with the additional property\footnote{The lineage of this condition is as follows: as written in \Cref{ampledef}, it agrees exactly with what is in \cite[Ch. III, \S2]{FC}. Lan remarks that it seems that Faltings--Chai should really require $\mathcal{L}$ be ample on $G$ \cite[Definition 4.2.1.1]{lan}. The two coincide in the case when $\mathcal{L}$ is symmetric, by \cite[Ch. VI, Proposition 3.1]{MB}, but \emph{not} necessarily otherwise --- to be addressed in a short note. However, after recent communication with Lan, the consensus is that the definition of \cite{FC} is the correct one. With Lan's convention, \cite[Corollary 4.5.4.9]{lan}, which states that the ``uniformization'' functor $M_{IS}$ restricts to give a functor from $\DD_{ample}^{\spli}$ to $\DEG_{\amplegg}^{\spli}$, fails to hold. We refer to later in this section for the relevant definitions.} that $\mathcal{L}_K$ is ample on the abelian variety $G_K$; the morphisms are the $S$-group homomorphisms $f: G_1 \to G_2$ such that $\mathcal{L}_1 = f^*\mathcal{L}_2$.   
\end{definition}
\begin{remark}
  At least after suitable base change, a polarization on $G_K$ is always induced by a symmetric ample line bundle $\mathcal{L}_K$ on $G_K$, which always extends to an object of $\DEG_{\amplegg}^{\spli}$ by \cite[Ch. II, Th\'eor\`eme 3.2.3]{MB}. Conversely, every object $(G, \mathcal{L})$ of $\DEG_{\amplegg}^{\spli}$ restricts to an ample line bundle on $G_K$ which induces a polarization on it. 
\end{remark}
\begin{definition}
  The category $\DEG_{\ISgg}^{\spli}$ is the category whose objects are pairs $(G, \mathcal{L})$, where $G$ is an object of $\DEG_{g, g}^{\spli}$, and $\mathcal{L}$ is any line bundle on $G$ rigidified along the identity section.
  The morphisms are the $S$-group homomorphisms $f: G_1 \to G_2$ such that $\mathcal{L}_1 = f^*\mathcal{L}_2$. 
\end{definition}

The pairs $(G_w^g, \mathscr{L}^{Q, \ua})$ from \Cref{intro2} are objects of the category $\DEG_{\ISgg}^{\spli}$, and are in $\DEG_{\amplegg}^{\spli}$ if $Q$ is positive-definite. 

One of the novel contributions of \cite{FC} (with some corrections by \cite{lan} to the proofs) was a complete classification of all degenerations (including the line bundle data) as defined above in terms of so-called \emph{degeneration data}. In the case of the line bundles on Kuga--Sato degenerations, the degeneration data of Faltings--Chai are exactly the promised linear algebraic data in step (I) of \Cref{intro}. We will use that theory to prove that the bundles $\LL^{Q, \ua}$ exhaust all the objects of $\DEG^{\spli}_{\ISgg}$ extending the object $G$ of $\DEG^{\spli}_{g, g}$; more importantly, we will use the viewpoint of how $\LL^{Q, \ua}$ fits into that classification to prove our main theorems \Cref{mainthm} and \Cref{mainthm-equiv} as well as our analysis of the Heisenberg group action in \Cref{heis}. 

The idea of the Faltings--Chai classification of objects in $\DEG_{g, g}^{\spli}$ is that in analogy to the analytic case they may be written as a kind of quotient of $\BG_m^{g}$ by an embedded lattice of periods $\BZ^g \cong Y \hookrightarrow \BG_m^g \times_R K$. We will prove that, as expected from \Cref{ana}, the Kuga--Sato degenerations fit into this framework using the specific choice of embedding $Y = \BZ^g \hookrightarrow \BG_m^g \times_R K$ given by
\[\underline{y} = (y_1, \ldots, y_g) \mapsto (q^{y_1}, \ldots, q^{y_g}).\]
Of course, there is no such scheme-theoretic quotient, as the Zariski topology is too coarse. There are at least two closely-related ways\footnote{Recall from \Cref{intro}: they should be related by Huber's functor $t_a$ \cite{huberbook, regensburg}, but the quasicompactness assumption in Huber's work needs to be weakened to quasi-paracompactness.} to resolve this question. One is to consider the adic space $\BG_m^g$ over $\Spa(\laurent, \pwr)$ and take the quotient by $Y$ following \cite{hubergen}. We will use this point of view in part of our \emph{proof} of \Cref{mainthm}, but for the setup it is more convenient (if only due to having more theory available surrounding it in the literature) to follow the method of uniformization by relatively complete models, which is due to Mumford \cite{mumford_compositio} and used in the viewpoint of \cite{FC} and \cite{lan}. 

This section is further organized into two subsections following \cite[Ch. II-III]{FC}. In \Cref{mumford-ss}, we will summarize the content of \cite[Ch. III]{FC} and its counterparts in \cite[Ch. 4]{lan}, which explain Mumford's construction generalizing that of (line bundles on) the Tate curve from \Cref{qt} to construct objects in $\DEG^{\spli}_{g, g}$ (resp. $\DEG^{\spli}_{\ISgg}$) as quotients of rank-$g$ split $R$-tori (resp. by descending $Y$-equivariant data to those quotients). In \Cref{raynaud}, we summarize \cite[Ch. II]{FC} and its counterparts in \cite{lan}, which explains how to reverse the process, to prove that every object in $\DEG_{\ISgg}^{\spli}$ can be uniformized in the way described in \Cref{mumford-ss}.  

\subsection{Mumford's quotient construction}\label{mumford-ss}

The objects used to uniformize objects of $\DEG_{g, g}^{\spli}$ are the \emph{degeneration data}:
\begin{definition}\label{DDdef}
  The category $\DD_{g, g}^{\spli}$ of split \emph{degeneration data} of rank $g$ and torus rank $g$ over $R$ has for its objects the tuples $(\widetilde{G}, i: Y \hookrightarrow \widetilde{G}_K)$, where:
  \begin{enumerate}
  \item $\widetilde{G}$ is a split torus over $R$ with character lattice $X \cong \BZ^g$ the constant $R$-group scheme. A section $\chi \in X$ interpreted as a character of $\Gt$ by the notation $(\cdot)^{\chi}$ or $\underline{Z}^\chi$ depending on whether the context demands that we view it as an abstract character or evaluate it at a given point. 
  \item $Y \cong \BZ^g$ is also the constant $R$-group scheme, and $i: Y_K \hookrightarrow \widetilde{G}_{K}$ is an inclusion of group schemes.
  \item The pairing $\tau: Y \times X \to K^\times$ that is equivalent to the data of $i: Y \hookrightarrow \Gt_K$ via
    \[i(y)^{\chi} = \tau(y, \chi) \qquad \mbox{for all $(y, \chi) \in Y \times X$}\]
    is nondegenerate, and
    satisfies the positivity condition: there exists a map of abelian groups $\phi: Y \to X$ with the property that $\tau(-, \phi(-))$ is symmetric and
    \[\tau(y, \phi(y)) \in I \text{ for all } y \in Y - \{0\}.\] 
  \end{enumerate}
  The last condition is related to the fact that objects of $\DEG_{g, g}^{\spli}$ always extend to objects of $\DEG_{\amplegg}^{\spli}$ by \cite[XI 1.16]{raynaud}, so the same should be true in the category of data that uniformize those objects (we will see how the positivity condition is related to existence of ample line bundle once we discuss degeneration data for line bundles).
  
  The morphisms in $\DD_{g, g}^{\spli}$ are defined to be the morphisms of group schemes $\Gt_1 \to \Gt_2$ such that the diagram
  \begin{center}
\begin{tikzcd}
Y_1 \arrow[r, "i_1"] \arrow[d, "f"] & \widetilde{G}_1 \arrow[d, "f"] \\
Y_2 \arrow[r, "i_2"]                & \widetilde{G}_2               
\end{tikzcd}
  \end{center}
  commutes.
\end{definition}

The purpose of this section is to define the functor $M: \DD_{g, g}^{\spli} \to \DEG_{g, g}^{\spli}$ which is the ``quotient by the period lattice $Y$'', as well as its variant $M: \DD_{\ISgg}^{\spli} \to \DEG_{\ISgg}^{\spli}$ that also takes into account the ability to produce line bundles on the quotient by providing $Y$-equivariant line bundles on the uniformization. This is Mumford's method of uniformization by relatively complete models \cite{mumford_compositio} --- it works by producing a kind of compactification $\Pt$ of $\widetilde{G}$ whose quotient by $Y$ can be taken, after formal completion, exactly as in \Cref{qt}. 

\begin{definition}\label{rcm}
  Let $(\widetilde{G}, i: Y_K \hookrightarrow \widetilde{G}_K)$ be an object of $\DD_{g, g}^{\spli}$. A \emph{relatively complete model} for $(\widetilde{G}, i: Y_K \hookrightarrow \widetilde{G}_K)$ is the following data:
  \begin{enumerate}
  \item An integral $R$-scheme $\widetilde{P}$ locally of finite type over $R$ containing $\widetilde{G}$ as a Zariski-dense open subset.
  \item An ample\footnote{The definition of ample requires some care because $\widetilde{P}$ is not of finite type over $R$. For these purposes one always takes it to be that the nonvanishing loci of all global sections of all positive powers of $\widetilde{\mathscr{N}}$ generate the topology of $\widetilde{P}$.} line bundle $\widetilde{\mathscr{N}}$ on $\widetilde{P}$. We remark that $\widetilde{\mathscr{N}}|_{\widetilde{G}} \cong \mathcal{O}$ by \cite[Corollary 3.2.5.7]{lan}. 
  \item An action of $\widetilde{G}$ on $\widetilde{P}$ extending the action of $\widetilde{G}$ on itself by translation.
  \item An action of $Y$ on $\widetilde{P}$ extending the action on $\widetilde{G}_K$ via $i$ and the group law of $\Gt_K$. 
  \item A $Y$-equivariant structure on $\widetilde{\mathscr{N}}$. We remark that after restriction to the generic fiber, this is the same as a $Y$-equivariant structure on $\mathcal{O}_{\widetilde{G}_K} \cong \widetilde{\mathscr{N}}|_{\widetilde{G}_K}$, which is given by a cocycle $\sigma$ for $H^1(Y, \mathcal{O}(\widetilde{G}_K)^\times)$. 
  \item A $\widetilde{G}$-equivariant structure on $\widetilde{\mathscr{N}}$ extending the translation action on $\mathcal{O}_{\widetilde{G}} \cong \widetilde{\mathscr{N}}|_{\widetilde{G}}$.
  \end{enumerate}
  The above data are required to satisfy some conditions:
  \begin{enumerate}
  \item[(i)] There is an open $\widetilde{G}$-invariant subset $U \subset \widetilde{P}$ such that $U$ is of finite type over $R$ and
    \[\widetilde{P} = \bigcup_{y \in Y} y \cdot U.\]
  \item[(ii)] $\widetilde{P}$ satisfies the following variant of the valuative criterion for properness: For all valuations $v$ (possibly not of rank $1$) on the field of rational functions of $\widetilde{G}$ for which $v(R) \geq 0$, $v$ has a center on $\widetilde{P}$ if and only if
    \begin{idea}\label{vcrit} for all  $\chi \in X$,  there exists  $y \in Y$ such that $v(i(y)^\chi \underline{Z}^{\chi}) \geq 0$.\end{idea} 
  \item[(iii)] A positivity condition: the map of abelian groups $\phi: Y \to X$ such that
    \[\frac{\sigma(y_1 + y_2)}{\sigma(y_1)\sigma(y_2)} = \tau(y_1, \phi(y_2))\] 
    for all $y_1, y_2 \in Y$ satisfies the property that $\tau(y, \phi(y)) \in I$ for all $y \in Y - 0$. This is to make sure that $\Nt$ is an extension to $\Pt$ of a genuine object of $\DD_{\amplegg}^{\spli}$. 
  \end{enumerate}
\end{definition}
We give credit to Mumford \cite[Definition 2.1]{mumford_compositio} for the exact statement of the completeness condition, which is the specialization of \cite[Ch. III, Definition 3.1]{FC} to the totally degenerate situation that we have put ourselves in.
\begin{proposition}
Every object $(\widetilde{G}, i: Y_K \to \widetilde{G}_K)$ of $\DD_{g, g}^{\spli}$ admits a relatively complete model.
\end{proposition}
\begin{proof}
See \cite[Theorem 2.5]{mumford_compositio}, or \cite[Proposition 3.3 and remarks above Lemma 3.2]{FC}, or \cite[Proposition 4.5.1.15 and Corollary 4.5.1.8]{lan}. These references prove it in the most general situation for the same reasons, but we do not go into it because the $\Proj$ construction of relatively complete models that is used here will not be convenient for our ultimate goals. Instead, in the specific case of the Kuga--Sato degenerations, we will construct bespoke relatively complete models by means of infinite-type toric geometry, as suggested but not carried out by \cite{FC} and \cite{mumford_compositio}. 
\end{proof}
Mumford's construction provides a functor
\[M: \DD_{g, g}^{\spli} \to \DEG_{g, g}^{\spli}\]
given as follows (we already summarized this in the special case of the Tate curve, in \Cref{qt}). 
\begin{construction}\label{mumford-construction} 
  Let $(\widetilde{G}, i: Y_K \hookrightarrow \widetilde{G}_K)$ be an object in $\DD_{g, g}^{\spli}$. Let $(\widetilde{P}, \widetilde{\mathscr{N}})$ be a relatively complete model for it.
  \begin{enumerate}
  \item Let $\Pt^\wedge$ be the formal completion along the ideal $I$, which is formed from the system of $R/I^k$-schemes $\Pt_k := \Pt \times_R (R/I^k)$ over all $k \geq 1$.
  \item The action of $Y$ on $\widetilde{P}$ induces an action of $Y$ on each $\widetilde{P}_k$. Unlike $\widetilde{P}$, the quotient of $\widetilde{P}_k$ by $Y$ as an fpqc sheaf is a scheme $P_k$ such that the quotient map is an \'etale surjection $\widetilde{P}_k \to P_k$. The ample line bundle $\widetilde{\mathscr{N}}_k = \widetilde{\mathscr{N}} \times_R R/I^k$ descends through the quotient to a line bundle $\mathscr{N}_k$ on $P_k$ which is ample by the Nakai--Moishezon criterion. The most complete available explanation of this step may be found in \cite[Proposition 4.5.2.15]{lan}, and we note that it uses both conditions (i) and (ii) in the definition of a relatively complete model. 
  \item The formal scheme $\mathfrak{P} := (P_k)_{k \geq 1}$ is proper over $\Spf(R)$ and admits an ample line bundle $(\mathscr{N}_k)$, so by Grothendieck's algebrization theorem \cite[Th\'eor\`eme (5.4.5)]{EGAIII}, it is uniquely the formal $I$-adic completion of a proper $R$-scheme $X$.
  \item We define a closed subscheme $C \subset X$ via \cite[Corollaire (5.1.8)]{EGAIII} by defining $\widehat{C} \subset \widehat{X} = \mathfrak{P}$ over $R/I^k$ to restrict over $\Spec(R/I^k)$ as 
    \[C_k  := \left(\widetilde{P} - \bigcup_{y \in Y} y \cdot \widetilde{G}\right)/Y \subset P_k.\] 
    One then obtains an $R$-scheme $G := X - C$. 
  \end{enumerate}
\end{construction}

Mumford \cite{mumford_compositio} proved the following facts about $G$, for which again we cite the very detailed exposition of \cite{lan}.
\begin{enumerate}
\item $G$ is an object of $\DEG_{g, g}^{\spli}$ --- see \cite[Corollaries 4.5.3.8, 4.5.3.9, and 4.5.3.13]{lan} for the fact that it is an object of $\DEG$. The fact that it is split and totally degenerate of torus rank $g$ is a consequence of the canonical isomorphism $\widehat{G} \simeq \Gt^\wedge$. 
\item The construction of $G$ from $(\widetilde{G}, i: Y_K \hookrightarrow \Gt_K)$ does not depend on the choice of relatively complete model $(\Pt, \Nt)$ --- see \cite[Corollary 4.5.3.7]{lan}.
\item Now that (1) and (2) have established that $M$ provides a map from objects of $\DD_{g, g}^{\spli}$ to objects of $\DEG_{g, g}^{\spli}$, it is suitable to remark that both of them require as main input \cite[Theorem 4.5.3.6]{lan}, which also establishes the functoriality of $M: \DD_{g, g}^{\spli} \to \DEG_{g, g}^{\spli}$. Again we leave the details of the proofs to \cite{lan}, as they are not relevant to the directions pursued here. 
\end{enumerate}

We now describe the variant for degeneration data of line bundles.
\begin{definition}\label{DDIS}
  The category $\DD_{\ISgg}^{\spli}$ has as its objects the triples $(\widetilde{G}, i: Y \hookrightarrow \widetilde{G}_K, \widetilde{\mathscr{L}})$, where the first two entries form an object of $\DD_{g, g}^{\spli}$ and the last entry $\widetilde{\mathscr{L}}$ is a rigidified line bundle on $\widetilde{G}$. By \cite[Corollary 3.2.5.7]{lan}, there is only one choice for $\widetilde{\mathscr{L}}$, namely $\mathcal{O}_{\widetilde{G}}$ with the canonical rigidification. But in a degeneration datum, the symbol $\widetilde{\mathscr{L}}$ also contains the data of a $Y$-equivariant structure on $\widetilde{\mathscr{L}}_K$, for the action of $Y$ on $\widetilde{G}_K$ given by the inclusion $i$ and the group structure on $\widetilde{G}(K)$. The $Y$-equivariant structure is also required to satisfy a \emph{finiteness condition} which is given further down in \Cref{finiteness} once a bit more notation has been set up. 
\end{definition}
The set of isomorphism classes of equivariant structures as in \Cref{DDIS} is in explicit bijection with the group cohomology $H^1(Y, \mathcal{O}(\widetilde{G}_K)^\times)$, where $Y$ acts on $\mathcal{O}(\widetilde{G}_K)^\times$ via $y: f(\underline{Z}) \mapsto f(i(y) \cdot \underline{Z})$. Since every $\alpha \in \mathcal{O}(\widetilde{G}_K)^\times$ is of the form $r \cdot \underline{Z}^\chi$ for $\chi \in X$ and $r \in K$, such a cocycle must be of the form
  \[y \mapsto \alpha(y) := \zeta(y) \cdot \underline{Z}^{\phi(y)}\]
  where $\zeta: Y \to K^\times$ and $\phi: Y \to X$ are maps of sets. The cocycle condition on $\alpha$ works out to the additional hypotheses that 
  \[\frac{\zeta(y_1 + y_2)}{\zeta(y_1) \zeta(y_2)} = \tau(y_1, \phi(y_2)) \qquad \text{ and } \phi \text{ is a group homomorphism}.\]

  We retain the notation $\zeta$ for the $K^\times$-part of the cocycle  and $\phi$ for the $X$-part of the cocycle, and we consider those to be part of the degeneration datum (that is, part of the data of $\Lt$, which includes the $Y$-equivariant structure on its generic fiber). The \emph{finiteness condition} is related to the fact that degenerations always admit an ample line bundle $\mathscr{N}$ \cite[XI 1.1.16]{raynaud}, and a line bundle $\mathscr{L}$ on a degeneration always has the property that $\mathscr{L} \otimes \mathscr{N}^n$ is ample for $n \gg 0$ \cite[Corollaire (4.5.8)]{EGAII}. It is given by:  
  \begin{idea}\label{finiteness}there exists  $n \geq 1$ and  $\phi': Y \to X$  satisfying the positivity condition of \Cref{DDdef} such that  $\zeta(y) \tau(y, \phi'(y))^n \in R$ for all  $y \in Y$\end{idea}
  \begin{remark}
    The isomorphism class of $\widetilde{\mathscr{L}}$ as a $Y$-equivariant structure on $\mathcal{O}_{\widetilde{G}_K}$ is determined by $\alpha$ as a cocycle for $H^1(Y, \mathcal{O}(\widetilde{G}_K)^\times)$. The coboundaries defining that cohomology group are given by the $1$-cochains of the form
    \[y \mapsto (i(y)^* f) \cdot f^{-1}, \qquad f = r \underline{Z}^\chi \in K^\times X.\]
    This is the same as
    \[y \mapsto \tau(y, \chi).\]
    So the map $\phi: Y \to X$ is a uniquely determined property of an isomorphism class in $\DD_{\ISgg}$, and the quadratic function $\zeta: Y \to K^\times$  is uniquely determined up to multiplication by multiplicatively linear functions of the form $y \mapsto \tau(y, \chi)$ for $\chi \in X$. So the data  $(\zeta, \phi)$ that comes with an object of $\DD_{\ISgg}^{\spli}$ is well-defined up to multiplication of $\zeta$ by functions of the form $y \mapsto \tau(y, \chi)$. 
  \end{remark}

  As indicated above, the finiteness condition \Cref{finiteness} (by \cite[Lemma 4.4.14]{lan}) implies that $\Lt$ is a difference $\Lt = \Lt_1 \otimes \Lt_2^{-1}$, where $(\Gt, i, \Lt_1)$ and $(\Gt, i, \Lt_2)$ are objects of $\DD_{\amplegg}^{\spli}$, the category defined by
  \begin{definition}
    $\DD_{\amplegg}^{\spli}$ is the full subcategory of $\DD_{\ISgg}^{\spli}$ whose objects are the triples $(\Gt, i: Y_K \hookrightarrow \Gt_K, \Lt)$ forming objects of $\DD_{\ISgg}^{\spli}$ but also satisfying the \emph{positivity condition},
    \begin{idea}
      $\tau(y, \phi(y)) \in I$ for all $y \in Y - \{0\}$. 
    \end{idea}
  \end{definition}
  In particular, the positivity condition we required in the definition of the objects of $\DD_{g, g}^{\spli}$ (\Cref{DDdef}) amounts to requiring that the objects of $\DD_{g, g}^{\spli}$ all admit an extension to $\DD_{\amplegg}^{\spli}$, as suggested in the text after \Cref{DDdef}.  

  The functor $M:  \DD_{g, g}^{\spli} \to \DEG_{g, g}^{\spli}$ constructed in \Cref{mumford-construction} is extended by \cite[Ch. III]{FC} to functors $M_{ample}$ and $M_{IS}$ fitting into the diagram
  \begin{idea}\label{Mdiag} 
  \begin{center}
    \begin{tikzcd}
                                                      & {\DD_{g, g}^{\spli}} \arrow[r, "M"]                                      & {\DEG_{g, g}^{\spli}}                              &                                  \\
                                                      & \DD_{\amplegg}^{\spli} \arrow[u] \arrow[r, "M_{ample}"] \arrow[ld, hook] & \DEG_{\amplegg}^{\spli} \arrow[u] \arrow[rd, hook] &                                  \\
\DD_{\ISgg}^{\spli} \arrow[ruu] \arrow[rrr, "M_{IS}"] &                                                                          &                                                    & \DEG_{\ISgg}^{\spli} \arrow[luu]
\end{tikzcd}
  \end{center}
\end{idea}
  The definition of the functors $M_{ample}$ and $M_{IS}$ is a somewhat more subtle point than one might expect. We discuss this in detail, because our proof of \Cref{mainthm} and its relatives are most conveniently proved using a slightly modified version of the original construction --- we do not claim it is an original idea, as the fact that it can be done this way is already hinted at in \cite[p. 72]{FC}. In the remainder of this section, we recall the construction of $M_{IS}, M_{ample}$ \cite{FC, lan}; we will return to state our modified version in \Cref{modified}, by which point we will have introduced enough material from \cite{FC, lan} to be able to prove that our version agrees with the original functors $M_{IS}$, $M_{ample}$. 

\begin{definition}
A \emph{Mumford-ample relatively complete model} $(\widetilde{P}, \widetilde{\mathscr{N}})$ for an object $(\Gt, i, \Nt)$ of $\DD_{\ISgg}^{\spli}$ is a relatively complete model $(\widetilde{P}, \Nt)$ for $(\Gt, i: Y \hookrightarrow \Gt_K)$ with the property that the $Y$-equivariant structure on $\Nt$ extends the one that is part of the data of $(\Gt, i,  \Nt)$. Let $\DD_{\ISgg}^{\spli, \Mample}$ denote the full subcategory of $\DD_{\ISgg}^{\spli}$ whose objects are those admitting a Mumford-ample relatively complete model. \emph{A priori}, we already know that $\DD_{\ISgg}^{\spli, \Mample}$ is a full subcategory of $\DD_{\amplegg}^{\spli}$, because of the positivity condition in \Cref{rcm}(iii).
\end{definition}

\begin{construction}[Faltings--Chai definition of $M_{IS}$]\label{FCdef}
  Let $D = (\Gt, i: Y \hookrightarrow \Gt_K, \Lt)$ be an object of $\DD_{\ISgg}^{\spli, \Mample}$. Then $M(D)$ is the object $(G, \LL)$ of $\DEG_{\amplegg}$ whose underlying degeneration $G$ is defined by \Cref{mumford-construction} for any Mumford-ample relatively complete model of $D$. To define a line bundle $\LL$ on $G$ extending $G$ to an object of $\DEG_{\ISgg}^{\spli}$, note in step (3) of \Cref{mumford-construction} that by \cite[Th\'eor\`eme (5.4.5)]{EGAIII}, the line bundle $\Nt^\wedge$ on $\mathfrak{P}$ algebrizes to an ample line bundle on the proper $R$-scheme $X$ (in fact we used the existence of $\Nt^\wedge$ to produce the algebrization of $\mathfrak{P}$ to begin with, by \cite[Th\'eor\`eme (5.4.5)]{EGAIII}). By restriction to $G \subset X$, one then obtains an ample rigidified line bundle $\LL$ on $G$, thus defining the object $M_{IS}(D) := (G, \LL)$ of $\DEG_{\amplegg}^{\spli}$. 

  An arbitrary degeneration datum $D$ in $\DD_{\ISgg}^{\spli}$ can be written as a \emph{difference} of Mumford-ample degeneration data given by $\Lt_1, \Lt_2$ --- again, this is a consequence of the finiteness condition (see \cite[Construction 4.5.4.6]{lan}). Then one defines $M_{IS}: \DD_{\ISgg}^{\spli} \to \DEG_{\ISgg}^{\spli}$ to be
  \[M_{IS}(\Lt_1 \otimes \Lt_2^\vee) := M_{IS}(\Lt_1) \otimes M_{IS}(\Lt_2)^{-1}.\]
  It is known that this does not depend on the choice of $\LL_1$ and $\LL_2$ and defines a functor $M_{IS}: \DD_{\ISgg}^{\spli} \to \DEG_{\ISgg}^{\spli}$ \cite[Corollary 4.5.4.9]{lan}.  
\end{construction}

One point that is implicit in \cite{lan} but not explicitly stated in the proof of \cite[Corollary 4.5.4.9]{lan} is:
\begin{lemma}
The functor $M_{IS}$, when restricted to objects in $\DD_{\amplegg}^{\spli}$, outputs objects of $\DEG_{\amplegg}^{\spli}$.
\end{lemma}
\begin{proof}
  Let $D = (\Gt, i, \Lt)$ be an object of $\DD_{\amplegg}^{\spli}$. Now let $(G, \LL)= M_{IS}(D)$; our goal is to prove that $\LL_K$ is ample. This follows from \cite[Lemma 4.5.4.5]{lan}, which says that for all sufficiently large $n$,
  \[D^{n + 1} \otimes ([-1]^* D)^{n} \in \DD_{\amplegg}^{\spli, \Mample}.\]
  We deduce by the existence of a cubical structure on $\LL$ \cite[Ch. I, \S 2.6]{MB} that 
  \[[2n + 1]^* \LL \simeq \left(\LL^{n + 1} \otimes [-1]^*\LL^n\right)^{2n + 1} \]
  and thus that $[2n + 1]^* \LL$ is ample. So far, this does not require passing to the generic fiber. On the other hand, it is only by passing to $\LL_K$ does this imply that $\LL_K$ is ample, by \cite[Corollaire (6.6.3)]{EGAII} and \cite[Ch. V, Proposition 8.2]{CS}. 
\end{proof}

We have therefore established the existence of the promised compatible functors $M_{IS}$ and $M_{ample}$ (defined as the restriction of $M_{IS}$ to $\DD_{\amplegg}^{\spli}$), fitting into the diagram \Cref{Mdiag}.

\subsection{Uniformization and Fourier expansions}\label{raynaud}
In this section, we summarize from \cite{FC, lan} the construction of an inverse functor $F_{IS}: \DEG_{\ISgg}^{\spli} \to \DD_{\ISgg}^{\spli}$, which is constructed by extending linearly from $\DEG_{\amplegg}^{\spli}$. It is only by constructing this functor and proving that it is a quasi-inverse to $M_{IS}$ that one knows that the Mumford uniformization procedure actually uniformizes all objects of $\DEG_{\ISgg}^{\spli}$ --- so it is logically necessary in order to know that the line bundles defined in \Cref{intro2} account for all of $\DEG_{\ISgg}^{\spli}$ covering the Kuga--Sato degeneration of width $w$ and rank $g$. Both steps (the construction of $F_{IS}$ as well as the proof that it is quasi-inverse to $M_{IS}$), due to Faltings--Chai \cite[Ch. II]{FC}, use in an essential way the Fourier expansion principle established in \cite[Ch. II, Theorem 4.1]{FC} (closely related to the special case of \Cref{mainthm} in the case $i = 0$ and $\LL$ ample, but for a general degeneration --- see \cite[Ch. III, Theorem 6.2]{FC} and better yet \cite[Corollary 4.5.4.24]{lan}). 

Since the construction of $M$ is not literally by quotienting $\widetilde{G}$ by $Y$, it is not \emph{a priori} obvious that every object of $\DEG_{\amplegg}^{\spli}$ can be uniformized using $M_{ample}$, nor is it obvious that the basic functor $M: \DD_{g, g}^{\spli} \to \DEG_{g, g}^{\spli}$ is essentially surjective. Therefore, Faltings--Chai devised the following method, based on the technique of \emph{Raynaud extensions}. Given an object $(G, \Lt)$ in $\DEG_{\amplegg}^{\spli}$, one canonically associates to $G$ a split torus over $R$ by taking the formal completion $\widehat{G}$ along $I$, which yields a formal split torus over $\Spf R$ since the deformation theory of a split torus is trivial by \cite[Expos\'e IX Th\'eor\`eme 3.6]{SGA3}. That formal split torus is uniquely algebrizable as the formal completion of a split torus $\widetilde{G}$ defined by character lattice $X$, which is the torus that we expect to uniformize $G$ via $M$.

In order to uniformize $G$ using $\widetilde{G}$, one still needs to determine the correct periods $i: Y \to \widetilde{G}_K$ to quotient by (i.e. plug into the functor $M$). This is the non-obvious part that Faltings--Chai innovated to solve. The key idea in \cite[Ch. II]{FC} is that, analogously to the analytic case as explained in \Cref{ana}, the extra data of $i: Y \to \widetilde{G}_K$ can be read off of the Fourier expansion (a concept we explain in the next paragraph, expanding on \Cref{periods-recover}) of algebraic theta functions on $G$, that is, nonzero global sections in $H^0(G_K, \mathscr{L}_K)$. A relevant subtle technical point\footnote{We point the reader only to \cite{lan}, as its arguments are a correction of a circular argument originally appearing in \cite{FC} --- see \cite[Remark 4.2.1.17]{lan}.} \cite[\S 4.3.4 and \S 4.5.5]{lan} is that while reading $i$ off of $G$ requires a choice of extension $(G, \mathscr{L})$ of $G$ to $\DEG_{\amplegg}^{\spli}$, it is not sensitive to the choice of $\LL$ and is in fact functorial.

Having replaced the analytic mechanism of pulling back to the universal cover via the covering map with the algebraic mechanism of passing to $\widehat{\widetilde{G}}$ by taking formal completion, the mechanism for the existence of Fourier expansions also changes along those lines: the ``Fourier expansion'' map is the natural map
\[H^0(G, \mathscr{L}) \hookrightarrow H^0(\widehat{G}, \widehat{\mathscr{L}}), s \mapsto \widehat{s}\]
which also applies to the generic fiber just by applying $- \otimes_R K$ to get
\[H^0(G_K, \LL_K) \hookrightarrow H^0(\widehat{G}, \widehat{\mathscr{L}}) \otimes_R K.\]
The data of a rigidification on $\mathscr{L}$ is the same as a cubical structure on it, by \cite[Ch. I, \S 2.6]{MB}. By the proof of \cite[Corollary 3.3.3.3]{lan}, the line bundle $\mathscr{L}^\wedge$ on $\widehat{G}$ is trivial. So $\widehat{G}$ being the formal completion of a split torus $\Gt$ over $R$, and $\widehat{\mathscr{L}}$ being the formal completion of its structure sheaf, every $\widehat{s} \in H^0(\widehat{G}, \widehat{\mathscr{L}}) = H^0(\Gt^\wedge, \widehat{\mathcal{O}})$ has a ``Fourier expansion'' 
\[\widehat{s} = \sum_{\chi \in X} \sigma_\chi(s) \cdot \underline{Z}^\chi,\]
where the $\sigma_\chi(s) \in R$ must converge $I$-adically to zero (this is another way of saying that for all finite $k \geq 1$, $s|_{R/I^k}$ lives in $H^0(\BG_m^g \times_R R/I^k, \mathcal{O})$ and is therefore a finite sum of various $\chi \in X$ weighted by elements of $R/I^k$, hence is an infinite sum that is finite mod $I^k$ for all $k \geq 1$). Therefore, every
\[s \in H^0(G_K, \mathscr{L}_K) = H^0(G, \mathscr{L}) \otimes K\]
has a Fourier expansion which is an infinite sum
\[s = \sum_{\chi \in X} \sigma_\chi(s) \cdot \underline{Z}^\chi,\]
where $\sigma_{\chi}(s) \in K$ can all be multiplied by a common element $r \in R$ to make them all reside in $R$, and $r \sigma_\chi(s)$ converge $I$-adically to zero. 

To explain how Faltings--Chai extract the periods $i: Y \to \widetilde{G}_K$ from the functionals $\sigma_\chi$, we need some notation related to the dual of $G$, following \cite[Ch. II, \S 2]{FC}. The dual abelian variety $G_K^\vee$ extends to a degeneration $G^\vee$ for which the formal completion along $I$ is a formal rank $g$ split torus with character lattice defined to be $Y$. And the polarization $\phi_{\mathscr{L}_K}: G_K \to G_K^\vee$ extends to a map of degenerations $\phi_{\mathscr{L}}: G \to G^\vee$, which induces a map $\phi: Y \to X$ on character lattices after passing to formal completions as above.  

In analogy to (\ref{fourierpairing}), Faltings--Chai \cite[Theorem II.4.1]{FC} proved that $s \in H^0(G_K, \mathscr{L}_K)$ have Fourier coefficients $\sigma_\chi(s)$ satisfying $\sigma_\chi(-) \neq 0$ as a functional on $H^0(G_K, \mathscr{L}_K)$ and 
\[\sigma_{\chi + \phi(y)}(s) = \zeta(y)\tau(y, \chi)\sigma_{\chi}(s)\]
where $\zeta: Y \to K^\times$ and $\tau: Y \times X \to K^\times$ is bimultiplicative and satisfies
\begin{equation}\label{coeffs-qp}\frac{\zeta(y_1 + y_2)}{\zeta(y_1)\zeta(y_2)} = \tau(y_1, \phi(y_2)) \qquad \text{for all } y_1, y_2 \in Y.\end{equation}
The function $\zeta$ has to do with the specific choice of line bundle $\mathscr{L}$ and the trivialization of $\widehat{\mathscr{L}}$, but \cite[\S 4.3.4 and \S 4.5.5]{lan} shows that the pairing $\tau$ only depends on $G$ --- indeed, $\tau$ defines embeddings that we may suggestively call $i: Y \hookrightarrow \widetilde{G}_K$ and $i^\vee: X \hookrightarrow \widetilde{G}^\vee_K$, where $\widetilde{G}$ and $\widetilde{G}^\vee$ are the split tori whose formal completions are $\widehat{G}$ and $\widehat{G^\vee}$, and whose character lattices are defined to be $X$ and $Y$, respectively. The embedding $i$ is precisely the one we would hope to extract from a degeneration in order to uniformize it by a degeneration datum. And the function $\zeta: Y \to K^\times$ further defines a $Y$-equivariant structure on the $\Gt_K$-fiber of the structure sheaf on $\Gt$; that is, $\zeta$ and $\tau$ together are equivalent to the data of an object of $\DD_{\amplegg}^{\spli}$ (part of the content of \cite[Ch. II, Theorem 4.1]{FC} is that they satisfy the required positivity condition) via the discussion after \Cref{DDIS}. This procedure defines a functor $F_{ample}: \DEG_{\amplegg}^{\spli} \to \DD_{\amplegg}^{\spli}$ fitting into a diagram
\begin{center}
\begin{tikzcd}
\DD_{\amplegg}^{\spli} \arrow[d] & \DEG_{\amplegg}^{\spli} \arrow[l, "F_{ample}"] \arrow[d] \\
{\DD_{g, g}^{\spli}}             & {\DEG_{g, g}^{\spli}} \arrow[l, "F"]                    
\end{tikzcd}
\end{center}
 
Finally, the main theorem of \cite[Ch. II-III]{FC} is that $F$ and $M$ are quasiinverse functors, thus proving the analog of the Appell--Humbert theorem for degenerating families of polarized abelian varieties in the sense that every polarizing line bundle on a split total degeneration $G$ can be uniformized by a trivial line bundle. The proof in \cite{FC} of this theorem goes through \cite[Ch. III, Theorem 6.2]{FC}, which is a complete construction of a basis of Fourier expansions that can appear in $H^0(G, \mathscr{L}) \otimes K$, where $(G, \mathscr{L})$ is already given as $M(\widetilde{G}, i: Y \hookrightarrow \widetilde{G}_K, \Lt)$ for some input in $\DD_{\amplegg}^\spli$. We pull the statement from \cite[Corollary 4.5.4.24]{lan} because it is stated in a way that is most notationally similar to how we wish to view the Fourier expansion principle in terms of group cohomology of the lattice of periods $Y$:
\begin{theorem}\label{definite}
  The global sections $H^0(G_K, \LL)$ are canonically identified with the $Y$-fixed points $H^0(Y, V)$, where $V$ is the $K$-vector space of formal power series
  \[\vartheta(\underline{Z}) = \sum_{\chi \in X} \sigma_\chi(\vartheta) \underline{Z}^\chi\]
  such that $\sigma_\chi(\vartheta) \in K$ are such that there exists $r \in R$ with $r \cdot \sigma_\chi(\vartheta) \in R$ for all $\chi \in X$, and such that $ r \vartheta(i(y) \cdot \underline{Z})$ has coefficients going to zero for all $y \in Y$; and the $Y$-action on $V$ is given by
  \[(y \cdot \vartheta)(\underline{Z}) := \zeta(y) \underline{Z}^{\phi(y)}\vartheta(i(y) \cdot \underline{Z}),\] 
  that is, by sending the $\chi$-th Fourier coefficient $\sigma_\chi$ to $\sigma_{\chi + \phi(y)} \zeta^{-1}(y)\tau^{-1}(y, \chi)$.  
\end{theorem}
When $G$ is a Kuga--Sato degeneration, as previously mentioned, this is the $i = 0$, $\LL$ ample case of \Cref{mainthm}. 

We have now finished explaining how the Faltings--Chai classification of degenerations of ample line bundles includes as part of its proof the complete determination of Fourier expansions of algebraic theta functions for ample line bundles. In the remainder of this article, we will use the Faltings--Chai uniformization of line bundles on Kuga--Sato degenerations (and thus implicitly that $i = 0$ ample special case of \Cref{mainthm}) to deduce \Cref{mainthm} as well as the accompanying results promised in \Cref{app}.

\subsection{Relatively complete models for invertible sheaves}\label{modified}
We now discuss the ability to use relatively complete models to calculate ``in one shot'' the degenerations uniformized by objects in $\DD_{\ISgg}^{\spli}$ rather than first decomposing as a difference of Mumford-ample objects and applying \Cref{FCdef}. The fact that such relatively complete models can be used is surely well-known, and is mentioned without explicit detail in \cite[p. 72]{FC}.
\begin{definition}
  A \emph{relatively complete model} for an object $(\widetilde{G}, i: Y \hookrightarrow \widetilde{G}_K, \widetilde{\mathscr{L}})$ of $\DD_{\ISgg}^{\spli}$ is a triple $(\widetilde{P}, \widetilde{\mathscr{N}}, \widetilde{\mathscr{L}})$, where $(\widetilde{P}, \widetilde{\mathscr{N}})$ is a relatively complete model for the object $(\widetilde{G}, i: Y \to \widetilde{G}_K)$ of $\DD_{g, g}^{\spli}$ in the sense of \Cref{rcm}, and $\widetilde{\mathscr{L}}$ is, somewhat abusively, a line bundle $\widetilde{\mathscr{L}}$ on $\widetilde{P}$ with a $Y$-equivariant structure and a $\widetilde{G}$-equivariant structure, such that $\widetilde{\mathscr{L}}|_{\widetilde{G}_K}$ agrees with the $\widetilde{\mathscr{L}}$ in the degeneration datum as a $Y$-equivariant structure on the structure sheaf.  
\end{definition}

\begin{definition}[Our definition of $M_{IS}$]\label{ourdef}
Let $D$ be an arbitrary object of $\DD_{\ISgg}^{\spli}$. If $D$ admits a relatively complete model $(\Pt, \widetilde{\mathscr{N}}, \widetilde{\LL})$, then $M_{IS}(D) = (G, \LL)$ may equivalently be defined by constructing the line bundle $\Lt_k := \Lt \times_R R/I^k$ on $\Pt_k$ for all $k$ (as is done for $\widetilde{\mathscr{N}}$ in step (2) of \Cref{mumford-construction}), then use the ample line bundle $(\mathscr{N}_k)_{k \geq 1}$ on $\mathfrak{P}$ to algebrize $(\Lt_k)_{k \geq 1}$ to a line bundle on $X$ (as in step (3) of \Cref{mumford-construction} except now the structure $\Lt_k$ is also being algebrized --- this time via \cite[Corollaire (5.1.6)]{EGAIII}), which can then be restricted to $G \subset X$ which is defined by step (4) of \Cref{mumford-construction}. That restriction is a line bundle $\LL$, and $M_{IS}(D)$ is defined to be $(G, \LL)$. 
\end{definition}
\begin{lemma}
  Let $D$ be an arbitrary object of $\DD_{\ISgg}^{\spli}$. Then:
  \begin{enumerate}
  \item[(1)] $D$ admits a relatively complete model.
\item[(2)]
    $M_{IS}(D)$ as defined in \Cref{ourdef} indeed agrees with \Cref{FCdef}.
    \end{enumerate}
\end{lemma}
\begin{proof}
  The argument of \cite[Construction 4.5.4.6]{lan} proves that there are Mumford-ample objects $(\Gt, Y \to \Gt_K, \Lt_1)$ and $(\Gt, Y \to \Gt_K, \Lt_2)$ in $\DD_{\ISgg}^{\spli}$ such that
  \[\Lt = \Lt_1 \otimes \Lt_2^{-1}.\]
  Let $(\Pt_1, \Lt_1)$ and $(\Pt_2, \Lt_2)$ be the relatively complete models for $\Lt_1$ and $\Lt_2$ that are guaranteed to exist by the definition of a Mumford-ample line bundle on a degeneration. If $\Pt_1$ and $\Pt_2$ were secretly the same, then we would immediately have a proof of (1): we could just form $\Lt = \Lt_1 \otimes \Lt_2^{-1}$ and combine it with $\widetilde{\mathscr{N}} = \Lt_1$ (or $\Lt_2$) to obtain our desired relatively complete model. More generally, we can apply the technique (already present in \cite{lan, FC}) of mimicking the construction of $\Lt_1 \otimes \Lt_2^{-1}$ by restriction of $\Lt_1 \boxtimes \Lt_2^{-1}$ along the diagonal: define $\Pt$ to be the Zariski closure of $\Gt$ along the diagonal embedding $\Gt \hookrightarrow \Pt_1 \times \Pt_2$. The restriction of $\Lt_1 \boxtimes \Lt_2^{-1}$ to $\Pt$ is denoted $\Lt$, as it provides a relatively complete model $(\Pt, \Lt, \widetilde{\mathscr{N}})$ for $\Lt_1 \otimes \Lt_2^{-1}$, where $\widetilde{\mathscr{N}}$ is the restriction of the ample line bundle $\Lt_1 \boxtimes \Lt_2$ to $\Pt$. We refer to the proof of \cite[Lemma 4.5.4.8]{lan} for an extremely similar argument. This proves (1). Moreover, $M_{IS}(D)$ defined using \emph{this particular choice} of relatively complete model in \Cref{ourdef} agrees with \Cref{FCdef}. 

  To prove (2), it therefore suffices to prove that \Cref{ourdef} doesn't depend on the choice of relatively complete model. Let $\LL$ be the algebrization of $\Lt^\wedge$ on $X$. Similarly, let $\mathscr{N}$ be the algebrization of $\Nt^\wedge$ on $X$. Since $\mathscr{N}$ is ample, by \cite[Corollaire (4.5.8)]{EGAII} there exists a positive integer $N$ such that 
  \[\LL \otimes \mathscr{N}^{\otimes N} \text{ is ample}.\]
  It follows from the same arguments using Fourier expansions that show $F_{ample}$ and $M_{ample}$ are quasiinverse (\Cref{raynaud}) that $F_{ample}(\LL \otimes \mathscr{N}^{\otimes N}) = \Lt \otimes \Nt^{\otimes N}$, i.e., our $M_{IS}(\Lt) \otimes M_{IS}(\Nt)^{\otimes N}$ agrees with $M_{IS}(\Lt) \otimes M_{IS}(\mathscr{N})^{\otimes N}$ defined by \Cref{FCdef}, hence $M_{IS}(\Lt)$ agrees with it defined Mumford's way as well. 
\end{proof}

\subsection{Line bundles on Kuga--Sato degenerations, from the point of view of Faltings--Chai}\label{KSDD}
We now use the characterization of $M_{IS}$ from \Cref{ourdef} to explicitly uniformize the Kuga--Sato degenerations. We will use this to prove the main theorems of this article in the subsequent sections.

Fix positive integers $g, w$, and let $R = \pwr$ (complete with respect to the ideal $I = (q^{1/w})$), so that the split $R$-torus $\Gt = \BG_{m, \pwr}^g$ is a valid uniformizing torus for a degeneration datum. The character lattice of $\Gt$ is $X = \BZ^g$, where $\chi = (\chi_1, \ldots, \chi_g) \in X$ represents the character given on points by $\underline{Z} \mapsto \underline{Z}^\chi$. We take the period lattice $Y = \BZ^g$ to embed into $\Gt_K$ via
\[i: Y \hookrightarrow \Gt_{\laurent}, \qquad \underline{y} = (y_1, \ldots, y_g) \mapsto (q^{y_1}, \ldots, q^{y_g}).\]
In other words, $i$ is defined by the pairing $\tau$, where 
\[\tau: Y \times X \to \laurent^\times \subset K^\times\]
is given by
\[\tau(y, \chi) := (i(y))^{\chi} =  q^{\chi_1 y_1 + \cdots + \chi_gy_g}.\]
\begin{lemma} The tuple $D := (\Gt, i: Y \to \Gt_{\laurent})$ defines an object of $\DD_{g, g}^{\spli}$.
\end{lemma}
\begin{proof}
  We need only verify the positivity requirement in \Cref{DDdef}. So let $\phi: Y \to X$ be the ``principal polarization'' map sending $(y_1, \ldots, y_g)$ to the character $\underline{Z} \mapsto Z_1^{y_1} \cdots Z_g^{y_g}$, and note that
  \[\tau(y, \phi(y)) = q^{\langle y, \phi(y) \rangle} = q^{y_1^2 + \cdots + y_g^2} \in (q^{1/w}) \text{ unless } y = 0.\]
\end{proof}

In the remainder of this section, we will accomplish two tasks:

\begin{enumerate}
\item We will explicitly classify all the ways to extend $D$ to an object of $\DD_{\ISgg}^{\spli}$.
\item For all of those ways, we will use toric geometry to produce explicit relatively complete models in the sense of \Cref{ourdef}. We will compare the resulting objects in $\DEG_{\ISgg}^{\spli}$ to those defined more directly in \Cref{intro2} and see that they are in fact the same.  
\end{enumerate}

To extend $D$ to an object of $\DD_{\ISgg}^{\spli}$ is to produce a rigidified line bundle $\Lt$ on $\Gt$ and a $Y$-equivariant structure on its generic fiber. We already used (from \cite[Remarque 7.2.4]{MB}) the fact that as a line bundle, $\Lt$ is isomorphic to the structure sheaf of $\Gt$. So the $Y$-equivariant structure that we desire to construct is the same as a class in\footnote{The following is standard, but we specify it here to make it unambiguous what the construction is: to specify a $Y$-equivariant structure on $\mathcal{O}$ is to specify isomorphisms $\iota(y)^*\mathcal{O} \overset{\sigma(y)}{\to} \mathcal{O}$, which have to be compatible in the sense that for $y_1, y_2 \in Y$, the isomorphism $\iota(y_1 + y_2)^*\mathcal{O} \overset{\sigma(y_1 + y_2)}{\to} \mathcal{O}$ coincides with the composition $\iota(y_1 + y_2)^*\mathcal{O} = \iota(y_1)^*\iota(y_2)^*\mathcal{O} \overset{\iota(y_1)^*(\sigma(y_2))}{\to} \iota(y_1)^*\mathcal{O} \overset{\sigma(y_1)}{\to} \mathcal{O}$, i.e. the $\sigma$ indeed has to satisfy the cocycle condition $\sigma(y_1 + y_2) =  \sigma(y_1) \cdot \iota(y_1)^*(\sigma(y_2))$. All of the isomorphisms $\sigma(y)$ are to be viewed as elements of $H^0(\widetilde{G}_K, \mathcal{O}^\times) = K[X]^\times = K^\times X$. Two $Y$-equivariant structures $\sigma, \sigma'$ are equivalent if there is an $f \in K^\times X$ with the property that $\sigma_1(y)\sigma_2(y)^{-1} = \frac{\iota(y)^*f}{f}$ for all $y \in Y$, meaning that it is exactly $H^1(Y, K^\times X)$ that classifies the isomorphism classes of $Y$-equivariant structures on $\mathcal{O}$.} $H^1(Y, \mathcal{O}(\Gt_K)^\times)$. We already saw following \Cref{DDIS} that such a cocycle is always of the form $y \mapsto \alpha(y) = \zeta(y) \cdot \underline{Z}^{\phi(y)}$, where
\begin{equation}\label{nthcocycle}\frac{\zeta(y_1 + y_2)}{\zeta(y_1)\zeta(y_2)} = \tau(y_1, \phi(y_2)).\end{equation}
and $\zeta: Y \to K^\times$, while $\phi: Y \to X$ is a group homomorphism. For an arbitrary fixed choice of $\phi$ (though only $\phi$ with finite cokernel will produce line bundles that induce isogenies $G_{\laurent} \to G_{\laurent}^\vee$), we determine all possible $\zeta$. First of all, if $\phi$ is given in the standard bases of $X$ and $Y$ by a matrix $B = (b_{ij})$, then $B$ must be symmetric, since the left hand side of \Cref{nthcocycle} is symmetric in $y_1$ and $y_2$. Otherwise, no $\zeta$ can possibly exist. This is the degeneration datum manifestation of the symmetry property of the maps induced by line bundles on abelian varieties.

On the other hand, once $B$ is symmetric, there is a full $\Hom(Y, K^\times)$-torsor of possibilities for $\zeta$ satisfying \Cref{nthcocycle}. Namely,
\[\tau(y_1, \phi(y_2)) = q^{\langle y_1, B(y_2)\rangle} = q^{\sum_{i, j} b_{ij}(y_1)_i(y_2)_j},\]
so we may take
\[\zeta(\uy) = q^{\frac{1}{2}\sum_i b_{ii}(y_i^2 - y_i) + \sum_{i < j} b_{ij} y_iy_j} = q^{Q'(\underline{y})}\] 
which can be modified by any element of  $\Hom(Y, K^\times)$, that is, any collection $a_1(q), \ldots, a_g(q) \in K^\times$ to produce
\[\zeta(y) = q^{Q'(\uy)} \cdot \prod_{i = 1}^g a_i(q)^{y_i}.\]

The only remaining condition that must be required is the \emph{finiteness condition}: there has to exist a positive integer $n$, a $\phi: Y \to X$ satisfying the positivity condition with respect to $\tau$ such that
\[\zeta(y) \cdot \tau(y, \phi(y))^n \in \pwr \text{ for all } y \in Y.\]
This is satisfied if and only if all the $a_i \in \laurent^\times \subset K^\times$.

Therefore, we have shown that the following definition accounts for \emph{all} of the objects of $\DD_{\ISgg}^{\spli}$ extending $D$:
\begin{definition}\label{DQa}
  For a fixed symmetric $\BZ$-entried matrix $B$, and a tuple $\underline{a} = (a_1, \ldots, a_g)$ of elements of $\laurent^\times$, let $D_{B, a}$ denote the object of $\DD_{\ISgg}^{\spli}$ given by the cocycle
  \[\uy \mapsto q^{Q'(\uy)} \underline{Z}^{B(\uy)} \prod_{i = 1}^g a_i(q)^{y_i}.\] 
\end{definition}

Now we begin the task of constructing a  relatively complete model for $D_{B, \underline{a}}$, using the theory of torus embeddings for the split $\pwr$-torus $\Gt$. We refer to \Cref{toric-pdf} for the notations involving the basic theory of torus embeddings over $\pwr$.

Let $\widetilde{P} = X_\Sigma \times_{\BZ[q^{1/w}]} \pwr$, where $\Sigma$ is the fan in $(\BZ \oplus Y) \otimes \BR$ given by the cone over a specific convex polyhedral decomposition of $\{1\} \times (Y \otimes \BR)$ (compare with \cite[\S 6]{mumford_compositio}). The symbol $X_\Sigma$ is notation for the toric scheme over $\BZ$ coming from the fan $\Sigma$. The coordinate $q^{1/w}$ is the $\BG_m$ character corresponding to the first coordinate in $\BZ \oplus X$, and the coordinates $Z_1, \ldots, Z_g$ are used for the usual basis of $X$.  Here we obtain a $\BZ[q^{1/w}]$-scheme $X_\Sigma$ in the first place by noting that by the definition of $\Sigma$ as a cone filling up an entire half-space, $X_\Sigma$ admits a natural map to $\mathbf{A}^1_{\BZ}$ (the toric $\BZ$-scheme from the fan given by a ray in $\BZ \otimes \BR$) with coordinate $q^{1/w}$. Importantly, $\Pt$ is only locally of finite type over $\pwr$, because $\Sigma$ is infinite. This departs from the typical situation in most toric geometry references like \cite{demazure, fulton} so it requires some care to handle rigorously --- the necessary pieces of the basic theory are adapted in \Cref{toric-pdf}. 

In an attempt to avoid confusion, we use $t_1, \ldots, t_g$ as the coordinates of $Y \otimes \BR \simeq \BR^g$, and we use $t_0$ as the extra coordinate on $(\BZ \oplus Y) \otimes \BR$. We only use $y_1, \ldots, y_g$ to denote the coordinates of an element of $Y \otimes \BR$ if it is already in the lattice $Y \subset Y \otimes \BR$. 

We now specify the polyhedral decomposition we use to construct $\Sigma$. It is taken to be the standard triangulation of the standard tiling by unit hypercubes of $Y \otimes \BR \simeq \BR^g$. That is, we triangulate the hypercube $[0, 1]^g$ by covering it with $g!$ $g$-simplices indexed by the permutation group $S_g$,
\begin{equation}\label{cube}\sigma_\pi = \{(t_1, \ldots, t_g) \in [0, 1]^g : 0 \leq t_{\pi(1)} \leq \cdots \leq t_{\pi(g)} \leq 1\}.\end{equation}
Then we triangulate $Y \otimes \BR \simeq \BR^g$ by the infinite collection of $g$-simplices
\[\sigma_{y, \pi} = y + \sigma_\pi, \qquad y \in Y.\]
The cone $\Sigma$ in $(\BZ \oplus Y) \otimes \BR$ over this triangulation is defined by the collection of cones
\[\widetilde{\sigma}_{y, \pi} = \BR_{\geq 0} \cdot (1 \times \sigma_{y, \pi}).\]

This notation being somewhat tiresome to parse, we provide some pictures of our choice of $\Sigma$ in the cases $g = 1$ and $g = 2$: \Cref{fig1} and \Cref{fig2}.

The action of $Y$ on $\Pt$ is defined to be induced by base change to $\pwr$ from the action of $Y$ on $(\BZ \oplus Y) \otimes \BR$ that is linear on cones and sends $(1, \underline{t})$ to $(1, \underline{t} + w\underline{y}')$. 
\begin{figure}
  \caption{The fan $\Sigma$ we use for $g = 1$}\label{fig1}
  \centering
\begin{tikzpicture}[
  x=1cm,y=-1cm,
  font=\small,
  line cap=round,
  line join=round,
  >={Stealth[length=1.7mm]},
  ray/.style={gray!75,line width=.55pt,->}
]
  \def\N{8}   
  \def\R{4.5} 

  \foreach \n in {-\N,...,\N} {
    \pgfmathsetmacro{\s}{\R/sqrt(1+\n*\n)}
    \coordinate (r\n) at (\s,{\s*\n});
  }

  \fill[blue!12]   (0,0)--(r0)--(r1)--cycle;
  \fill[orange!18] (0,0)--(r1)--(r2)--cycle;

  \foreach \n in {-\N,...,\N}
    \draw[ray] (0,0)--(r\n);

  \draw[gray,densely dashed] (1,-3.8)--(1,3.8); 

  \foreach \n in {-3,...,3}
    \fill (1,\n) circle (1.35pt);

  \node[fill=white,inner sep=1pt]
    at (1,-4.13) {$\vdots$};
  \node[fill=white,inner sep=1pt]
    at (1,4.13) {$\vdots$};

  \fill (0,0) circle (1.6pt);
  \node[left=3pt] at (0,0) {$(0, 0)$};
\end{tikzpicture}
\end{figure}
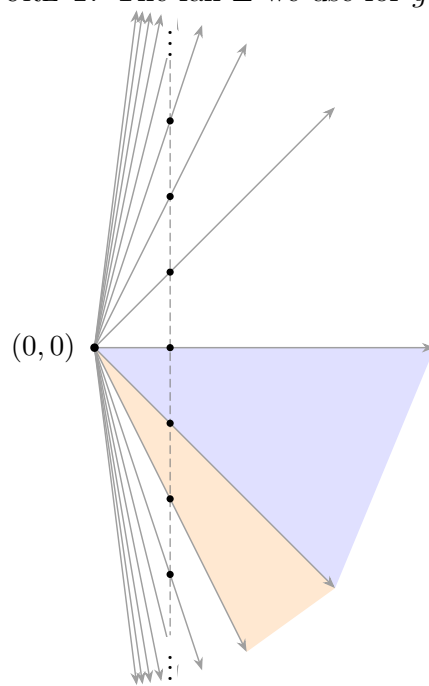

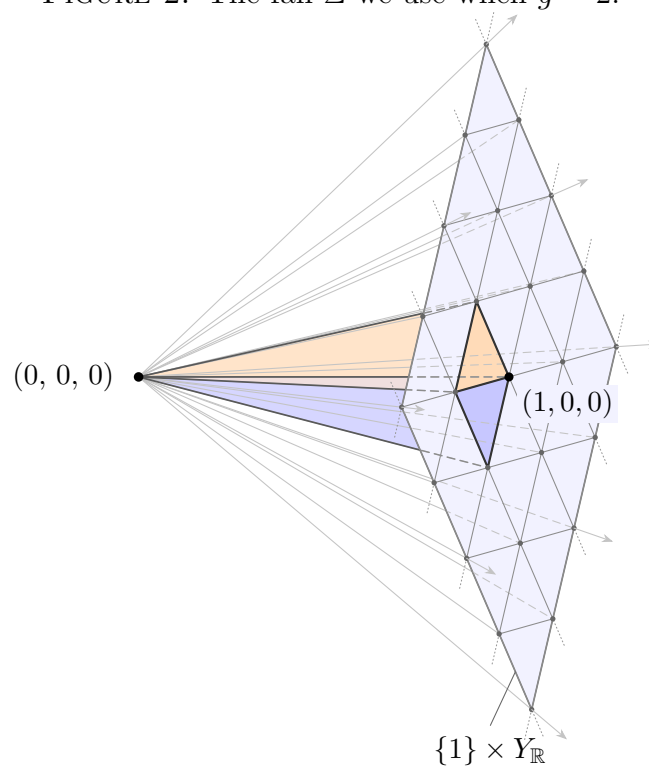
\begin{figure}
  \caption{The fan $\Sigma$ we use when $g = 2$.}\label{fig2}
  \centering
\begin{tikzpicture}[
  x={(4.90cm,0cm)},
  y={(-.28cm,-1.20cm)},
  z={(-.43cm,1.00cm)},
  font=\small, line cap=round, line join=round,
  >={Stealth[length=1.5mm]},
  mesh/.style={black!45,line width=.4pt},
  ray/.style={black!23,line width=.4pt},
  continuation/.style={black!40,densely dotted,line width=.4pt}
]
  \def\N{2}
  \def\H{1.09}
  \pgfmathtruncatemacro{\Nm}{\N-1}
  \pgfmathsetmacro{\E}{\N+.35}
  \coordinate (O) at (0,0,0);
  \coordinate (A) at (1,0,0);
  \coordinate (B) at (1,1,0);
  \coordinate (C) at (1,1,1);
  \coordinate (D) at (1,0,1);

  \fill[blue,opacity=.07] (O)--(A)--(B)--cycle;
  \fill[blue,opacity=.10] (O)--(B)--(C)--cycle;
  \fill[orange,opacity=.09] (O)--(A)--(D)--cycle;
  \fill[orange,opacity=.13] (O)--(D)--(C)--cycle;
  \foreach \i in {-\N,...,\N}
    \foreach \j in {-\N,...,\N}
      \draw[ray] (O)--(1,\i,\j);
  \foreach \P in {A,B,C,D}
    \draw[black!65,line width=.65pt] (O)--(\P);

  \filldraw[fill=blue!5,draw=black!50,line width=.65pt]
    (1,-\N,-\N)--(1,\N,-\N)--
    (1,\N,\N)--(1,-\N,\N)--cycle;
  \fill[blue!22] (A)--(B)--(C)--cycle;
  \fill[orange!28] (A)--(D)--(C)--cycle;

  \begin{scope}
    \clip (1,-\N,-\N)--(1,\N,-\N)--
          (1,\N,\N)--(1,-\N,\N)--cycle;
    \foreach \i in {-\N,...,\N}
      \foreach \j in {-\N,...,\N}
        \draw[ray,densely dashed] (O)--(1,\i,\j);
    \foreach \P in {A,B,C,D}
      \draw[black!45,densely dashed,line width=.55pt] (O)--(\P);
  \end{scope}

  \foreach \k in {-\N,...,\N} {
    \draw[mesh] (1,-\N,\k)--(1,\N,\k);
    \draw[mesh] (1,\k,-\N)--(1,\k,\N);
    \draw[continuation] (1,-\E,\k)--(1,-\N,\k);
    \draw[continuation] (1,\N,\k)--(1,\E,\k);
    \draw[continuation] (1,\k,-\E)--(1,\k,-\N);
    \draw[continuation] (1,\k,\N)--(1,\k,\E);
  }
  \foreach \i in {-\N,...,\Nm}
    \foreach \j in {-\N,...,\Nm}
      \draw[mesh] (1,\i,\j)--(1,{\i+1},{\j+1});

  \draw[black!80,line width=.8pt] (A)--(B)--(C)--(D)--cycle;
  \draw[black!80,line width=.8pt] (A)--(C);
  \foreach \i in {-\N,...,\N}
    \foreach \j in {-\N,...,\N}
      \fill[black!60] (1,\i,\j) circle (1pt);

  \foreach \i/\j in {-2/-2,0/-2,2/-2,-2/0,2/0,-2/2,0/2,2/2}
    \draw[ray,->] (1,\i,\j)--(\H,{\H*\i},{\H*\j});

  \draw[black!60] (1,2,-1.5)--++(0,.8,.25)
    node[below,align=left,text=black]
      {$\{1\}\times Y_{\mathbb R}$};
  \fill (A) circle (1.7pt);
  \node[below right=3pt,fill=blue!5,inner sep=1.5pt]
    at (A) {$(1,0,0)$};
  \fill (O) circle (1.8pt);
  \node[left=5pt] at (O) {(0, 0, 0)};
\end{tikzpicture}
\end{figure} 
 
The basic reason for selecting this triangulation is the following
\begin{lemma}
The fan $\Sigma$ is nonsingular.
\end{lemma}
\begin{proof}
  We need to show that for all $y \in Y = \BZ^g$ and $\pi \in S_g$, the $g + 1$ vertices of $1 \times \widetilde{\sigma}_{y, \pi}$ are a basis for $\BZ \oplus Y$. Those vertices are precisely the points in $\BR^g$ given by $(1, v_i)$, where $i = 0, \ldots, g$, and $v_i$ has coordinates $v_{\pi(1)}= \cdots =  v_{\pi(i)} = 0 $, $v_{\pi(i+1)} = \cdots = v_{\pi(g)} = 1$. Using the basis $e_{\pi(1)}, e_{\pi(2)}, \ldots, e_{\pi(g)}$ for $Y \otimes \BR$, this comes down to the fact that 
  \[\det \left(\begin{matrix}1 & 1 & 1 & \cdots & 1 & 1 & 1\\
    1 & 0 & 0 & \cdots & 0 & 0 & 0\\
    1 & 1 & 0 & \cdots & 0 & 0 & 0\\
    1 & 1 & 1 & \cdots & 0 & 0 & 0\\
    \vdots & \vdots & \vdots & \ddots & \vdots & \vdots & \vdots \\
    1 & 1 & 1 & \cdots & 1 & 0 & 0\\
    1 & 1 & 1 & \cdots & 1 & 1 & 0\\
  \end{matrix}\right) = \pm \det \left(\begin{matrix}
    1 & 0 & 0 & \cdots & 0 & 0 & 0\\
    1 & 1 & 0 & \cdots & 0 & 0 & 0\\
    1 & 1 & 1 & \cdots & 0 & 0 & 0\\
    \vdots & \vdots & \vdots & \ddots & \vdots & \vdots & \vdots \\
    1 & 1 & 1 & \cdots & 1 & 0 & 0\\
    1 & 1 & 1 & \cdots & 1 & 1 & 0\\
    1 & 1 & 1 & \cdots & 1 & 1 & 1\\
  \end{matrix}\right) = \pm 1.\]
\end{proof}
One embeds $\widetilde{G} = \Spec \pwr[X] = \BG_{m, \pwr}^g$ into $\widetilde{P}$ as the affine from the cone \[\BR_{\geq 0} \cdot (1, 0, 0, \ldots, 0, 0).\] The fact that the maximal cones in $\Sigma$ are smooth means that we can define a piecewise linear function $\psi \in \mathbf{PL}(\Sigma)$ (and therefore a line bundle on $\widetilde{P}$) by just specifying its values on the vertices  $(1, y), y \in Y$.

\begin{remark}
It will be convenient for some of what we do to remark that even after refining $\Sigma$ to a finer cone $\Sigma'$, we can still define a line bundle on $X_{\Sigma'}$ by a piecewise linear function $\psi \in \mathbf{PL}(\Sigma)$. Thus we can take advantage of the nonsingular nature of $\Sigma$ to construct $\psi$ even if we ultimately do not work with the exact fan $\Sigma$. 
\end{remark}

We first construct the required ample line bundle $\widetilde{\mathscr{N}}$ on $\widetilde{P}$. Since it will be constructed from a function $\psi \in \PL(\Sigma)$, the only things that actually require checking are:

\begin{itemize}
  \item We have to arrange for $\widetilde{\mathscr{N}}$ to be ample. 
  \item We have to produce a $Y$-equivariant structure on $\widetilde{\mathscr{N}}$.
\end{itemize}
  For the first point, we will use the ampleness criterion \Cref{ampleness} (one has to slightly modify the usual ampleness criterion for line bundles on toric varieties because there is no finite-type hypothesis available, as $\Sigma$ is infinite). 
Let $\mathsf{Q}: Y \to \BZ$ be the (quadratic) function of the form
\[\mathsf{Q}(y_1, \ldots, y_g) = \sum_{i = 1}^g a_{ii}y_i^2 + 2\sum_{1 \leq i < j \leq g} a_{ij} y_iy_j,\]
where $a_{ij} \in \BZ$ and $a_{ij} = a_{ji}$. 
By smoothness of $\Sigma$, there is a unique $\psi \in \mathbf{PL}(\Sigma)$  with the property that $\psi(1, y) = -\mathsf{Q}(y)$ for all $y \in Y$.
\begin{lemma}\label{ample-exists}
The line bundle $\mathscr{N}$ coming from $\psi$ in the toric language is ample, if $\mathsf{Q}$ is chosen appropriately. 
\end{lemma}
\begin{proof}
  We apply the ampleness criterion of \Cref{ampleness}.
  As in the finite-type toric geometry situation, the main thing is to prove that $\psi$ is strictly convex. First, it is clear that not any positive-definite $Q$ will work. For example, if $Q(y) = \sum_i y_i^2$ (or any other diagonal form), then $\psi$ is linear on all of $\cup_\pi \widetilde{\sigma}_{y, \pi}$ for all $y \in Y$, since if $\epsilon = \sum_{i \in S} e_i$, then 
  \[\psi: (1, y + \epsilon) \mapsto -\mathsf{Q}(y) - \sum_{i \in S} 2y_i - |S| \]
  while $(1, y + \epsilon) = \sum_{i \in S} (1, y + e_i) - (|S| - 1)(1, y)$ and
  \begin{align*}\sum_{i \in S} -\psi(1, y + e_i) + (|S| - 1)\psi(1, y) &= \sum_{i \in S} -\mathsf{Q}(y + e_i) + (|S| - 1)\mathsf{Q}(y)\\ &= (|S| - 1)\mathsf{Q}(y) - \sum_{i \in S} (\mathsf{Q}(y) + 2y_i + 1)\\ &= - \mathsf{Q}(y) - \sum_{i \in S} 2y_i - |S|\end{align*}
    as well.

    The idea is that non-diagonal terms have to be added, while preserving the positive-definiteness of the form.

    For full rigor, let us begin with the $g = 1$ case, where this amounts to constructing  a Mumford-ample bundle on the Tate curve. In that situation, it is well-known that a diagonal form actually works (and it has to for this approach to work --- there aren't any other choices of rank-1 quadratic form). To be more precise, we use the continuous piecewise-affine function $\varphi_0$ on $Y \otimes \BR \simeq \BR$ given by \[\varphi_0|_{[n, n+1]}(t_1) = n^2 + n - (2n + 1)t_1.\]
    We also define the global affine-linear functions \[L_n(t_1) = n^2 + n - (2n + 1)t_1\] which extend  $\psi_0|_{[n, n+1]}(-)$ globally.
    These are just the secant lines at integer points to the concave continuous function $t_1 \mapsto -t_1^2$, so 
    \[\varphi_0|_{[n, n+1]}(t_1) = \min_{m \in \BZ} L_m(t_1).\]
    In fact, the minimum is attained by $L_m$ if and only if $y \in [m, m+1]$; and only by $L_m$ if $y \in (m, m+1)$. This makes $\varphi_0$ strictly convex even after extending to a piecewise-linear function $\psi_0$ on the cone over $\{1\} \times \BZ$.
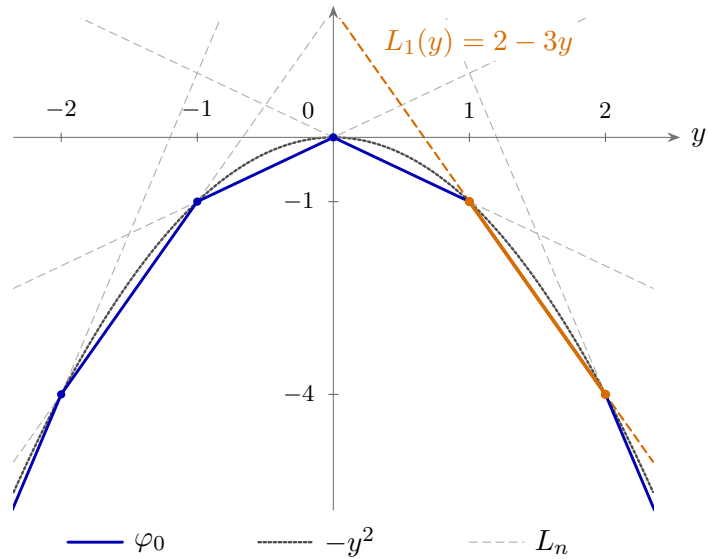
    \begin{center}
\begin{tikzpicture}[
  x=1.8cm,y=.85cm,
  font=\small,
  line cap=round,line join=round,
  >={Stealth[length=1.7mm]},
  affine/.style={black!30,densely dashed,line width=.45pt},
  parabola/.style={black!70,densely dotted,line width=.9pt},
  envelope/.style={blue!70!black,line width=1.1pt},
  chosen/.style={orange!85!black}
]
  \def\xmin{-2.35}
  \def\xmax{2.35}
  \def\ymin{-5.8}
  \def\ymax{1.85}

  \begin{scope}
    \clip (\xmin,\ymin) rectangle (\xmax,\ymax);

    \foreach \n in {-3,...,2}
      \draw[affine]
        (\xmin,{\n*\n+\n-(2*\n+1)*\xmin}) --
        (\xmax,{\n*\n+\n-(2*\n+1)*\xmax});

    \draw[chosen,densely dashed,line width=.8pt]
      (\xmin,{2-3*\xmin})--(\xmax,{2-3*\xmax});

    \draw[parabola]
      plot[domain=\xmin:\xmax,samples=120] (\x,{-\x*\x});

    \draw[envelope]
      (-3,-9)--(-2,-4)--(-1,-1)--(0,0)--
      (1,-1)--(2,-4)--(3,-9);

    \draw[chosen,line width=1.35pt] (1,-1)--(2,-4);
  \end{scope}

  \draw[black!55,->]
    (\xmin,0)--(2.55,0) node[right,text=black] {$y$};
  \draw[black!55,->] (0,\ymin)--(0,2.05);

  \foreach \n in {-2,-1,1,2} {
    \draw[black!60] (\n,-.065)--(\n,.065);
    \node[above=3pt,font=\scriptsize] at (\n,0) {$\n$};
  }
  \node[above left=3pt,font=\scriptsize] at (0,0) {$0$};

  \foreach \v in {-1,-4} {
    \draw[black!60] (-.035,\v)--(.035,\v);
    \node[left=3pt,font=\scriptsize] at (0,\v) {$\v$};
  }

  \foreach \n in {-2,...,2}
    \fill[blue!70!black] (\n,{-\n*\n}) circle (1.55pt);
  \foreach \n in {1,2}
    \fill[orange!85!black] (\n,{-\n*\n}) circle (1.75pt);

  \node[above right=2pt,chosen,fill=white,inner sep=1.5pt]
    at (.30,1.10) {$L_1(y)=2-3y$};

  \draw[envelope] (-1.95,-6.3)--(-1.55,-6.3)
    node[right,text=black] {$\varphi_0$};
  \draw[parabola] (-.55,-6.3)--(-.15,-6.3)
    node[right,text=black] {$-y^2$};
  \draw[affine] (1,-6.3)--(1.4,-6.3)
    node[right,text=black] {$L_n$};
\end{tikzpicture}
    \end{center}
    The condition \Cref{ampleness}(2) is obviously satisfied, because the numerator is linear while the denominator grows quadratically as the cone $\sigma$ varies in either direction. 

    In the higher rank $g$ case, we consider the piecewise-affine continuous function on $Y \otimes \BR$ given by
    \[\varphi(\underline{t}) := \sum c_{ii} \varphi_0(t_i) + \sum_{i < j} c_{ij}\varphi_0(t_i - t_j)\]
    for some choice of $c_{ij} \in \BZ_{>0}$. 
When $y \in Y = \BZ^g$, all the inputs to $\varphi_0$ are integers, so $\varphi$ is integer-valued on $Y$, in particular agreeing with the quadratic form $-\sum c_{ii}y_i^2 - \sum_{i < j} c_{ij}(y_i - y_j)^2$ on those values. In other words, $\psi$ extends linearly on cones to a well-defined element of $\mathbf{PL}(\Sigma)$ after extending linearly from $\sigma_{y, \pi}$ to $\widetilde{\sigma}_{y, \pi}$ for all $y \in Y$ and $\pi \in S_g$.
    We mostly just need to verify that the extension of $\psi$ is strictly convex as an element of $\mathbf{PL}(\Sigma)$; that is, $\varphi|_{\sigma_{y, \pi}}$ is not equal on $\sigma_{y, \pi}$ to the affine extensions of $\varphi|_{\sigma_{y', \pi'}}$ for all $(y', \pi') \neq (y, \pi)$, and in fact dominates it on the interior of $\sigma_{y, \pi}$. This is true as long as the coefficients $c_{ij}$ are all positive, as we now prove. The cone $\sigma_{y, \pi}$ is defined by various equalities $t_i \in [k_i, k_i + 1]$ and $t_i - t_j \in [k_{ij}, k_{ij} + 1]$, where the integers $k_i, k_{ij}$ depend on $y$, $\pi$, $i$, and $j$; so on $\sigma_{y, \pi}$, we have
    \[\varphi(\underline{t}) = \sum_{i} c_{ii} L_{k_i}(t_i) + \sum_{ij} c_{ij} L_{k_{ij}}(t_i - t_j).\]
    A point $\underline{t}' \in Y \otimes \BR$ not in $\sigma_{y, \pi}$ satisfies $t_i' \in [k_i', k_i' + 1]$ and $t_i' - t_j' \in [k_{ij}', k_{ij}' + 1]$, where at least one of the $k_i'$ or $k_{ij}'$ is not equal to the corresponding $k_i$ or $k_{ij}$. Since the $c_{ij} > 0$, it follows from this and the fact that $\min_{m \in \BZ} L_m(y)$ is attained by $L_m$ if and only if $y \in [m, m+1]$ that
    \[\sum_{i} c_{ii} L_{k_i}(t_i') + \sum_{ij} c_{ij} L_{k_{ij}}(t_i' - t_j') > \sum_{i} c_{ii} L_{k_i'}(t_i') + \sum_{ij} c_{ij} L_{k_{ij}'}(t_i' - t_j'),\]
    which is exactly the desired strict convexity property for the extension of $\varphi$ to $\mathbf{PL}(\Sigma)$. Finally, the condition (2) in \Cref{ampleness} is obviously true, because the denominator is a quadratic function of $\rho$ in the obvious sense and the numerator grows only linearly. So \Cref{ampleness} implies that $\Nt$ is ample on $\Pt$.\end{proof}

The line bundle $\Nt$ defined by $\psi$ after base change by $\pwr$ (that is, pullback to $\Pt$)  admits a $Y$-equivariant structure. Indeed, for $y, x \in Y$, 
    \begin{align*}(x^*\psi)(1, y) &= \psi(1, y + wx)\\
      &= -\mathsf{Q}(y + wx)\\
      &= -\mathsf{Q}(y) - w^2\mathsf{Q}(x) - w \langle y, x\rangle_{\mathsf{Q}} \\
      &= -\mathsf{Q}(y) - (1, y) \cdot (w^2\mathsf{Q}(x), wx)\\
&= \psi(1, y) - (1, y) \cdot (w^2\mathsf{Q}(x), wx);
    \end{align*}
    so for all $x \in Y$, there is an isomorphism
    \[x^*\Nt \simeq \Nt, \qquad \text{multiplication of sections by } q^{w\mathsf{Q}(x)}\underline{Z}^{w\phi(x)}\]
    which indeed defines the desired $Y$-equivariant structure. 
    From now on it only matters that we have this auxiliary $Y$-equivariant ample line bundle on $\Pt$ so that we may use it to produce relatively complete models in the sense of \Cref{ourdef} --- the exact choice of $\Nt$ will not matter. 

    We now produce the remaining datum $\Lt$ in a relatively complete model for $D_{B, \ua}$. The invertible sheaf $\widetilde{\mathscr{L}}$ on $\widetilde{P}$ is constructed using the piecewise-linear function $\psi \in \mathbf{PL}(\Sigma)$ which satisfies 
    \[\psi(1, wy) = -w\left(\sum_{i} b_{ii} \frac{y_i^2 - y_i}{2} +  \sum_{i < j} b_{ij} y_iy_j\right)\]
    and is linear on the cone over the decomposition as defined in \Cref{cube} except with $[0, 1]^g$ replaced with $[0, w]^g$ and the cubulation $Y + [0, 1]^g$ replaced with the coarser cubulation $wY + [0, w]^g \subset Y \otimes \BR$. The multiplication by $w$ was necessary to ensure that $\psi$ is given on the simplices $\widetilde{\sigma}_{y, \pi}$ by elements of $\BZ \otimes X$. Indeed, we have $\psi(1, wy) \in w\BZ$ for all $y \in \BZ$, so the various $\psi(1, y)$ are defined by affine-linear interpolation of these and can be deduced to be in $\BZ$.  

    The $Y$-action on $\widetilde{P}$ is induced by the action of $Y$ on $Y \otimes \BR$ given by $\underline{y} \cdot \underline{t} = \underline{t} + w\underline{y}$. After taking $\laurent$-fiber, this becomes the multiplication action of $\iota(y)$ on $\Gt_{\laurent} = \mathbf{G}_{m, \laurent}^g$. Since it comes from a $\psi \in \mathbf{PL}(\Sigma)$, the line bundle $\Lt$ also has a $\widetilde{G}$-equivariant structure, which is just the multiplication structure on $\mathcal{O}_{\widetilde{G}}$ when restricted to $X_{\BR_{\geq 0} \cdot (1, \underline{0})} = \widetilde{G}$, since $\psi(1, 0) = 0$ so the Cartier divisor defining $\mathscr{L}$ does not intersect $X_{(1, 0)}$.
    
    The line bundle $\widetilde{\mathscr{L}}$ enjoys the property of being $Y$-equivariant, because $y^*\widetilde{\mathscr{L}}$ is defined by $y^*\psi \in \mathbf{PL}(\Sigma)$, which satisfies
    \begin{align*}\psi(1, wx) - (y^*\psi)(1, wx)  &= \psi(1, wx) - \psi(1, w(x + y))\\ 
      &=  -wQ'(x) + wQ'(x + y)\\
      &= wQ'(y) + \langle wx, y \rangle_B.
    \end{align*}
    A $Y$-equivariant structure on $\Lt$ can therefore be given by the isomorphisms
    \[y^*\Lt \to \Lt\]
    which are multiplication by \[(q^{1/w})^{wQ'(y)}\underline{Z}^{\phi(y)} = q^{Q'(y)}\Z^{B(y)}.\] In fact, to actually get a relatively complete model for $D_{B, \ua}$ we can take a slightly more general $Y$-equivariant structure on $\Lt$, namely multiplication by
    \begin{equation}\label{equivstr}q^{Q'(y)}\Z^{B(y)} \prod_{i = 1}^g a_i(q)^{y_i}.\end{equation}
    \begin{lemma}\label{model-works}
     The tuple $(\Pt, \Lt, \widetilde{\mathscr{N}})$ is a relatively complete model for $(\Gt, i, \Lt) = D_{B, a}$. 
    \end{lemma}
    \begin{proof}
      Once we prove that $\Pt$ with the open embedding of $\Gt$ that we already defined satisfies the conditions (i)-(iii) of \Cref{rcm}, we are done; as the fact that $\Lt$ is compatible with the $\Lt$ of the degeneration datum is immediate from \Cref{equivstr}, and the requisite properties of $\Nt$ have already been proved above.
  We deal with the properties of a relatively complete model one at a time:

  \begin{enumerate} 
  \item[(i)] \ul{There exists an open $\widetilde{G}$-invariant subset $U \subset \widetilde{P}$ of finite type over $\BZ\llbracket q^{1/w}\rrbracket$ such that $\widetilde{P} = \bigcup_{y \in Y} y \cdot U$.}
    Simply take
\[U = \bigcup_{\substack{\pi \in S_g \\ y \in Y/wY}} X_{\widetilde{\sigma}_{y, \pi}} \times_{\BZ[q^{1/w}]} \pwr.\]
    This has the desired property because the action of $y' \in Y$ maps $X_{\widetilde{\sigma}_{y, \pi}}$ isomorphically to $X_{\widetilde{\sigma}_{y + wy', \pi}}$, and $X_\Sigma$ is covered by the affines $X_{\widetilde{\sigma}_{y, \pi}}$ which are finite type over $\BZ[q^{1/w}]$.  

  \item[(ii)] \ul{For all valuations $v$ (possibly not of rank 1) on the field of rational functions of $\widetilde{G}$ for which $v \geq 0$ on $\pwr$, $v$ has a center on $\widetilde{P}$ if and only if}
    \[\underline{\text{for all } \chi \in X \text{ there exists } y \in Y \text{ such that } v(i(\uy)^{\chi}\underline{Z}^\chi) \geq 0}.\]
    
    This form of the valuative criterion for properness essentially amounts to the fact that the infinite fan $\Sigma$ is complete over the fan for $\Spec \BZ[q^{1/w}]$. However, we cannot argue as in the classical setting (\cite[\S 2.4]{fulton}) just by looking at the vector for the valuations in $N \otimes \BR$ and noting that they have centers if and only if their corresponding vectors in $N \otimes \BR$ are in the support of the fan; this is because of the requirement that we also consider valuations of rank $\geq 1$.
    
    Note also that we are requiring an old-fashioned version of the properness condition from \cite{FC, mumford_compositio}. A weaker version is proven in \cite{lan} to be sufficient, but the point of this proof is that (as expected from \cite{mumford_compositio}) the strongest form is true for the relatively complete models constructed by toric geometry. 
    
    Let $v: K(\widetilde{G}) \to \Gamma$ be a valuation, where $\Gamma$ is a totally ordered abelian group. We prove each direction separately. Suppose the condition
    \begin{idea}\label{star}
for all $\chi \in X$, there exists $y  \in Y$ such that $v(i(y)^\chi\underline{Z}^\chi) \geq 0$
    \end{idea}
    holds. The element of $\Gamma$ given by $v(i(y)^\chi\underline{Z}^\chi)$ is, in our case, the same as
\[(\chi_1y_1 + \cdots \chi_gy_g)v(q) + \sum_{i = 1}^g\chi_iv(Z_i).\]
    If we could assume that $v$ was of rank $1$, then we could consider $v$ as a vector in $N \otimes \BR$ and easily conclude from \Cref{star} that $v \in |\Sigma|$  and therefore has a center on $\widetilde{P}$. The argument we end up doing essentially goes like this, but we have to be extra careful. First, since $\Gamma$ is a totally ordered abelian group, it is torsion-free; so $\Gamma \subset \Gamma \otimes \BQ$, which also admits the structure of a totally ordered $\BQ$-vector space. The valuation $v$ therefore induces an element of $\Hom((\BZ \oplus X)_\BQ, \Gamma_\BQ)$, which generalizes $(\BZ \oplus Y)_\BQ$. The object generalizing the cone
  \[\sigma_{\pi, y} = \{m \in  (\BZ \oplus Y)_\BQ : \langle m, n \rangle \geq 0 \text{ for all } n \in \sigma_{\pi, y}^\vee \cap (\BZ \oplus X)\} = (\sigma_{\pi, y}^\vee)^\vee \subset (\BZ \oplus Y)_\BQ\]
  (which in the rank-one valuation case would contain precisely the rank-1 valuations with a center on $X_{\sigma_{\pi, y}}$) is the cone
  \[\sigma_{\pi, y}' = \{m \in \Hom((\BZ \oplus X)_\BQ, \Gamma_\BQ) : \langle m, n \rangle \geq 0 \text{ for all } n \in \sigma_{\pi, y}^\vee \cap (\BZ \oplus X)\} \subset \Hom((\BZ + X)_\BQ, \Gamma_\BQ).\]
  If $v$ with $v \geq 0$ on $\pwr$ has associated vector (\emph{a priori} an element of $\Hom((\BZ \oplus X)_\BQ, \Gamma_\BQ)$) in $\sigma_{\pi, y}'$, then it has a center on $X_{\sigma_{\pi, y}} \times_{\BZ[q^{1/w}]} \pwr$ (as it is $\geq 0$ on the coordinate ring of $X_{\sigma_{\pi, y}} \times_{\BZ[q^{1/w}]} \pwr$, which is generated over $\pwr$ by the generators of $\sigma_{\pi, y}^\vee \cap (\BZ \oplus X)$). The condition \Cref{star} applied to the generators $Z_i$ for $X$ and their negations $Z_i^{-1}$  just says that $v \in \Hom((\BZ \oplus X)_\BQ, \Gamma_\BQ)$ satisfies some set of inequalities:
  \begin{idea} 
    For all $i = 1, \ldots, g$, there exist $y_i^+, y_i^- \in \BZ$ such that $y_i^+v(q) + v(Z_i) \geq 0$ and $y_i^-v(q) + v(Z_i) \leq 0$. 
  \end{idea}
  which together with the omnipresent condition $v(q) \geq 0$ define a set $\Omega' \subset \Hom((\BZ \oplus X)_\BQ, \Gamma_\BQ)$. 

  The set $\Omega'$ generalizes the set $\Omega \subset \Hom((\BZ \oplus X)_\BQ, \BQ)$ given by the same inequalities --- an example of such a set (which depends on the values of $y_i^{\pm}$) in the case $g = 1$ is pictured here, superimposed on \Cref{fig1} with horizontal axis standing for $v(q)$ and vertical axis standing for $v(Z)$.
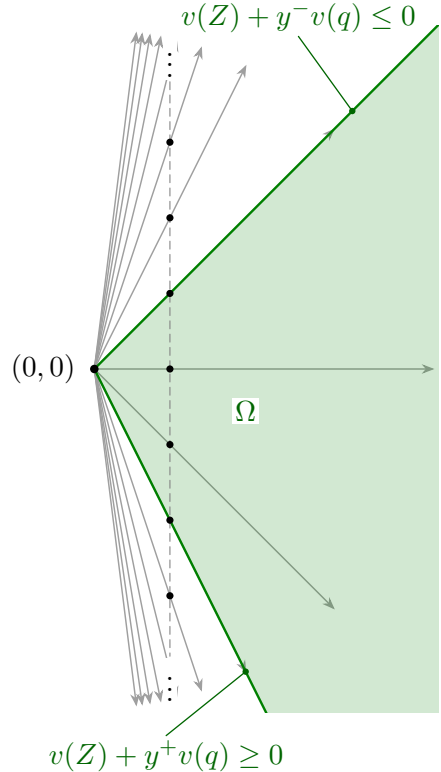
  \begin{center}\begin{tikzpicture}[
  x=1cm,y=-1cm,
  font=\small,
  line cap=round,
  line join=round,
  >={Stealth[length=1.7mm]},
  ray/.style={gray!75,line width=.55pt,->}
]
  \def\N{8}
  \def\R{4.5}

  \foreach \n in {-\N,...,\N} {
    \pgfmathsetmacro{\s}{\R/sqrt(1+\n*\n)}
    \coordinate (r\n) at (\s,{\s*\n});
  }

  \foreach \n in {-\N,...,\N}
    \draw[ray] (0,0)--(r\n);

  \def\sp{1}
  \def\sm{-2}
  \def\Xvis{4.55}
  \def\Yvis{4.55}
  \def\L{10}

  \begin{scope}
    \clip (0,-\Yvis) rectangle (\Xvis,\Yvis);
    \fill[green!50!black,opacity=.18]
      (0,0)--(\L,{-\sp*\L})--(\L,{-\sm*\L})--cycle;
  \end{scope}

  \begin{scope}
    \clip (0,-\Yvis) rectangle (\Xvis,\Yvis);
    \draw[green!50!black,line width=.9pt]
      (0,0)--(\L,{-\sp*\L});
    \draw[green!50!black,line width=.9pt]
      (0,0)--(\L,{-\sm*\L});
  \end{scope}

  \node[green!40!black,fill=white,inner sep=1pt]
    at (2.0,.55) {$\Omega$};

  \draw[gray,densely dashed] (1,-3.8)--(1,3.8);

  \foreach \n in {-3,...,3}
    \fill (1,\n) circle (1.35pt);

  \node[fill=white,inner sep=1pt]
    at (1,-4.13) {$\vdots$};
  \node[fill=white,inner sep=1pt]
    at (1,4.13) {$\vdots$};

  \fill (0,0) circle (1.6pt);
  \node[left=3pt] at (0,0) {$(0,0)$};
    \begin{scope}[
    raylabel/.style={
      text=green!40!black,
      fill=white,
      inner sep=1.5pt
    },
    leader/.style={green!40!black,line width=.5pt}
  ]
    \coordinate (upperRayPoint)
      at ({.75*\Xvis},{-.75*\sp*\Xvis});
    \coordinate (lowerRayPoint)
      at ({-.88*\Yvis/\sm},{.88*\Yvis});

    \node[raylabel,anchor=south] (upperRayLabel)
      at (2.7,-4.4) {$v(Z)+y^-v(q)\leq 0$};

    \node[raylabel,anchor=north] (lowerRayLabel)
      at (.95,4.85) {$v(Z)+y^+v(q)\geq 0$};

    \draw[leader] (upperRayLabel.south)--(upperRayPoint);
    \draw[leader] (lowerRayLabel.north)--(lowerRayPoint);

    \fill[green!40!black] (upperRayPoint) circle (1.2pt);
    \fill[green!40!black] (lowerRayPoint) circle (1.2pt);
  \end{scope}
\end{tikzpicture}
  \end{center}

  The property of $\Omega$ that is important in this context is that despite the infinitude of $\Sigma$, the cone $\Omega$ is covered by \emph{finitely many} cones $\sigma$ of $\Sigma$, which we call $\sigma_n$ for $n$ in a finite index set $T$. In other words, the finite system of linear inequalities defining $\Omega - \cup_{n \in T} \sigma_n$ has no solution in $N \otimes \BQ$. But those linear inequalities are exactly the same as the ones defining $\Omega' - \cup_{n \in T} \sigma_n'$ (recall the definition of $\sigma'$), so by the basic technique of elimination of quantifiers (see for example \cite[\S 3.1]{marker}) there is also no solution to that in $\Hom((\BZ \oplus X)_\BQ, \Gamma_\BQ)$, i.e. $\Omega'$ is covered by the $\sigma_n'$. This means that the valuation vector of $v$ belongs to some $\sigma_n'$, i.e. $v$ has a center on $\widetilde{P}$.

  Now let us address the converse. Suppose $v$ with $v \geq 0$ on $\pwr$ has a center on $\Pt$. Let $\sigma \in \Sigma$ be such that the center is on $X_{\sigma} \times_{\BZ[q^{1/w}]} \pwr$. Let $\chi \in X$. The coordinate ring of $X_{\sigma} \times_{\BZ[q^{1/w}]} \pwr$ contains $q^k\underline{Z}^\chi$, for all sufficiently large $k$. This is obvious from the general toric geometry of the situation: you just have to make $k$ large enough to make the relevant inequality true at the finitely many vertices of the polyhedral decomposition that $\Sigma$ is the cone over. Therefore,
  \[k \gg_{\sigma, \chi} 0 \implies q^k\underline{Z}^\chi \in \mathcal{O}(X_{\sigma}) \implies v(q^k\underline{Z}^\chi) \geq 0.\]

  In particular, by taking $\underline{y}$ such that $y_1\chi_1 + \cdots + y_g\chi_g$ is large enough, we know that $v(q^{y_1\chi_1 + \cdots + y_g\chi_g}\underline{Z}^\chi) \geq 0$, hence \Cref{star} is satisfied. 

\item[(iii)]\ul{The action of $\widetilde{G}$ and $Y$ on $\widetilde{P}$ extends the translation action of $\widetilde{G}$ and $Y$ on $\widetilde{G}$ given by the group law.}

  The part about the action of $\widetilde{G}$ is part of the data that comes with the toric scheme $X_\Sigma$. The part about the action of $Y$ follows from noting (by definition of the action of $Y$ in terms of the defining fan $\Sigma$) that after inverting $q$, all the $X_{\sigma_n}$ coincide and are equal to $\BG^g_{m, \laurent}$, where $\underline{y}$ acts by multiplication by $i(y) = q^{\underline{y}}$.
  (the actual (iii) in \Cref{rcm} is obvious from the positive-definiteness of the $\mathsf{Q}$ that we used). 
  \end{enumerate}
\end{proof}

\subsection{This degeneration datum uniformizes the line bundle $\mathscr{L}^{Q, \underline{a}}$}
In \Cref{intro2}, the rigidified line bundle $\mathscr{L}^{Q, \underline{a}}$ on $E_{w, \Tate}^g$ (the $g$-th fibered power of the Tate curve over $\laurent$) was defined in a fairly concrete way. In \Cref{KSDD}, we defined an object $M_{IS}(\Gt, i: Y \to \Gt_K, \Lt^{Q, \underline{a}})$ that depends on $Q, \underline{a}$. 
\begin{lemma}\label{unif} In fact, 
  $M_{IS}(\Gt, i: Y \hookrightarrow \Gt_K, \Lt^{Q, \underline{a}})$ and $\LL^{Q, \underline{a}}$ are isomorphic objects of $\DEG_{\ISgg}^{\spli}$ extending $G_w^g$. Consequently, the line bundles $\LL^{Q, \ua}$ account for all of the rigidified line bundles on Kuga--Sato degenerations. 
\end{lemma}
\begin{proof}
  First, the fact that $M(\Gt, i: Y \hookrightarrow \Gt_K)$ is $G_w^g$is obvious in the case $g=1$, as $G_w$ is by definition \cite{DR} given by applying $M$ to the $g = 1$ case of the degeneration datum we constructed, using the $g = 1$ case of the relatively complete model we constructed. For larger $g$, it follows from the standard fact that $M$ respects products, since products of relatively complete models (correctly fibered) are relatively complete models of products (in fact the one constructed in \Cref{KSDD} is a toric refinement of the $g$-th fibered power of the one for $g = 1$).

  Now we handle the line bundle part of the assertion. Let us begin with the case $\underline{a} = (1, \ldots, 1)$. Then both sides only depend on $Q$. The functor $M_{IS}$ respects tensor products of degeneration data, so by definition of the twisting procedure defining $\LL^Q$ in \Cref{twistQ} (and the obvious relatively complete models of the maps $m_{ij}$ and $\pi_i$), this case is reduced to the case of the quadratic form $Q(x) =  x^2$. The corresponding line bundle in either definition is symmetric and therefore determined by its $\laurent$-fiber. That fiber in the definition of \Cref{intro2} is nothing but $\mathcal{O}_{E_{w, \Tate}}([-1])$. In the definition of \Cref{KSDD}, it is a symmetric line bundle on $E_{w, \Tate}$ that induces a principal polarization $E_{w, \Tate} \to E_{w, \Tate}^\vee$ (since the homomorphism $\phi: Y \to X$ in the degeneration datum is an isomorphism). So there are in reality only $2$ or $4$ choices, corresponding to the $2$ or $4$ $2$-torsion sections of $E_{w, \Tate}$ over $\laurent$, depending on the parity of $w$. To see that this $2$-torsion section is exactly $[-1]$, one way is to use the Fourier expansion principle \Cref{definite} to produce a global section with Fourier expansion
  \[\sum_{\lambda \in \BZ} q^{\frac{\lambda^2 - \lambda}{2}}Z^\lambda\]
  and note that it vanishes along the section $Z = -1$. 
  
  It remains to show that the $\Pic^0$-twisting procedure, which on both sides depends on $\ua$, is the same. As before, this can be reduced to the $g = 1$ case, where it amounts to showing that the degeneration datum corresponding to the cocycle $y \mapsto a^y$ where $a \in \laurent^\times$ uniformizes $\mathcal{O}_{E_{w, \Tate}}([-a^{-1}] - [-1])$. Again, we wish to do this by explicitly comparing Fourier expansions of global sections by \Cref{definite}, so it is more convenient to actually reduce to the case of twists of $\LL^{Q}$ with $Q = x^2$. In \Cref{intro2}, this meaning of $\LL^{Q, a}$ is precisely $\mathcal{O}_{E_{w, \Tate}}([-a^{-1}])$, i.e., $a^*\mathcal{O}_{E_{w, \Tate}}([-1])$. On the other hand, by \Cref{definite}, the version defined in \Cref{KSDD} has a global section given by
  \[\sum_{\lambda \in \BZ} q^{\frac{\lambda^2 - \lambda}{2}}a^\lambda Z^{\lambda}\]
which vanishes on the relative Cartier divisor $[-a^{-1}]$. 
\end{proof}
\section{Computation of cohomology over $\laurent$}\label{nonequiv}
This section concerns the proof of \Cref{mainthm}. By \Cref{unif}, the line bundle $\LL^{Q, \ua}$ whose cohomology we wish to calculate is uniformized by the relatively complete model $(\Pt, \Lt^{Q, \ua}, \Nt)$ constructed in \Cref{KSDD}. For brevity, we treat $Q$ and $\ua$ as fixed and use the notation $\LL$ to denote $\LL^{Q, \ua}$.  Our technique of calculation of $H^\bullet(E_{w, \Tate}^g, \LL^{Q, \ua})$ is correspondingly to consider the double complex
\[E_{0}^{p, q} = C^p(Y, \check{C}^q(\mathcal{U}_\Sigma, \Lt^\wedge)),\]
where
\begin{enumerate}
\item The first $C^p$ denotes group cohomology cochains with respect to the action of $Y$.
\item $\mathcal{U}_\Sigma$ denotes the cover of $\Pt^\wedge$ given by the $(q^{1/w})$-adic formal completions of the basic toric affines $X_{\sigma} \times_{\BZ[q^{1/w}]} \pwr$ in $\Pt = X_\Sigma \times_{\BZ[q^{1/w}]} \pwr$, ranging over the cones $\sigma \in \Sigma$. 
  \item The second $\check{C}^q$ denotes the \v{C}ech complex for sheaves of abelian groups on $\Pt^\wedge$. 
\end{enumerate}

The cohomology of the total complex $H^\bullet(\mathrm{Tot}(E_0^{\bullet, \bullet}))$ can be computed by either one of the spectral sequences $^{I}E_{\bullet}^{\bullet, \bullet}$ or $^{II}E_{\bullet}^{\bullet, \bullet}$, where the first one has ``horizontal'' boundary maps in the $E_0$ page,
\[d: {}^IE_0^{p, q} \to {}^IE_0^{p + 1, q}\]
which are just the boundary maps for group cohomology cochains; and the second one has ``vertical'' boundary maps in the $E_0$ page,
\[d: {}^{II}E_0^{p, q} \to {}^{II}E_0^{p, q + 1}\]
which are induced by the boundary maps for \v{C}ech cochains.

We will first prove that the spectral sequence $^I E^{\bullet, \bullet}_\bullet$ degenerates at $E_2$, where it calculates the $\pwr$-modules $H^\bullet(P_{w, g}, \mathscr{L})$. By flatness of localization, the cohomology of the $\laurent$-fiber (the computation of which is the goal of all of this) is the cohomology of the total complex of $E_0^{\bullet, \bullet} \otimes \laurent$, which is computed by $^{II}E_{\bullet}^{\bullet, \bullet} \otimes \laurent$. We will then deduce \Cref{mainthm} by proving that this spectral sequence also degenerates at $E_2$ and converges to $H^i(Y, \mathcal{O}(\adictorus))$ where $y \in Y$ acts by \Cref{finalaction}.

Here the notation is that $P_{w, g}$ is the algebrization of $\Pt^\wedge/Y = \mathfrak{P}$ involved in the definition of Mumford's uniformization functor, \Cref{mumford-construction}. Recall that $P_{w, g}$ is a proper $\pwr$-scheme, and it contains the Kuga--Sato degeneration $G_w^g$ as an open subscheme (they are in fact both equal to $E_{w, \Tate}^g$ over $\laurent$). 

\subsection{Spectral sequence, first direction}
We are meant to prove:
\begin{proposition}\label{ss1}
The spectral sequence $^IE^{\bullet, \bullet}_{\bullet}$ degenerates at $E_2$, where it converges to $H^i(P_{w, g}, \mathscr{L})$.
\end{proposition}
\begin{proof}
  Since $P_{w, g}$ is a proper $\pwr$-scheme, formal GAGA (for instance \cite[Proposition (5.1.2)]{EGAIII}) says that
  \[H^i(P_{w, g}, \mathscr{L}) \simeq H^i(\widehat{P}_{w, g}, \LL^\wedge).\]
  By construction, $\widehat{P}_{w, g}$ is the quotient by the action of $Y$ of $\Pt^\wedge$, which is called $\mathfrak{P}_{w, g}$. And $\LL^\wedge$ is the descent of $\Lt^\wedge$ (from \Cref{KSDD}) to that quotient. So our goal is to prove that $^IE^{\bullet, \bullet}_{\bullet}$ degenerates at $E_2$, where it computes $H^i(\mathfrak{P}_{w, g}, \Lt^\wedge/Y)$.

  The action of $Y$ on the set of cones $\widetilde{\sigma}_{y, \pi}$ in $\Sigma$ is free. Letting $I = Y \times S_g$ be the index set for those cones (so $Y$ acts on $I$ by $x \cdot (y, \pi) = (y + wx, \pi)$), it follows that the diagonal action of  $Y$  on $I^{q + 1}$ is free for all $q \geq 0$. Let $\mathcal{F}_q$ be a fundamental set for the action of $Y$ on the subset of $I^{q+1}$ consisting of $\underline{i} \in I^{q+1}$ such that $\cap_{k = 0}^q \sigma_{i_k} \neq \emptyset$ (this is the same as requiring $\cap_{k = 0}^q \widehat{X}_{\widetilde{\sigma}_{i_k}} \neq \emptyset$, since it means that $\cap_k \widetilde{\sigma}_{i_k} = \{0\}$ so the corresponding affines in $X_\Sigma$ only intersect over $\laurent$). The reason why we restricted to $\mathcal{F}_q$ is that it is a finite set.  In this notation,
  \[\check{C}^q(\mathcal{U}_\Sigma, \Lt^\wedge) = \prod_{y \in Y} \prod_{\underline{i} \in \mathcal{F}_q}\Gamma(\cap_{k=0}^q \widehat{X}_{y \cdot (i_k)}, \Lt^\wedge),\]
  where $X_{i}$ is shorthand for $X_{\sigma_{i}}$.
  The action of $Y$ on the \v{C}ech complex is by the product of the collection of isomorphisms coming from the $Y$-equivariant structure on $\Lt$,
  \[y^*: \Gamma(\cap_{k = 0}^q \widehat{X}_{i_k}, \Lt^\wedge) \to \Gamma(\cap_{k = 0}^q \widehat{X}_{(-y) \cdot i_k}, \Lt^\wedge).\]
  In fact the intersections
  $\cap_{k=0}^p X_{i_k}$ are themselves affines in $\widetilde{P}$ coming from the cone $\cap_{k=0}^p \widetilde{\sigma}_{i_k}$. The line bundle $\widetilde{\mathscr{L}}$ is trivial on those affines (it is locally principal given on affines by divisor of the element of $\BZ \otimes X$ that defines the piecewise-linear function $\psi$ on the cone corresponding to the affine), and the pullback map may be concretely described by the isomorphism
  \[y^*: \Gamma(\cap_{k = 0}^q \widehat{X}_{i_k}, \widehat{\mathcal{O}}) \to \Gamma(\cap_{k = 0}^q \widehat{X}_{(-y) \cdot i_k}, \widehat{\mathcal{O}})\]
  given by
  \[f(\underline{Z}) \mapsto f(q^{y_1}Z_1, \ldots, q^{y_g}Z_g).\]
  In other words,
\[\check{C}^q(\mathcal{U}_\Sigma, \Lt^\wedge) = \mathrm{coInd}_{1}^Y \left(\prod_{\underline{i} \in \mathcal{F}_q}\Gamma(\cap_{k=0}^q \widehat{X}_{i_k}, \Lt^\wedge)\right).\]
The next page of the spectral sequence is therefore
\begin{align*}^IE_{1}^{p, q} &= H^p\left(Y, \mathrm{coInd}_{1}^Y \left(\prod_{\underline{i} \in \mathcal{F}_q}\Gamma(\cap_{k=0}^q \widehat{X}_{i_k}, \Lt^\wedge)\right)\right)\\
  &= \begin{cases}\prod_{\underline{i} \in \mathcal{F}_q}\Gamma(\cap_{k=0}^q \widehat{X}_{i_k}, \Lt^\wedge), \qquad \text{ if } p = 0, \\
    0, \text{ if } p \geq 1
    \end{cases}
\end{align*}
by Shapiro's lemma. The vertical morphisms on the $E_1$ page are induced by taking $Y$-fixed points from the \v{C}ech boundary maps
\[d: \prod_y \prod_{\underline{i} \in \mathcal{F}_q} \Gamma(\cap_{k = 0}^q \widehat{X}_{i_k}, \Lt^\wedge) \to \prod_y \prod_{\underline{i} \in \mathcal{F}_{q+1}} \Gamma(\cap_{k = 0}^{q+1} \widehat{X}_{i_k}, \Lt^\wedge)\]
given by
\[d((s_{\underline{i}})_{\underline{i} \in I^{q+1}})_{\underline{j}} = \sum_{k = 0}^{q+1} (-1)^k s_{i_0, \ldots, \hat{i}_k, \ldots, i_{q+1}|_{\cap_k \widehat{X}_{i_k}}}\]
for all $j \in I^{q+2}$. For
\[(s_{\underline{i} \in \mathcal{F}_q}) \in \prod_{\underline{i}} \Gamma(\cap_{k = 0}^q \widehat{X}_{i_k}, \widehat{\mathscr{L}}) \cong H^0\left(Y, \mathrm{coInd}_{1}^Y \left(\prod_{\underline{i} \in \mathfrak{F}_q}\Gamma(\cap_{k=0}^q \widehat{X}_{i_k}, \Lt^\wedge\right)\right),\]
its boundary
\[d((s_{\underline{i} \in \mathcal{F}_q})) \in {^IE_{1}^{q+1, 0}}\]
is given for all $\underline{j} \in \mathcal{F}_{q+1}$ by
\[d((s_{\underline{i} \in \mathcal{F}_q}))_{\underline{j}} = \sum_{k=0}^{q+1} (-1)^k \widetilde{s}_{j_0, \ldots, \hat{j}_{k}, \ldots, j_{q+1}},\]
where $\widetilde{s}_{j_0, \ldots, \hat{j}_{k}, \ldots, j_{q+1}}$ is accounting for the possibility that $(j_0, \ldots, \hat{j}_{k}, \ldots, j_{q+1}) \in I^{q+1}$ might not be in $\mathcal{F}_q$: so it refers to the element of the $Y$-fixed points of the coinduced module corresponding to $s$. In other words, if $y \in Y$ is the unique element such that
\[(j_0+y, \ldots, \hat{j}_{k}, \ldots, j_{q+1} + y) \in \mathcal{F},\]
then
\[\widetilde{s}_{j_0, \ldots, \hat{j}_{k}, \ldots, j_{q+1}} := y^*(s_{j_0+y, \ldots, \hat{j}_{k}, \ldots, j_{q+1} + y}).\]
The page $^IE_2^{\bullet, \bullet}$ has $^IE_2^{p, q} = 0$ for $p \geq 1$, and 
\[^IE_2^{0, q} = H^q\left(\prod_{\underline{i} \in \mathcal{F}^\bullet} \Gamma(\cap_{k = 0}^\bullet \widehat{X}_{i_{k}}, \Lt^\wedge), d\right)\]
for $d$ as we have just defined. This is essentially the \v{C}ech cohomology of $\Lt^\wedge$, except everything is taken up to the $Y$-action. In fact, this is exactly the \v{C}ech cohomology that computes the cohomology of $P_{w, g}$ with coefficients in $\widehat{\mathscr{L}}$, which is what remains to prove. We will do it by proving it for $P_\ell := \widehat{P}/Y \times_A A/I^\ell$ for all $\ell \geq 0$  (this is precisely what is done below in \Cref{qt-cech}) and carefully passing to the inverse limit. Indeed, since $P_{w, g}$ is proper over $\pwr$, by the theorem on formal functions \cite[Corollaire (4.1.7)]{EGAIII} the cohomology we want is 
\begin{align*}H^q(\mathfrak{P}_{w, g}, \widehat{\mathscr{L}}) &\cong \varprojlim_\ell H^q(P_\ell, \Lt_\ell)\\
  &\overset{\text{by \Cref{qt-cech}}}{\cong} \varprojlim_\ell H^q\left(\prod_{\underline{i} \in \mathcal{F}_q} \Gamma(\cap_{k = 0}^q X_{i_{k}} \times_A A/I^\ell, \Lt_\ell), d_\ell\right)\\ 
  &\cong H^q\left(\prod_{\underline{i} \in \mathcal{F}_q} \Gamma(\cap_{k = 0}^q \widehat{X}_{i_{k}}, \widehat{\mathscr{L}}), d\right)\\ 
\end{align*}
To pass this projective limit past the cohomology, we invoke \cite[Ch. 0, Proposition 13.2.3]{EGAIII}, which says that this is allowed because (1) the inverse system of complexes involved satisfies the Mittag--Leffler condition, and (2) the $H^q$'s satisfy the Mittag--Leffler condition (itself following from \cite[Corollaire (4.1.7)]{EGAIII}). In the last line of the above computation, this was done, and the finiteness of $\mathcal{F}_q$ was also used to pass the inverse limit past the product. 
\end{proof}

\begin{lemma}\label{qt-cech}
  Let $\ell \geq 1$. The complex of $R/I^\ell$-modules
  \[\left(\prod_{\underline{i} \in \mathcal{F}_p} \Gamma(X_{\tilde{\sigma}_{i_0} \cap \cdots \cap \tilde{\sigma}_{i_p}} \times_R R/I^\ell, \widetilde{\mathscr{L}}), d\right)\]
  defined in the analogous way to the previous proof computes the cohomology of $P_\ell := (\widetilde{P} \times_R R/I^\ell)/Y$ with coefficients in the sheaf descended from $\widetilde{\mathscr{L}}$. 
\end{lemma}
\begin{proof}
  Since one is now working with schemes, we are allowed to use the \'etale site to compute Zariski cohomology of quasicoherent sheaves. This is useful because unfortunately the action of $Y$ on $\widetilde{P} \times_R R/I^\ell$ is not always properly discontinuous (if $w = 1$ in particular this is a problem), so the opens $X_{\widetilde{\sigma}}\times_R R/I^\ell$ do not map isomorphically onto opens of the $Y$-quotient. However, 
  \[\mathcal{U}_\ell = \left\{X_{\widetilde{\sigma}_i} \times_R R/I^\ell: i \in \mathcal{F}_0\right\}\]
  is an \'etale cover of $P_\ell$ (\cite[Theorem 3.10]{mumford_compositio}). Plus, the pairwise fibered products over $P_\ell$ of the various objects of the cover are disjoint unions of affine toric pieces,
  \[X_{\widetilde{\sigma}} \times_{P_\ell} X_{\widetilde{\tau}} = \bigsqcup_{y \in Y} X_{\widetilde{\sigma} \cap y \cdot \widetilde{\tau}},\]
  so this is an acyclic cover. Finally, the \v{C}ech cohomology of $\mathscr{L}$ with respect to this \'etale cover is by definition the same as the cohomology of the complex involved in the statement of this lemma. Therefore, that complex computes
  \[\check{H}^\bullet_{\text{\'et}}(\mathcal{U}_\ell, \mathscr{L}_\ell),\]
  which is the same as
  \[H^\bullet_{\text{\'et}}(P_\ell, \mathscr{L}_\ell)\]
  since $P_\ell$ is a projective $A/I^\ell$-scheme and $\mathcal{U}_\ell$ is acyclic. But since $\mathscr{L}_\ell$ is quasicoherent, this is exactly the same as the Zariski cohomology of $\mathscr{L}_\ell$, as desired. 
\end{proof}

\subsection{Spectral sequence, second direction after taking generic fiber}

The previous subsection shows that $^{I}E_{\bullet}^{\bullet, \bullet}$ and $^{II}E_{\bullet}^{\bullet, \bullet}$ both compute the cohomology of $P_{w, g}$ with coefficients in $\mathscr{L}$. 

The thing we actually want to compute is the cohomology of the restriction to $E_{w, \Tate}^g$
  \begin{align*}H^{\bullet}(E_{w, \Tate}^g, \mathscr{L}|_{E_{w, \Tate}^g}) &= H^\bullet(P_{w, g} \times_{\pwr} \laurent, \mathscr{L})\\ &\cong H^\bullet(P_{w, g}, \mathscr{L}) \otimes \laurent \\ &\cong H^\bullet(\mathfrak{P}_{w, g}, \LL^\wedge) \otimes \laurent \end{align*}
  where the last line is by \cite[Proposition (5.1.2)]{EGAIII}. We proved in \Cref{ss1} that this is the same as \[^IE_2^{0, \bullet} \otimes \laurent = {}^IE_\infty^{0, \bullet} \otimes \laurent.\]
  By flatness of localization and finiteness of the products involved in the \v{C}ech complex in the proof of \Cref{ss1} whose cohomology is equal to $^IE_2^{0, \bullet}$, we deduce that $H^{\bullet}(E_{w, \Tate}^g, \mathscr{L}|_{E_{w, \Tate}^g})$ is converged to by the spectral sequence starting at
  \[^{I, \BZ((q^{1/w}))}E_1^{p, q} =  H^p\left(Y, \mathrm{coInd}_1^Y \left(\prod_{\underline{i} \in \mathcal{F}_q} \Gamma(\cap_{k=0}^q \widehat{X}_{i_k}, \widehat{\mathscr{L}}) \otimes_{\pwr} \laurent\right)\right),\]
  and therefore equals the cohomology of the total complex of
  \[^{I, \BZ((q^{1/w}))}E_0^{p, q} = C^p\left(Y, \mathrm{coInd}_1^Y \left(\prod_{\underline{i} \in \mathcal{F}_q} \Gamma(\cap_{k=0}^q \widehat{X}_{i_k}, \widehat{\mathscr{L}}) \otimes_{\pwr} \laurent\right)\right).\]
  A subtle point already alluded to in \Cref{invertq} is that due to the fact that tensor products do not commute with infinite products, this is not the same as $^IE_0^{p, q} \otimes \laurent$. This ``unbounded denominators'' issue is why we had to be so careful about choosing a finite fundamental domain for the action of $Y$ on our infinite cover. Our goal being to compute the cohomology of the total complex of this double complex in some different way, we take the same double complex but view it as the $E_0$ page of the vertical-first double complex spectral sequence with identical $0$-th page
  \[^{II, \laurent}E_0^{p, q} =C^p\left(Y, \mathrm{coInd}_1^Y \left(\prod_{\underline{i} \in \mathcal{F}_q} \Gamma(\cap_{k=0}^q \widehat{X}_{i_k}, \Lt^\wedge) \otimes_{\pwr} \laurent\right)\right).\]
  The proof of \Cref{mainthm} is therefore reduced to
  \begin{proposition}
    The spectral sequence $^{II, \laurent}E_\bullet^{\bullet, \bullet}$ degenerates at $E_2$ and converges to the group cohomology of $Y$ with coefficients in $\mathcal{O}(\adictorus)$. 
  \end{proposition}
  Taking cohomology in the vertical direction certainly commutes with group cohomology cochains, so
  \[^{II, \BZ((q^{1/w}))}E_1^{p, q} = C^p\left(Y, H^q\left(\prod_{\overline{i} \in I^{\bullet + 1}} \Gamma(\cap_{k = 0}^\bullet \widehat{X}_{i_k}, \Lt^\wedge) \otimes_{\pwr} \laurent\right)\right).\]

  All we need to do is calculate the $H^q$'s on the inside to recover the desired result. 
  
  It therefore suffices to prove the following two facts: the identification of the inner $H^0$ with $\mathcal{O}(\adictorus)$ (in \Cref{H0})  and the vanishing of the inner $H^q$ for $q \geq 1$ (in \Cref{Hq}). 

  \begin{proposition}\label{H0}
    \[H^0\left(\prod_{\overline{i} \in I^{\bullet + 1}} \Gamma(\cap_{k = 0}^\bullet \widehat{X}_{i_k}, \Lt^\wedge) \otimes_{\pwr} \laurent\right)\]
    is canonically isomorphic to $\mathcal{O}(\adictorus)$.  
  \end{proposition}
  \begin{proof}
    Let $i = (y, \pi) \in I$ and let the cone $\widetilde{\sigma} \in \Sigma$ be the cone over  $\sigma_i \subset Y \otimes \BR$. The coordinate ring of the affine patch $X_i \times_{\BZ[q^{1/w}]} \pwr \subset \Pt$ is $\pwr[\widetilde{\sigma}^\vee \cap (\BZ \oplus X)]$, so
    \[\Gamma(\widehat{X}_{i}, \widehat{\mathcal{O}}) = \Gamma(X_i, \mathcal{O})^\wedge\]
    is just the set of formal power series $f(\underline{Z}) = \sum a_\chi \underline{Z}^\chi$ with $a_\chi \in \laurent$ with the property that $f$ converges $q$-adically for all input $\underline{Z}$ such that the vector $(v_q(Z_i))_{i = 1}^g$ is in $\sigma$. Being in $H^0$ of this \v{C}ech complex therefore means (after multiplying by transition functions, which are all just monomials and therefore do not affect the convergence properties at all) that $f$ converges $q$-adically regardless of the vector $(v_q(Z_i))_i$, since the $\sigma$'s tile $Y \otimes \BR$. This is the same as the assertion $f \in \mathcal{O}(\adictorus)$ --- see \Cref{veryconvergent}. 
  \end{proof}
  \begin{proposition}\label{Hq}

    \[H^p\left(\prod_{\overline{i} \in I^{\bullet + 1}} \Gamma(\cap_{k = 0}^\bullet \widehat{X}_{i_k}, \Lt^\wedge) \otimes_{\pwr} \laurent\right) = 0\]
    for all $p \geq 1$.
  \end{proposition}
  \begin{proof}
    At this stage it is convenient to outsource the bulk of the proof to the adic space machinery. The idea is that we need a language to study the ``generic fiber'' of the formal $\Spf(\pwr)$-scheme $\Pt^\wedge$, which is exactly what the adic space $\adictorus$ is. Using Huber's functor $t$ from locally Noetherian formal schemes to adic spaces \cite{hubergen, huberbook, regensburg}, we can construct the adic space $t(\Pt^\wedge)$. It is constructed by patching together affinoids constructed from affines in $\Pt^\wedge$. More specifically, from the affine
    \[\mathrm{Spf} (\pwr[\widetilde{\sigma}^\vee \cap X]^\wedge) \subset \Pt^\wedge\]
    one gets the affinoid
    \[\mathrm{Spa}(\pwr[\widetilde{\sigma}^\vee \cap X]^\wedge, \pwr[\widetilde{\sigma}^\vee \cap X]^\wedge)\]
    whose analytic points (the points corresponding to valuations that are nonzero on an ideal of definition) are 
    \[t(\mathrm{Spf} (\pwr[\widetilde{\sigma}^\vee \cap X]^\wedge))_a = \mathrm{Spa}(\pwr [\widetilde{\sigma}^\vee \cap X]^\wedge \otimes_{\pwr} \laurent, \pwr[\widetilde{\sigma}^\vee \cap X]^\wedge).\]
    Gluing all of these together to get $t(\Pt^\wedge)_a$ yields a copy of $\adictorus$ covered by affinoids defined by inequalities corresponding to the inequalities defining the various polyhedra $\sigma \subset Y \otimes \BR$, and the complex whose cohomology we want to compute is manifestly the \v{C}ech complex with respect to this cover for a line bundle on this analytic adic space. This cover being acyclic for the usual reasons, we are done because \emph{any} line bundle on $\adictorus$ has trivial higher cohomology, by Cartan's Theorem B for quasi-Stein adic spaces \cite[Theorem 2.6.5]{KL2}.  
  \end{proof}

  \section{Equivariant version}\label{equivariant}
  Let $\Gamma$ be a subgroup of $\GL(Y)$ that preserves the quadratic form $Q$ defining the line bundle $\mathscr{L}$.  The abelian scheme $E_{w, \Tate}^g$ admits an action of the abstract group $\Gamma$ as a $\laurent$-scheme. The action of $\gamma = (\gamma_{ij}) \in \Gamma$ is just defined by
  \[(z_1, \ldots, z_g) \mapsto \left(\prod_{j = 1}^g z_j^{\gamma_{ij}}\right)_{i = 1}^g.\] We now restrict to a very specific case that is chosen so that $\LL$ has a natural $\Gamma$-equivariant structure. We retain the arbitrary quadratic form $Q$ determined by the arbitrary $g \times g$ symmetric $\BZ$-entried matrix $B = (b_{ij})$. However, we now impose a very harsh restriction on $a = (a_1, \ldots, a_g)\in \Hom(Y, \laurent^\times)$:
  \begin{definition}
    From now on, $\LL^{Q, \underline{a}}$ is taken to be one of the two to four symmetric rigidified line bundles
    \[(\LL^{x^2, a}) \otimes Q, a \in \{\pm 1, \pm q^{1/2}\} \cap \laurent,\]
    that is,
    \[\left(\bigotimes_{i = 1}^g \pi_i^*\LL^a\right)^{b_{ii}} \otimes \bigotimes_{i < j} \left(m_{ij}^* \LL^a \otimes \pi_i^*(\LL^a)^{-1} \otimes \pi_j^*(\LL^a)^{-1}\right)^{b_{ij}}\]
    where $\LL^a = \mathcal{O}([-a^{-1}])$ with the implicit rigidification (and normalized symmetric structure) and $[-a^{-1}]$ is the relative Cartier divisor in $E_{w, \Tate}$ defined via \cite[Ch. VII, \S1.2, (d)]{DR} as in \Cref{intro2}.  
  \end{definition}

  \begin{definition}
    The line bundle $\LL^{Q, a}$ admits a $\Gamma$-equivariant structure as a consequence of the theorem of the cube, which ensures that for all $\gamma \in \Gamma$, there \emph{exists} an isomorphism
    \[\gamma^*(\LL \otimes Q) \simeq \LL \otimes (\gamma^\top Q \gamma) = \LL \otimes Q.\]
    We fix the $\Gamma$-equivariant structure coming from the isomorphisms that are rigidified, i.e. that are isomorphisms of cubical line bundles. 
  \end{definition}
  
  It will be useful for our proofs to have access to a toric model for the $\Gamma$-action, that is, a map of $\pwr$-schemes $\widetilde{\gamma}: \Pt_1 \to \Pt_2$ between relatively complete models for the Kuga--Sato degeneration $G_w^g$ such that $M(\widetilde{\gamma}) = \gamma$. We will need also a model that uniformizes the $\Gamma$-equivariant structure on $\LL$ using $M_{IS}$. 
\subsection{Relatively complete model for $\Gamma$ acting on the scheme $E_{w, \Tate}^g$}
It is a minor inconvenience that we cannot use the relatively complete model $\Pt$ for $E_{w, \Tate}^g$ from before, because a putative toric model $\gamma: \Pt \to \Pt$ defined by a map $Y \otimes \BR \to Y \otimes \BR$ would have to be induced by $\gamma: Y \to Y$. But the required map on $(\BZ \oplus Y) \otimes \BR$ does \emph{not} necessarily preserve the fan $\Sigma$.

  One way to handle this is to define two \emph{different} relatively complete models $\Pt_1$ and $\Pt_2$ for $E_{w, \Tate}^g$. If all we want to do is uniformize the map $\gamma: E_{w, \Tate}^g \to E_{w, \Tate}^g$, then we may take $\Pt_2$ to be the same as $\Pt$ as defined in \Cref{KSDD}, and, in order to make sure that we can define a toric model for $\gamma$, we define $\Pt_1$ from the fan \[\gamma^{-1}\Sigma = \text{the cone in } (\BZ + Y) \otimes \BR \text{ over } \gamma^{-1}(\text{the triangulation used to define } \Sigma).\]
  (on the left hand side one can interpret $\gamma$ as acting on $(\BZ \oplus Y) \otimes \BR$ by the linear map $(1, \gamma)$). An example of the two resulting triangulations of $Y \otimes \BR$ is pictured in \Cref{triangulation-diag}.
  
  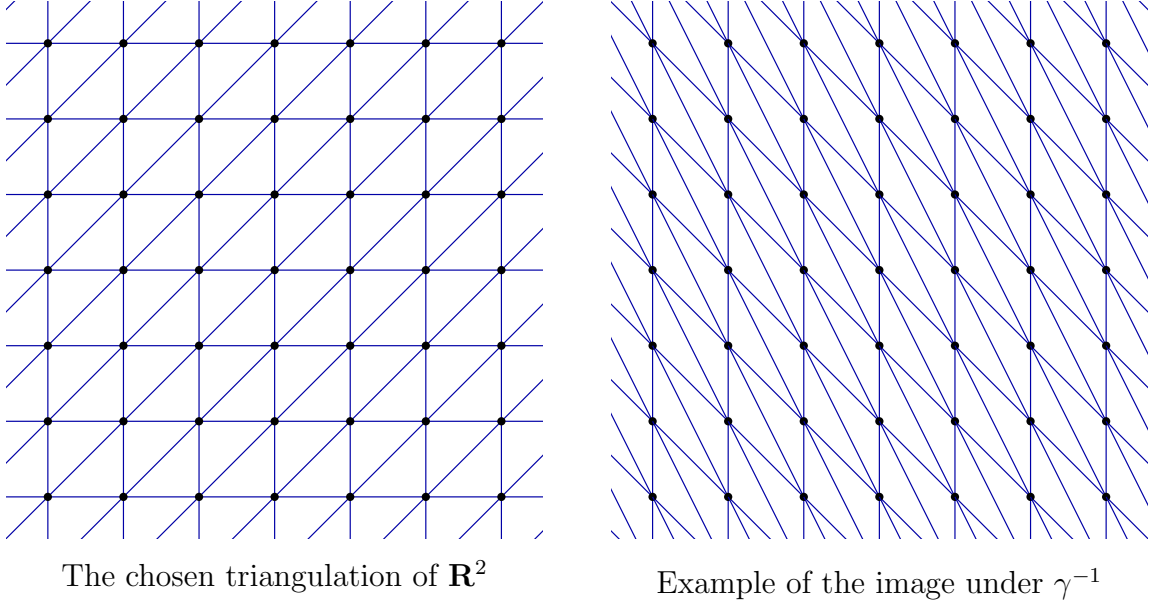
\begin{figure}
    \caption{The triangulations of $Y \otimes \BR$ inducing the fans that define $\Pt_1$ and $\Pt_2$ for a certain $\gamma$.}\label{triangulation-diag}
  \begin{center}
\begin{tikzpicture}[
  edge/.style={line width=.45pt, draw=blue!65!black},
  vertex/.style={circle, fill=black, inner sep=1.1pt}
]

\begin{scope}[transparency group, scope fading=soft edge]
  \clip (-3.55,-3.55) rectangle (3.55,3.55);
  \foreach \i in {-4,...,3} {
    \foreach \j in {-4,...,3} {
      \draw[edge] (\i,\j) -- ({\i+1},\j);
      \draw[edge] (\i,\j) -- (\i,{\j+1});
      \draw[edge] (\i,\j) -- ({\i+1},{\j+1});
    }
  }
  \draw[edge] (-4,4)--(4,4);
  \draw[edge] (4,-4)--(4,4);
  \foreach \i in {-4,...,4}
    \foreach \j in {-4,...,4}
      \node[vertex] at (\i,\j) {};
\end{scope}
\node at (0,-4.05)
  {The chosen triangulation of $\BR^{2}$};

\begin{scope}[xshift=8cm, transparency group, scope fading=soft edge]
  \clip (-3.55,-3.55) rectangle (3.55,3.55);
  \foreach \i in {-8,...,8} {
    \foreach \j in {-8,...,8} {
      \draw[edge] (\i,\j) -- ({\i+1},{\j-1});
      \draw[edge] (\i,\j) -- ({\i-1},{\j+2});
      \draw[edge] (\i,\j) -- (\i,{\j+1});
      \node[vertex] at (\i,\j) {};
    }
  }
\end{scope}
\node[align=center] at (8,-4.15)
  {Example of the image under $\gamma^{-1}$}; 

\end{tikzpicture}
  \end{center}
  \end{figure}
  Another option slightly more adapted to our homological computations is to replace $\gamma^{-1}\Sigma$ with a common refinement of the two fans defining $\Pt_1$ and $\Pt_2$ --- the proof of \Cref{model-works} ensures that the resulting fan $\Sigma_{\gamma}$ provides a relatively complete model $\Pt_\gamma$ for $E_{w, \Tate}^g$, and $\gamma: E_{w, \Tate}^g \to E_{w, \Tate}^g$ has a model as the toric map $\gamma: \Pt_\gamma \to \Pt_\Sigma$, which fits into the diagram
  \begin{center}
\begin{tikzcd}
                                             & \Pt_{\gamma} \arrow[ld, "\mathrm{id}"] \arrow[rd, "\gamma"] &              \\
\Pt_{\gamma^{-1}\Sigma} \arrow[rr, "\gamma"] &                                                             & \Pt_{\Sigma}
\end{tikzcd}
  \end{center}
  An example of this refinement (restricted to $\{1\} \times Y$) in the same context as \Cref{triangulation-diag} is pictured in \Cref{refinement-diag}.

  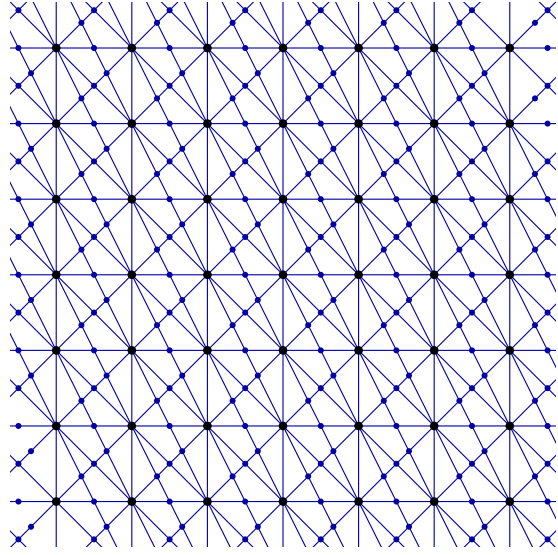
\begin{figure}
    \caption{An example of the polyhedral decomposition of $Y \otimes \BR \simeq \BR^2$ used to define $\Pt_\gamma$ in the case of \Cref{triangulation-diag}}\label{refinement-diag}  
    \begin{center}
\begin{tikzpicture}[
  edge/.style={line width=.42pt, draw=blue!65!black},
  old vertex/.style={circle, fill=black, inner sep=1.15pt},
  new vertex/.style={circle, fill=blue!65!black, inner sep=.8pt}
]

\begin{scope}[transparency group, scope fading=soft edge]
  \clip (-3.6,-3.6) rectangle (3.6,3.6);

  \foreach \k in {-8,...,8} {
    \draw[edge] (\k,-8) -- (\k,8);                 
    \draw[edge] (-8,\k) -- (8,\k);                 
    \draw[edge] (-8,{-8+\k}) -- (8,{8+\k});        
    \draw[edge] (-8,{8+\k}) -- (8,{-8+\k});        
    \draw[edge] (-8,{16+\k}) -- (8,{-16+\k});      
  }

  \foreach \i in {-4,...,4}
    \foreach \j in {-4,...,4}
      \node[old vertex] at (\i,\j) {};

  \foreach \i in {-5,...,4} {
    \foreach \j in {-5,...,4} {
      \node[new vertex] at ({\i+1/2},\j) {};
      \node[new vertex] at ({\i+1/3},{\j+1/3}) {};
      \node[new vertex] at ({\i+1/2},{\j+1/2}) {};
      \node[new vertex] at ({\i+2/3},{\j+2/3}) {};
    }
  }
\end{scope} 
\end{tikzpicture}
    \end{center}
  \end{figure}

  \subsection{Relatively complete model for the $\Gamma$-equivariant structure on $\LL$.}\label{gammamodel}

  For all $\gamma \in \Gamma$, we have a unique rigidified isomorphism $\gamma^* \LL \simeq \LL$. That is to say, we have an isomorphism \[(E_{w, \Tate}^g, \gamma^*\LL)  \simeq (E_{w, \Tate}^g, \LL) \text{ in the category } \DEG_{\ISgg}^{\spli}.\]
  It will be useful for the proof of the $\Gamma$-equivariant Fourier expansion principle to have a relatively complete model for that isomorphism, that is, for a relatively complete model for $D = (G, i: Y \to \Gt, \Lt)$ in $\DD_{\ISgg}^{\spli}$ such that $(E_{w, \Tate}^g, \LL) = M_{IS}(D)$ in an explicit way, and an isomorphism $\widetilde{\gamma}^*D \simeq D$ which induces the $\gamma$ -action on $\LL$ under $M_{IS}$. 

  Fix $\gamma \in \Gamma$, and use the relatively complete model $\Pt_\gamma$ adapted to $\gamma$ for $(\Gt, i)$ as before. After extending to $\DD_{\ISgg}^{\spli}$, we already constructed the object
  \[(\Gt, i, \Lt) \in \DD^{\spli}_{\ISgg}, \qquad M_{IS}(\Gt, i, \Lt) = (E_{w, \Tate}^g, \LL).\]
  This provides another point of view on the fact that $\LL$ admits a $\Gamma$-equivariant structure in the first place: recall that the $Y$-equivariant structure on $\Lt_K$ is given by the cocycle
  \[y \mapsto q^{Q'(y)}a^{b_{11}y_1 + \cdots + b_{gg}y_g}\underline{Z}^{B(y)},\]
  where whether a given $a \in \{\pm 1, \pm q^{1/2}\}$ is allowed in the first place depends on whether the required powers of $q$ are in $\frac{1}{w}\BZ$. Let $b = (b_{11}, \ldots, b_{gg})$. 
  
  Then (as a consequence of the fact that $\gamma$ preserves $Q$ and relatedly that $\gamma^\top B \gamma = B$) the pullback $\gamma^*\Lt$ is defined by a very similar cocycle
  \[y \mapsto q^{Q'(y)}q^{(b^\top y - b^\top \gamma y)/2}a^{b^\top \gamma y}\underline{Z}^{B(y)}.\]
  In particular, there is a rigidified isomorphism $f_\gamma: \gamma^*\Lt \to \Lt$ given by
  \begin{equation}\label{gammaequivstr}[\gamma^*\Lt \to \Lt] = \text{mult by }\begin{cases}
    1 \text{ if } a = \pm q^{1/2}\\
    \underline{Z}^{(b - \gamma^\top b)/2} \text{ if } a = \pm 1
  \end{cases}\end{equation}
  which is well-defined because $\gamma^\top b - b \in 2 \BZ^g$.

  We now construct a relatively complete model for the morphism $f_\gamma$ in $\DD_{\ISgg}^{\spli}$. Take the relatively complete model $\Pt_\gamma$ for $E_{w, \Tate}^g$. It admits an ample line bundle $\widetilde{\mathscr{N}}$ obtained by tensoring the pullbacks of the ample line bundles on $\Pt_1$ and $\Pt_2$ which we know exist by \Cref{ample-exists}. Since the fan $\Sigma_\gamma$ defining $\Pt_\gamma$ is a refinement of $\Sigma$, there is no problem defining the piecewise-linear function $\psi$ on $|\Sigma_\gamma|$ which we previously used to define the relatively complete models $\Lt$. The resulting line bundle (abusively called $\Lt$) still admits a $Y$-equivariant structure given by the same formulas, which makes it part of a relatively complete model for $\LL$, $(\Pt_\gamma, \Lt, \Nt)$. Recall that $\psi$ was defined to be piecewise-linear on $|\Sigma|$ (and hence on $\Sigma_\gamma$), and to satisfy
  \[\psi(1, w\underline{y}) = -wQ'(\uy).\]
  Then $\gamma^*\Lt$ is defined by $\gamma^*\psi$, which is given by
  \[\gamma^*\psi(1, w \underline{y}) = \psi(1, w \gamma \underline{y}) = -wQ'(\gamma \uy) = -wQ'(\uy) + w\uy^\top (\gamma^\top b -b)/2\]
  and it has a $Y$-equivariant structure that makes it a relatively complete model for $\gamma^*\Lt$.
  The same calculation as above with the $Y$-equivariant structures implicit in the definitions (following \Cref{KSDD}) shows that we have an isomorphism $\gamma^*\Lt \simeq \Lt$ given by the same formula \Cref{gammaequivstr}. The functor $M_{IS}$ sends that map to the $\gamma$-action on $\LL$ as an object of $\DEG_{\ISgg}^{\spli}$. 

  \subsection{Relatively complete model for the identity section}\label{rcm-identity}
  The following is implicit in the construction of $M(D)$. The identity section $e: \Spec \laurent \to E_{w, \Tate}^g$ is the base-change of $e: \Spec \pwr \to X_{w, g}$ coming from algebrizing the formal completion of the identity section of $\widetilde{G} \subset \Pt_{\ugamma}$ over $\Spec \pwr$. The group $\Gamma$ acts trivially on $e$, and the natural cover of $\Pt_\gamma$  restricts via $e$ to a cover of $\Spec \pwr$ which is just a finite number of copies of the entire scheme $\Spec \pwr$ (one for every cone in $\Sigma_{\ugamma}$ touching the line through $(1, \underline{0})$), and countably many copies of $\Spec \laurent$.
  
\subsection{Calculation of the $\Gamma$-equivariant coherent cohomology and the evaluation map in equivariant cohomology}
  We cannot carry out the refinement procedure that we used to define relatively complete models for the $\gamma$-action on $\LL$ infinitely many times (that is, for all $\gamma \in \Gamma$)--- so in order to build a complex that calculates $H^i_\Gamma(E_{w, \Tate}^g, \LL)$, we need to use many different relatively complete models that individually allow one to build complexes computing the maps $\gamma: H^i(E_{w, \Tate}^g, \LL)$. That is, we will take advantage of the spectral sequence 
  \[H^p(\Gamma, H^q(E_{w, \Tate}^g, \LL)) \implies H^{p + q}_\Gamma(E_{w, \Tate}^g, \LL)\]
  to allow us to compute $H^{p + q}_\Gamma$ without having a single \v{C}ech complex that can compute the entire $\Gamma$-action on $H^\bullet(E_{w, \Tate}^g, \LL)$. To tell the truth: if all we wanted to do was prove the abstract isomorphism
  \[H^i(\Gamma \ltimes Y, \mathcal{O}(\adictorus)) \simeq H^i_\Gamma(E_{w, \Tate}^g, \LL),\]
  then we could plausibly do with a lot less. We could just bootstrap from the isomorphisms
  \[H^q(Y, \mathcal{O}(\adictorus)) \simeq H^q(E_{w, \Tate}^g, \LL)\]
  to deduce the $\Gamma$-equivariant Fourier expansion principle by means of the spectral sequences
  \[H^p(\Gamma, H^q(E_{w, \Tate}^g, \LL)) \implies H^{p + q}_\Gamma(E_{w, \Tate}^g, \LL)\]
  and 
  \[H^p(\Gamma, H^q(Y, \mathcal{O}(\adictorus))) \implies H^{p + q}(\Gamma \ltimes Y, \mathcal{O}(\adictorus)).\]
  which are identified (including the transition maps) by means of the fact that the isomorphism $H^q(Y, \mathcal{O}(\adictorus)) \simeq H^q(E_{w, \Tate}^g, \LL)$ is $\Gamma$-equivariant, which can be proved by looking at a single $\gamma \in \Gamma$ at a time and using the model $\Pt_\gamma$. But in the promised \Cref{mainthm-equiv}, we need to prove further information, about the compatibility with the procedure of evaluating at the zero section. Since this is not a statement that has any meaning\footnote{In fact in applications such as \cite{KV}, taking $\Gamma$-equivariant cohomology is done \emph{in order to} make sure that evaluation at the identity section gives something that is meaningful.} at the level of $H^q(E_{w, \Tate}, \LL)$, at least when $q > 0$, we must take apart the spectral sequence a little bit more in order to prove the full detail of the claim.

  By ``take it apart,'' we mean to go all the way back to a suitable $E_0$ page of the spectral sequence. For this, we require a double complex replacing
  \[C^p(Y, \laurent \otimes \text{Cech complex for } (\Pt, \Lt))\]
  which also takes into account the $\Gamma$-action. Since there is no single relatively complete model that is adapted to the action of every $\gamma \in \Gamma$, the complex $C^p(\Gamma \ltimes Y, \cdots)$ wouldn't make any sense. On the other hand, note that the group of $\Gamma$ group $p$-cochains is nothing but a product $\prod_{\underline{\gamma} \in \Gamma^{p + 1}} (\cdots)$ --- the inside can be adapted to the particular value of $\gamma$ so that the $\Gamma$-group cohomology differentials may be defined appropriately.

  As a result, we can consider the double complex
  \[K^{p, q} := \prod_{\underline{\gamma} \in \Gamma^{p}} \mathrm{Tot}^q(C^{\bullet}(Y, \check{C}^\bullet_{\BZ((q^{1/w}))}(\mathcal{U}_{\Sigma_{\underline{\gamma}}}, \Lt^\wedge)))\]
  where the fan $\Sigma_{\gamma_0, \ldots, \gamma_p}$ is defined recursively by being a common refinement of $\Sigma_{\gamma_1, \ldots, \gamma_p}$ and $\gamma_0^{-1}\Sigma_{\gamma_1, \ldots, \gamma_p}$, and the vertical maps are defined by the alternating sum of
  \[p_i^{*}: \mathrm{Tot}^q(C^{\bullet}(Y, \check{C}^\bullet_{\laurent}(\mathcal{U}^{\Sigma_{\underline{\gamma}^i}}, \Lt^\wedge))) \to \mathrm{Tot}^q(C^{\bullet}(Y, \check{C}^\bullet_{\laurent}(\mathcal{U}^{\Sigma_{\underline{\gamma}}}, \Lt^\wedge)))\]
  for $\underline{\gamma} = (\gamma_0, \ldots, \gamma_p) \in \Gamma^{p+1}$, and
  \[\underline{\gamma}^i = \begin{cases}(\gamma_1, \ldots, \gamma_p), \qquad \text{if } i = 0 \\  
    (\gamma_0, \ldots, \gamma_{i-1}\gamma_i, \ldots, \gamma_p), \qquad \text{if } 1 \leq i \leq p-1\\
(\gamma_0, \ldots, \gamma_{p-1}), \qquad \text{ if } i = p. 
  \end{cases}.\] Here $p_0^*$ is defined to be the pullback-by-$\gamma_0$ map (defined via the $\Gamma$-equivariant structure on $\widetilde{L}$), while $p_i^*$ for $i \geq 1$ is just the pullback-by-id map, which has an obvious meaning here because of the fact that the fan $\Sigma_{\underline{\gamma}^i}$ is a refinement of $\Sigma_{\underline{\gamma}}$. And $\check{C}^\bullet_{\laurent}(\cdots)$ is shorthand for a \v{C}ech complex where each individual group of local sections has been base changed to $\laurent$ (recall the ``unbounded denominators'' subtlety from \Cref{nonequiv} that required us to consider complexes like this). As we will see in the proof of \Cref{K-cpt}, this is relevant to the $\Gamma$-equivariant cohomology of $\LL$ because it is related to the simplicial scheme $E_{w, \Tate}^g \times \Gamma^\bullet$.

  The basic idea of our proof of \Cref{mainthm-equiv} is to produce a commutative diagram
  \begin{center}
\begin{tikzcd}
{H^\bullet_\Gamma(E_{w, \Tate}^g, \LL)} \arrow[d, "\mathrm{ev}_0"] & {K^{\bullet, \bullet}} \arrow[r, "\mathrm{computes}", dashed] \arrow[l, "\mathrm{computes}", dashed] \arrow[d, "e^*"]                                              & {H^p(\Gamma \ltimes \BZ^g, \mathcal{O})} \arrow[d, "\mathrm{res}"] \\
{H^\bullet(\Gamma, \laurent)}                             & {\prod_{\underline{\gamma} \in \Gamma^{p}} C^q(\Spec \laurent, \mathcal{O})} \arrow[l, "\mathrm{computes}", dashed] \arrow[r, "\mathrm{computes}", dashed] & {H^\bullet(\Gamma, \laurent)}
\end{tikzcd} 
  \end{center}
  where each dotted arrow is an explicit certificate (for us, a zig-zag of quasiisomorphisms) for why a certain complex computes a certain cohomology group. It is thanks to the explicitness of those certificates that we can certify that the evaluation map on cohomology is given by the explicit restriction formula on group cohomology under the $\Gamma$-equivariant Fourier expansion principle. 
  \begin{lemma}\label{K-cpt}
The $\Gamma$-equivariant cohomology $H_\Gamma^\bullet(E_{w, \Tate}^g, \mathscr{L})$,  is computed by the double complex $K^{\bullet, \bullet}$. 
  \end{lemma}
  \begin{proof}

    We consider the equivariant cohomology $H^\bullet_\Gamma(E_{w, \Tate}^g, \LL)$ to be computed by a double complex derived from the simplicial scheme
    \[X_n = E_{w, \Tate}^g \times \Gamma^{n},\]
    which is the functorially-defined Godement resolution at each level. The horizontal maps are just the Godement resolution maps, and the vertical maps are the ones that arise from functoriality and the $\Gamma$-action on $\LL$ (plus the specific definitions of the vertical maps in $X_\bullet$ given above). Then there is a zig-zag of maps of double complexes that induce isomorphisms on cohomology of total complexes:
\begin{center}
\begin{tikzcd}[
  column sep=1em,
  row sep=2em,
  cells={nodes={inner xsep=.25em,inner ysep=2pt}}
] 
\displaystyle
\prod_{\underline\gamma\in\Gamma^{\bullet}}
  \laurent\otimes C^\bullet_{\gd,\LL}
  \arrow[d,"\mathrm{GAGA}"']
  & & \\ 

\displaystyle
\prod_{\underline\gamma\in\Gamma^{\bullet}}
  \laurent\otimes C^\bullet_{\gd,\LL^\wedge}
  \arrow[r]
  & 

  \displaystyle
  \prod_{\underline\gamma\in\Gamma^{p}}
    \laurent\otimes\Tot^\bullet
  \check C^\bullet\!\left(
    \mathfrak{P}_{\underline{\gamma}},\,
    \mathscr G_{\underline\gamma}^\bullet
  \right)
  &

  \displaystyle\prod_{\underline\gamma\in\Gamma^{p}}
    \laurent\otimes
  \check C^n\!\left(
    \mathfrak{P}_{\underline{\gamma}},\,
    \Lt^\wedge
  \right)
  \arrow[l]
  \arrow[d]
  \\

  & &
\begin{gathered}
  \displaystyle
  \prod_{\underline\gamma\in\Gamma^{\bullet}}
    \Tot^\bullet C^\bullet\!\bigl(Y, \check C^\bullet_{\laurent}(\!
    \Pt^\wedge_{\underline{\gamma}}, \widehat{\LL})
  \bigr)
\end{gathered}
  \arrow[d,equal]
  \\

  & & K^{\bullet,\bullet}
\end{tikzcd}
\end{center}
In this combination of zig-zags:
\begin{itemize}
\item The first vertical map is, at each level, the map induced by looking at the stalks making up the Godement resolution inside their formal completion. Formal GAGA says that this induces an isomorphism on cohomology, and it obviously respects all the maps in the double complex.
\item The zig-zag in the middle is the standard one comparing Cech cohomology to the cohomology computed by the Godement resolution for $(\mathfrak{P}, \widehat{\LL})$--- we refer to \cite[Ch. II, \S 5.4]{godement}. A comparison between the $\Gamma$-action on the stalks and on our bespoke Cech complex shows that it respects the double complex maps. Technically, at this point one must be careful in the comparison between \'etale and Zariski topologies, but this can entirely be skipped by simply refining the fan $\Sigma$ until the affine toric cover of $\Pt_{\Sigma}$ is already a suitable Zariski cover. This kind of refinement is performed in \cite{bosch} for example. 
\item The final vertical map is the product over $\underline{\gamma}$ of the maps of complexes that are well-defined thanks to the same Shapiro isomorphism used in the proof of the Fourier expansion isomorphism. 
\end{itemize}
  
  \end{proof}

  Similarly:
  \begin{lemma}
    The double complex $K^{\bullet, \bullet}$ also computes the group cohomology $H^\bullet(\Gamma \ltimes Y, \mathcal{O}(\adictorus))$. 
  \end{lemma}
  \begin{proof}
    We replicate the technique of the previous proof, starting on the bottom-right corner of the zig-zag diagram with $K^{\bullet, \bullet}$. Previously, we used the fact (from Shapiro's lemma) that the $E_1$-pages of the horizontal-first double complex spectral sequences of the complexes $C^\bullet(Y, \check C^\bullet_{\laurent}(\Pt^\wedge_{\ugamma}, \widehat{\LL}))$ (the plural being over each $\underline{\gamma} \in \Gamma^{\bullet}$) are concentrated in $p = 0$ where they are equal to $\laurent \otimes \check C^q(\mathfrak{P}_{\ugamma}, \widetilde{\LL}^\wedge)$. This implies that this complex maps quasiisomorphically into the total complex of the double complex as the $H^0$ of the horizontal differentials.

    We now copy this reasoning with the vertical-first spectral sequences for each $\underline{\gamma}$, from which we obtain the map of double complexes
    \[C^\bullet(\Gamma, C^\bullet(Y, \mathcal{O}(\adictorus))) = \prod_{\ugamma \in \Gamma^{\bullet + 1}} C^\bullet(\mathcal{O}(\adictorus)) \to K^{\bullet, \bullet}\]
    inducing a quasiisomorphism on total complexes. The double complex on the left calculates the cohomology of $\Gamma \ltimes Y$, by the Hochschild--Serre spectral sequence \cite[Ch. I]{HS}. 
  \end{proof}
  The proofs above show that $H^\bullet_\Gamma(E_{w, \Tate}^g, \mathscr{L}) \simeq H^\bullet(\Gamma \ltimes Y, \mathcal{O}(\adictorus))$ by explicitly identifying both with the cohomology of $K^{\bullet, \bullet}$ via explicit zig-zags. On the equivariant coherent cohomology side, that zig-zag starts with a functorial double complex derived from 
  \[E_{w, \Tate}^g \times \Gamma^{p+1} \overset{p_0, \ldots, p_p}{\to} E_{w, \Tate}^g \times \Gamma^{p}\]
  by choosing for row $p$ a functorial complex (in particular the global sections of the Godement resolution) computing in the $\underline{\gamma}$-coordinate the cohomology of $\mathscr{L}$, and where the vertical maps are the ones induced by that functorial complex. We used the \emph{definition} of the equivariant cohomology as being the cohomology of that double complex. We are interested in calculating the induced map on $\Gamma$-equivariant cohomology of pullback along the identity section $e: \Spec \laurent \to E_{w, \Tate}^g$. The target of that map is
  \[H^\bullet_{\Gamma}(\Spec \laurent, e^*\LL \simeq \mathcal{O}_{\Spec \laurent}) = H^\bullet(\Gamma, \laurent)\]
  where $\laurent$ is endowed with the trivial action of $\Gamma$, since $\Gamma$ preserves the rigidification. This cohomology is again naturally computed in two ways: (1) by the base change to $\laurent$ of the $q^{1/w}$-adic formal completion of the Cech complex on $\Spec \pwr$ for $e^*\Lt$, using the cover of $\Spec \pwr$ derived in \Cref{rcm-identity}. And (2), by the functorial double complex induced via levelwise Godement resolutions on the simplicial scheme $\Spec \laurent \times \Gamma^{\bullet}$. The method (2) is the ``official'' functorial definition that is the definition of how the pullback $e^*: H^\bullet_\Gamma(E_{w, \Tate}^g, \LL) \to H^\bullet(\Gamma, \laurent)$ is \emph{defined}: one just pulls back under the natural maps of functorial resolutions induced by the map of simplicial schemes
\begin{center}
\begin{tikzcd}
E_{w, \Tate}^g \times \Gamma^2 \arrow[d]                                      & (\Spec \laurent) \times \Gamma^2 \arrow[l] \arrow[d] \\
E_{w, \Tate}^g \times \Gamma \arrow[d] \arrow[d] \arrow[d] \arrow[d, no head] & (\Spec \laurent) \times \Gamma \arrow[l] \arrow[d]   \\
E_{w, \Tate}^g                                                                & \Spec \laurent \arrow[l]                            
\end{tikzcd}
  \end{center}
coming from $\Gamma^{\bullet}$ copies of the identity section at each level. One verifies that under the zig-zag that certifies the fact that $K^{\bullet, \bullet}$ computes $H^{\bullet}_{\Gamma}(E_{w, \Tate}^g, \LL)$, $e^*$ is computed by the pullback procedure from $K^{\bullet, \bullet}$ to the complex in (1). More specifically, there is a natural commutative diagram

  \begin{center}
\begin{tikzcd}[
  column sep=2em,
  row sep=2.1em,
  cells={nodes={inner xsep=.15em,inner ysep=2pt}}
]
\displaystyle
\prod_{\underline\gamma\in\Gamma^{\bullet}}
  \laurent \otimes C^\bullet_{\gd,\LL}
  \arrow[r,"e^*"]
  \arrow[d,"\text{formal GAGA}"']
&
\displaystyle
\prod_{\underline\gamma\in\Gamma^{\bullet}}
  \laurent \otimes C^\bullet_{\gd,e^*\LL}
  \arrow[d,"\text{formal GAGA}"]
\\
\displaystyle
\prod_{\underline\gamma\in\Gamma^{\bullet}}
  \laurent \otimes
  C^\bullet_{\gd,\LL^\wedge}
  \arrow[r,"\widehat e^*"]
  \arrow[d]
&
\displaystyle
\prod_{\underline\gamma\in\Gamma^{\bullet}}
  \laurent \otimes
  C^\bullet_{\gd,\widehat e^*\widehat{\LL}}
  \arrow[d]
\\
\displaystyle
\prod_{\underline\gamma\in\Gamma^{\bullet}}
  \laurent \otimes
  \Tot^\bullet\check C^\bullet\!\left(
    \mathfrak{P}_{\underline\gamma},\,
    \mathscr G^\bullet_{\ugamma}
  \right)
  \arrow[r,"\widehat e^*"]
&
\displaystyle
\prod_{\underline\gamma\in\Gamma^{\bullet}}
  \laurent \otimes
  \Tot^\bullet\check C^\bullet\!\left(
    \Spf\pwr,\,
    \mathscr G^\bullet_{\widehat e^*\widehat{\LL}}
  \right)
\\
\displaystyle
\prod_{\underline\gamma\in\Gamma^{\bullet}}
  \laurent \otimes
  \check C^\bullet\!\left(
    \mathfrak P_{\underline\gamma},\,
    \widehat{\LL}
  \right)
  \arrow[u]
  \arrow[r,"\widehat e^*"]
  \arrow[d]
&
\displaystyle
\prod_{\underline\gamma\in\Gamma^{\bullet}}
  \laurent \otimes
  \check C^\bullet\!\left(
    \Spf\pwr,\,
    \widehat e^*\widehat{\LL}
  \right)
  \arrow[u]
  \arrow[d]
\\
\displaystyle
\prod_{\underline\gamma\in\Gamma^{\bullet}}
  \Tot^\bullet C^\bullet\!\left(
    Y,\,
    \check C_{\laurent}^\bullet\!\left(
      \widehat{\widetilde P}_{\underline\gamma},\,
      \widehat{\LL}
    \right)
  \right)
  =K^{\bullet,\bullet}
  \arrow[r,"\operatorname{ev}"]
&
\displaystyle
\prod_{\underline\gamma\in\Gamma^{\bullet}}
  \check C_{\laurent}^\bullet\!\left(
    \Spf\pwr,\,
    \widehat e^*\widehat{\mathcal L}
  \right)
\end{tikzcd}
  \end{center}
  In particular, the map of double complexes $\mathrm{ev}$ computes the natural evaluation map $H^\bullet_\Gamma(E_{w, \Tate}^g,  \LL) \to H^\bullet(\Gamma, \laurent)$ under the explicit zig-zag used in the proof of the comparison.

  On the other hand, the map $\mathrm{ev}$ also computes the natural restriction map $H^{\bullet}(\Gamma \ltimes \BZ^g, \mathcal{O}(\adictorus)) \to H^\bullet(\Gamma, \laurent)$. Indeed, there is a commutative diagram extending the zig-zag from the proof that $K^{\bullet, \bullet}$ computes $H^\bullet(\Gamma \ltimes \BZ^g, \mathcal{O}(\adictorus))$,
  \begin{center}
\begin{tikzcd}
{\displaystyle\prod_{\ugamma} \Tot^\bullet C^\bullet(Y, \check C^\bullet_{\laurent}(\Pt^\wedge_{\ugamma}, \Lt^\wedge))} \arrow[d] \arrow[rr] &  & {\displaystyle\prod_{\ugamma \in \Gamma^{\bullet}} \check C^\bullet_{\laurent}(\Spf \pwr, \mathcal{O})} \arrow[d] \\
{C^\bullet(\Gamma, C^\bullet(Y, \mathcal{O}(\adictorus)))} \arrow[rr]                                                                               &  & {C^\bullet(\Gamma, \laurent)}                                                                                               
\end{tikzcd}
  \end{center}

  \section{The Heisenberg group action on Fourier expansions}\label{heis} 
  We continue to assume that $\LL$ is as in the previous section, except with the even further hypothesis that $Q$ is a nondegenerate quadratic form whose discriminant has absolute value $\disc > 0$. We denote by $\langle -, - \rangle$ the pairing $\langle x, y \rangle := x^\top B y$. We change the base ring from $\BZ$ to $\BZ' := \BZ[1/(2 \disc), \zeta_\disc]$. We assume that $w$ is divisible by $\disc$, though this can obviously be weakened to only requiring divisibility by the exponent of the discriminant group of $Q$. We refer to \cite[\S 3, III]{KV} for much of this setup.  Also assume for now that $\disc$ is odd.

  The group $K(\LL)$ is a finite flat $\laurentp$-group scheme, and it even satisfies the containment of $\laurentp$-group schemes
  \[K(\LL) \subset E_{w, \Tate}[\disc]^g \subset E_{w, \Tate}^g.\]
  In fact, since we have base changed to $\BZ'$, the group scheme $K(\LL)$ is constant and given by the square of the discriminant group of $Q$,
\[\kappa_B \times \kappa_B := (B^{-1}\BZ^g/\BZ^g) \times (B^{-1}\BZ^g/\BZ^g) = (Y^\vee/Y) \times (Y^\vee/Y).\]
  Specifically, there is an isomorphism
  \[(Y^\vee/Y) \times (Y^\vee / Y) \to K(\LL) \subset E_{w, \Tate}[\disc]^g\]
  given in the Tate curve coordinates by a convention that will suit our calculations, namely
  \[(x, y) \mapsto q^{-\underline{x}}e(\underline{y}), \qquad e(\underline{y}) := (\exp(2 \pi i y_i))_{i = 1}^g.\]
  Here $\exp(2 \pi i y_i)$ is treated as an abstract element of the ring $\BZ'$, that is, a specific power of $\zeta_\disc$. 

  By definition, we know that for all $z = (x, y) \in K(\LL)$, there is an isomorphism $T_z^*\LL \overset{\varphi}{\simeq} \LL$. The group scheme $\mathcal{G}(\LL)$ fits into the short exact sequence
  \[1 \to \BG_m \to \mathcal{G}(\LL) \to K(\LL) \to 0,\]
  which admits a unique symmetric theta structure, that is, an isomorphism with the standard Heisenberg group from \cite[III]{KV}
  \[1 \to \BG_m \to H_B \to \kappa_B \times \kappa_B \to 0\]
  that respects the symmetric structures of both and restricts on $K(\LL)$ to the isomorphism given above. This uses the hypothesis that $\disc$ is odd --- see \cite{KV}.

  \begin{definition}\label{thetaaction}
    A theta structure $H_B \simeq \mathcal{G}(\LL)$ induces ``the'' action of $H_B$ on $H^\bullet(E_{w, \Tate}^g, \LL)$, which is by $(\varphi, z) \in \mathcal{G}(\LL)$ acting on $H^\bullet(E_{w, \Tate}^g, \LL)$ by
    \[H^{\bullet}(E_{w, \Tate}^g, \LL) \overset{T_z^*}{\to} H^\bullet(E_{w, \Tate}^g, T_z^*\LL) \overset{H^\bullet(\varphi)}{\simeq} H^\bullet(E_{w, \Tate}^g, \LL).\]
  \end{definition}

  The purpose of this section is to calculate explicit formulas for this action on $H^\bullet(E_{w, \Tate}^g, \LL)$ in terms of the Fourier expansion principle we already proved, \Cref{mainthm}. As usual, we will make use of relatively complete models of all the objects involved, namely of $T_z: E_{w, \Tate}^g \to E_{w, \Tate}^g$ and of the isomorphisms $T_z^*\LL \simeq \LL$. This is extremely similar to\footnote{in fact it is simpler because it doesn't require any procedure of choosing refinements of fans in order to produce the required relatively complete models} the exercise we already carried out of producing relatively complete models for the action of $\gamma \in \Gamma$ on $E_{w, \Tate}^g$ and the $\Gamma$-equivariant structure $\gamma^*\LL \simeq \LL$, so we are intentionally brief. 

  \subsection{Relatively complete models for the action of $K(\LL)$}

  Let $z = (x, y) \in \kappa \times \kappa$. Take the relatively complete model $\Pt_\Sigma$ for $E_{w, \Tate}^g$. The toric scheme $X_\Sigma$ admits a self-map given by $(1, y) \mapsto (1, y + wx)$. After base change to $\Pt_{\Sigma}$, this is a relatively complete model for translation by $q^{x} \in (q^{\frac{1}{\disc}\BZ})^{g}$ on $E_{w, \Tate}^g$.

  Similarly, $X_\Sigma$ admits a self-map given on each affine toric chart by $\underline{Z} \mapsto e(\underline{y})\underline{Z}$. This induces on $\Pt_{\Sigma}$ the multiplication-by-$e(\underline{y})$ map. Therefore, we may construct relatively complete models for all the translation maps $T_z: E_{w, \Tate}^g \to E_{w, \Tate}^g$, $z \in \kappa \times \kappa$.

  It remains to produce relatively complete models for at least some choice of isomorphisms $T_z^* \LL \to \LL$. By this we really mean: we will produce sections of $T_z^*\LL^\vee \otimes \LL$ over $G_{\laurent}$ by means of explicit construction at the level of relatively complete models (or even, if desired, at the level of Fourier expansions of global sections using \Cref{mainthm}).

  We retain our prior choice of relatively complete model $\Lt$ for $\LL$. Then $T_z^*\LL$ has a relatively complete model, induced by the piecewise linear function $z^*\psi \in \PL(\Sigma)$, namely the piecewise-linear function on the cone over the $wx$-translate of the decomposition of $Y \otimes \BR$ from \Cref{KSDD} with vertices at $wY$ with 
  \[z^*\psi(1, wy + wx) = \psi(1, wy).\]

  One verifies that for all $u \in Y \otimes \BR$, \begin{equation}\label{error}(\psi - z^*\psi)(1, u) = wQ'(-\underline{x}) - \langle u, \underline{x}\rangle + (wY\text{-periodic \emph{bounded} function})(u).\end{equation}
  The rational function $q^{Q'(-x)}Z^{-Bx}$ is a meromorphic section of the divisor $T_z^*\Lt^\vee \otimes \Lt$, and this is the natural choice of multiple of $Z^{-Bx}$ (the natural intertwiner between the $Y$-equivariant bundles  $T_z^*\Lt|_{\Gt_K}$ and $\Lt|_{\Gt_K}$) since the bounded error in \Cref{error} is in fact nonnegative when $u = 0$ --- so $q^{Q'(-x)}Z^{-Bx}$ extends to a morphism $T_z^*\Lt|_{\Gt} \to \Lt|_{\Gt}$. In any event, the boundedness in \Cref{error} ensures that multiplication by some large enough power of $q$ turns $q^{Q'(-x)}Z^{-Bx}$ into a $Y$-equivariant global section of $T_z^* \Lt^\vee \otimes \Lt$, which descends and algebrizes following the usual combination of \Cref{ourdef} with formal GAGA \cite[Proposition (5.1.2)]{EGAIII} (already used in \Cref{nonequiv}) to a morphism $T_z^*\LL \to \LL$. After passing to $G_{\laurent}$, one is allowed to divide out by the required powers of $q$ to get back to what is reasonable to call ``$q^{Q'(-x)}Z^{-Bx}$''. Over $G_{\laurent}$ this is an isomorphism $T_z^*\LL \to \LL$ and therefore lives in $\mathcal{G}(\LL)$ over $z \in K(\LL)$, as the inverse can be constructed in exactly the same way thanks to the boundedness in \Cref{error}. It is therefore admissible to refer to specific sections of $\mathcal{G}(\LL)$ lying over $z = q^{-x}e(y) \in K(\LL)$ by expressions like $q^{Q'(-x)}Z^{-Bx}$ or any $\laurentp^{\times}$-multiple thereof.

  As a result, we may construct the unique (by \cite[III]{KV}) symmetric theta structure on $\LL$,
  \[\varphi: H_B \simeq \mathscr{G}(\LL),\]
  by the rule
  \begin{equation}\label{thetastruct}(\lambda, x, y) \mapsto \begin{cases}\lambda q^{Q'(-x)}\underline{Z}^{-Bx}\exp(-2\pi i \langle x, y \rangle - (\disc + 1)\pi i b^\top y) \text{  if } a = 1\\
    \lambda  (-1)^{\disc b^\top x}q^{Q'(-x)}\underline{Z}^{-Bx}\exp(-2\pi i \langle x, y \rangle -(\disc + 1)\pi i b^\top y)) \text{  if } a = -1\\
    \lambda  q^{Q(x)}\underline{Z}^{-Bx}\exp(- 2\pi i \langle x, y \rangle) \text{  if } a = q^{1/2}\\ 
    \lambda  (-1)^{\disc b^\top x} q^{Q(x)}\underline{Z}^{-Bx}\exp(- 2\pi i \langle x, y \rangle) \text{  if } a = -q^{1/2}\\
\end{cases}
  \end{equation}
  The explanation of the various extra factors in different cases is as follows:
  \begin{itemize}
  \item The extra factors $\exp(-2\pi i \langle x, y \rangle)$ and the possible $(-1)^{\disc b^\top x}$ (in the cases $a = -1, -q^{1/2}$) are there to ensure that the theta structure is a genuine homomorphism $H_B \to \mathcal{G}(\LL)$. In particular, without the extra $(-1)^{\disc b^\top x}$, it wouldn't be the case that the image in $\mathcal{G}(\LL)$ of an element of $H_B$ doesn't depend on the choice of representative in $\BG_m \times B^{-1}\BZ^g \times B^{-1}\BZ^g$ (the fact that $\disc$ is odd makes this nontrivial).
  \item The extra factors of $\exp(-(\disc + 1)\pi i b^\top y)$ in the cases $a = \pm 1$ are there to ensure that the theta structure is symmetric as required. Indeed, when $a = \pm 1$, the normalized symmetric structure $[-1]^*\LL \to \LL$ is given by multiplication by $\underline{Z}^{b}$ rather than $1$ and this must be taken into account when conjugating the action of $(t, x, y)$ by $\vartheta(Z) \mapsto \underline{Z}^{b} \vartheta(Z^{-1})$.  
  \end{itemize}
  
  The factors of $\disc + 1$ provide a canonical choice of square-root of $\zeta_\disc$ that doesn't require taking more than a $\disc$-th root of unity and hence remains well-defined over $\laurentp$

  The theta structure \Cref{thetastruct} defines via \Cref{thetaaction} and \Cref{mainthm} an action of $H_B$ on $H^\bullet(\BZ^g, \mathcal{O}(\adictorus))$ which is given on cocycles by multiplication by exactly the formula \Cref{thetastruct} and pullback by $\underline{Z} \mapsto q^{-x}e(y)\underline{Z}$.

  When $\disc$ is even, one can still produce a symmetric theta structure on $\LL$ (as in \cite[Appendix B]{KV}), but at the cost of making sure the width $w$ is divisible by $2\disc$, and allowing a $2\disc$-th root of unity in the base ring. Since $\disc$ is now even, division by $2$ in $\kappa_B$ is not as simple a matter as multiplication by $(\disc + 1)/2$. Instead, as seen in \cite[Appendix B]{KV}, one has to reach outside of $\kappa_B$ for preimages under multiplication by $2$. This causes some non-uniqueness in the symmetric theta structure, which form a $\Hom(\kappa_B \times \kappa_B, \mu_2)$-torsor due to the standard short exact sequence of \cite[Lemma III$_a$.11]{KV}. In any event, one such theta structure may be constructed as follows. Let $c \in \BZ^g$  such that $Bc \equiv b \mod 2$. Then we may construct theta structures given as follows, in the same language as above:
  \begin{equation}\label{even_thetastruct}
      (\lambda, x, y) \mapsto
    \begin{cases}
      \lambda q^{Q'(-x)}\underline{Z}^{-Bx}\exp(- 2 \pi i \langle x, y \rangle - \pi i b^\top y)(-1)^{\langle c, y\rangle}, \text{ if } a = 1\\
      \lambda q^{Q'(-x)}\underline{Z}^{-Bx}\exp(- 2 \pi i \langle x, y \rangle - \pi i b^\top y)(-1)^{\langle c, x + y\rangle}, \text{ if } a = -1\\\lambda q^{Q(x)}\underline{Z}^{-Bx}\exp(- 2 \pi i \langle x, y \rangle), \text{ if } a = q^{1/2}\\\lambda q^{Q(x)}\underline{Z}^{-Bx}\exp(- 2 \pi i \langle x, y \rangle)(-1)^{\langle c, x\rangle}, \text{ if } a = -q^{1/2}
  \end{cases}.
  \end{equation}
  As before, the extra factors $\exp(- \pi i b^\top y)$ are there to ensure ensure that the theta structure is symmetric when $a = \pm 1$, as in that case the symmetric structure $[-1]^*\LL \to \LL$ also involves multiplication by $\underline{Z}^b$; and the extra factors of the form $(-1)^{\langle c, -\rangle}$ exist to make sure that these are well-defined group homomomorphisms valued in $\mathcal{G}(\LL)$ (in particular that they only depend on $x, y$ as elements of $\kappa_B$ and not just $B^{-1}\BZ^g$). 
\begin{appendices}
\section{Appendix on line bundles on torus embeddings}\label{toric-pdf}
Let $N$ and $M$ be abstract finite-rank lattices. We view $M$ to be the character lattice of a split torus $T$ over $\BZ$, and $N$ to be its dual.
For a (possibly infinite) fan $\Sigma$ in $N \otimes \BR$, as in the literature on toric varieties \cite{demazure, fulton}, we may produce a scheme $X_\Sigma$ produced by gluing together the various toric schemes $X_\sigma = \Spec \BZ[\sigma^\vee \cap M]$.

For a ray $\rho = \BR_{\geq 0} \cdot v_\rho$ in $\sigma$, there is a way to associate a codimension-1 irreducible closed subset $D_\rho$ of $X_\Sigma$. This is by defining, for $\sigma$ that have $\tau$ as a face,
\[D_\rho \cap U_\sigma = \Spec \BZ[\rho^\perp \cap \sigma^\vee \cap M],\]
which is a closed subscheme of $U_\sigma = \Spec \BZ[\sigma^\vee \cap M]$ via the (obviously surjective) map of rings
\[\BZ[\sigma^\vee \cap M] \to \BZ[\rho^\perp \cap \sigma^\vee \cap M]\]
given by sending all $m \in \sigma^\vee \cap M$ which do not pair to zero with $v_\rho$ to $0$. The $D_{\rho}$ are $T$-invariant Cartier divisors on $X_\Sigma$. 

Let $\PL(\Sigma)$ denote the abelian group of piecewise-linear continuous functions $\varphi: |\Sigma| \to \BR$, with the further property that for all cones $\sigma \in \Sigma$, $\varphi|_\sigma$ is the restriction to $\sigma$ of a linear function $\widetilde{\varphi|_{\sigma}}: N \otimes \BR \to \BR$, which itself arises from base-change from a linear $m: N \to \BZ$ (i.e. $\varphi$ is linear on all the cones and integer-valued on $N$).

The reason for considering such functions is
\begin{lemma}
$\PL(\Sigma)$ is in natural bijection with the group of $T$-invariant Cartier divisors on $X_\Sigma$ supported on $X_\Sigma - T$. 
\end{lemma}

\begin{corollary}\label{picardgp}
The Picard group $\Pic(X_\Sigma)$ is in natural bijection with $\PL(\Sigma) / M$, where $M$ maps to $\PL(\Sigma)$ (injectively if the cones of $\Sigma$ span all of $N \otimes \BR$) via sending $m \in M$ to the corresponding globally linear map on $|\Sigma| \subset N \otimes \BR$. 
\end{corollary}

The explicit bijection between $\PL(\Sigma)$ and $\CDiv_T(X_\Sigma)$ is chosen to take $\psi$ to $D = \sum_\rho -\psi(v_\rho) D_\rho$, so that the various $m_\sigma \in M$ defining $\psi$ are choices of single generating sections of $\mathcal{O}(D)$ over $U_\sigma$. This is the negation of a more naive convention.

Everything above is restated from the standard literature to reassure the reader that it goes through in the case where $\Sigma$ is infinite. On the other hand, the ampleness criteria in the literature  do actually use the finiteness of $\Sigma$, which is why this appendix had to be written.

Let $\Sigma$ be a possibly infinite fan in $N \otimes \BR$. In the usual ampleness criteria involving strict convexity of piecewise linear functions (\cite[\S 4.4, Corollaire 1]{demazure} and its descendents including \cite[Lemma on p.70]{fulton}), one assumes that $\Sigma$ is complete. One reason to do this is that at least a textbook definition of ``ample'' \cite{vakil} technically only applies to schemes that are proper over an affine scheme. Our goal is to construct ample line bundles using the definition from \cite{mumford_compositio} which applies without quasicompactness or properness hypotheses:
\begin{definition}
A line bundle $\mathscr{L}$ on a scheme $X$ is \emph{ample} if the open sets $X_s$ form a basis for the Zariski topology on $X$, as $s$ runs over all global sections of $\mathscr{L}^{\otimes N}$, $N \geq 0$. 
\end{definition}
\begin{lemma}\label{ampleness}
Let $\mathscr{L}$ be the line bundle on $X_\Sigma$ coming from $\psi \in \PL(\Sigma)$. Let $a_\rho := -\psi(v_\rho)$ for all rays $\rho$ in $\Sigma$, so that the Cartier divisor corresponding to $\psi$ is $\sum_\rho a_\rho D_\rho$. Suppose that both of the following conditions are satisfied:
  \begin{enumerate}
  \item $\psi$ is strictly convex, that is, for all maximal cones $\sigma \in \Sigma$ and rays $\rho \not\in\sigma$, $\left(\widetilde{\psi|_\sigma}\right)(v_\rho) > -a_\rho$; and
  \item The quantities \[\frac{-\langle m, v_\rho \rangle}{\left(\widetilde{\psi|_\sigma}\right)(v_\rho) + a_\rho}\]
    are bounded above, as $\sigma$ ranges over the maximal cones in $\Sigma$, $\rho$ ranges over the rays not contained in $\sigma$, and $m$ ranges over a (fixed a priori) finite set of semigroup generators for $\sigma^\vee \cap M$. 
  \end{enumerate}
  Then, $\mathscr{L}$ is ample on $X_\Sigma$. 
\end{lemma}
Note that condition (2) is always true when the fan $\Sigma$ is finite, as the set that we want that quantity to be bounded over is finite.
\begin{proof}
  For all maximal cones $\sigma \in \Sigma$, let $m_\sigma \in M$ be such that $\widetilde{\varphi|_\sigma} = \langle m, - \rangle$. Strict convexity ensures that $\chi^{m_\sigma} \in H^0(X, \mathscr{L})$, and that $U_\sigma = X_{\chi^{m_\sigma}}$ where $\chi^{m_\sigma}$ is considered as a global section of $\mathscr{L}$.  If the fan $\Sigma$ were finite then this fact (having the various $X_s$ cover $X$) would be enough to have $\mathscr{L}$ ample --- the standard ``QCQS'' argument is made even slightly simpler in this situation: the local sections of $\mathscr{L}$ over $U_{\sigma} = X_{\chi^{m_\sigma}}$, which are the same as the regular functions on the affine $U_\sigma$ (as $\mathcal{L}|_{U_\sigma} = \mathcal{O}_{U_\sigma}$), generate the topology on $U_\sigma$, and those local sections $s \in \mathcal{O}_{U_\sigma}$ have the property that $s \cdot (\chi^{m_\sigma})^N$ extends toa global section of $\mathscr{L}^N$ for $N \gg_s 0$; since $s \cdot (\chi^{m_\sigma})^N$ has the same nonvanishing locus within $U_\sigma$ as $s$, adding in these global sections is enough to generate the topology on $X_\Sigma$. The part of this argument that fails when $X_\Sigma$ is not quasicompact is the assertion that $s \in H^0(U_\sigma, \mathcal{O})$ all have the property that $s \cdot (\chi^{m_\sigma})^N$ extend to a global section of $\mathscr{L}^N$ --- this part is proved by gluing together various $s \cdot (\chi^{m_\sigma})^{N_{\sigma'}}$ that extend $s$ to the affine $U_{\sigma \cap \sigma'}$, and you need the finiteness of the fan $\Sigma$ to take the maximum of the $N_\sigma'$. 

  So our goal is to write elements of the coordinate ring $H^0(U_\sigma, \mathcal{O})$ as quotients of elements of $H^0(X_\Sigma, \mathscr{L}^{\otimes N})$ by $(\chi^{m_\sigma})^N$, which we will just translate into part (2) of the hypothesis. Taking $N$ larger than the upper bound from (2), which doesn't depend on $\sigma$ (that is the whole point of assumption (2)) we know that by (2), for all $m$ in a finite generating set of the semigroup  $\sigma^\vee \cap M$ and $\rho \not\in \sigma$, 
  \[\langle m + Nm_\sigma, v_\rho\rangle = \langle m, v_\rho\rangle + N \langle m_\sigma, v_\rho \rangle = \langle m, v_\rho\rangle + N \left(\widetilde{\varphi|_\rho}\right)(v_\sigma) \geq -N a_\rho.\]   
  Similarly, by definition of $\sigma^\vee$ and $m_\sigma$, if $\rho \in \sigma$, then
  \[\langle m + Nm_\sigma, v_\rho\rangle = \langle m, v_\rho\rangle + N \langle m_\sigma, v_\rho \rangle \geq - N a_\rho. \]
  In other words, $\chi^m \cdot (\chi^{m_\sigma})^N$ is a global section of $\mathscr{L}^{\otimes N}$. Since $m$ runs over the set of semigroup generators of $\sigma^\vee \cap M$, we deduce that $\sigma^\vee \cap M$ is generated by finitely many elements of the form $u - Nm_\sigma$, where $u \in M$ is a global sections of $\mathscr{L}^{\otimes N}$. This proves that the ring defining $U_\sigma$ is generated as a $\BZ$-algebra by functions of the form $\frac{\text{global section of } \mathscr{L}^{\otimes N}}{m_\sigma^N}$, which proves the desired ampleness of $\mathscr{L}$ (since the bound $N$ is uniform over the choice of $\sigma$).
\end{proof}
\end{appendices}
\bibliographystyle{plain} 
\bibliography{degens_bib}

\begin{biblist}

  \bib{SGA3}{book}{
    label = {SGA 3}
    AUTHOR = {M. Artin and J.E. Bertin and M. Demazure and P. Gabriel and  A. Grothendieck and M. Raynaud and J.-P. Serre}, 
     TITLE = {Sch\'emas en groupes. {F}asc. 3: {E}xpos\'es 8 \`a{} 11},
      NOTE = {Troisi\`eme \'edition,
              S\'eminaire de G\'eom\'etrie Alg\'ebrique de l'Institut des
              Hautes \'Etudes Scientifiques, 1963/64, dirig\'e{} par Michel
              Demazure et Alexander Grothendieck},
 PUBLISHER = {Institut des Hautes \'Etudes Scientifiques, Paris},
      YEAR = {1964},
     PAGES = {ii+178 pp. (not consecutively paged)},
   MRCLASS = {14.50},
  MRNUMBER = {207705},
} 

  \bib{sga71}{book}{
    label = {SGA {$7_1$}} 
     TITLE = {Groupes de monodromie en g\'eom\'etrie alg\'ebrique. {I}},
    SERIES = {Lecture Notes in Mathematics},
    VOLUME = {Vol. 288},
      NOTE = {S\'eminaire de G\'eom\'etrie Alg\'ebrique du Bois-Marie
              1967--1969 (SGA 7 I),
              Dirig\'e{} par A. Grothendieck. Avec la collaboration de M.
              Raynaud et D. S. Rim},
 PUBLISHER = {Springer-Verlag, Berlin-New York},
      YEAR = {1972},
     PAGES = {viii+523},
   MRCLASS = {14-06}, 
  MRNUMBER = {354656},
}
  \bib{EGAII}{article}{
    label = {EGA II}
    AUTHOR = {Alexander Grothendieck (r\'edig\'e avec la collaboration de Jean Dieudonn\'e)}, 
     TITLE = {\'El\'ements de g\'eom\'etrie alg\'ebrique. {II}. \'Etude
              globale \'el\'ementaire de quelques classes de morphismes.},
   JOURNAL = {Inst. Hautes \'Etudes Sci. Publ. Math.},
  FJOURNAL = {Institut des Hautes \'Etudes Scientifiques. Publications
              Math\'ematiques},
    NUMBER = {8},
      YEAR = {1961},
     PAGES = {222},
      ISSN = {0073-8301,1618-1913},
   MRCLASS = {14.55},
  MRNUMBER = {217084},
       URL = {http://www.numdam.org/item?id=PMIHES_1961__8__222_0},
}

  \bib{EGAIII}{article}{
    label = {EGA {$\text{III}_1$}}
    AUTHOR = {Alexander Grothendieck (r\'edig\'e avec la collaboration de Jean Dieudonn\'e)},
     TITLE = {\'El\'ements de g\'eom\'etrie alg\'ebrique: {III}. \'Etude
              cohomologique des faisceaux coh\'erents, Premi\`ere partie.},
   JOURNAL = {Publications math\'ematiques de l'Institut des Hautes \'Etudes Scientifiques},
  FJOURNAL = {Institut des Hautes \'Etudes Scientifiques. Publications
              Math\'ematiques},
    NUMBER = {11},
      YEAR = {1961},
     PAGES = {5-167},
      ISSN = {0073-8301,1618-1913},
   MRCLASS = {14.55},
  review = {\MR{217085}},
       URL = {http://www.numdam.org/item?id=PMIHES_1961__11__167_0},
}

\bib{EGAIV4}{article}{
  LABEL = {EGA {$\text{IV}_4$}}
AUTHOR = {Alexander Grothendieck (r\'edig\'e avec la collaboration de Jean Dieudonn\'e)},
     TITLE = {\'El\'ements de g\'eom\'etrie alg\'ebrique: {IV}. \'Etude
              locale des sch\'emas et des morphismes de sch\'emas, Quatri\`eme partie},
   JOURNAL = {Publications math\'ematiques de l'Institut des Hautes \'Etudes Scientifiques},
  FJOURNAL = {Institut des Hautes \'Etudes Scientifiques. Publications
              Math\'ematiques},
    NUMBER = {32},
      YEAR = {1967},
     PAGES = {5--361},
      ISSN = {0073-8301,1618-1913},
   MRCLASS = {14.55},
  review = {\MR{238860}},
MRREVIEWER = {J.\ P.\ Murre},
URL = {http://www.numdam.org/item?id=PMIHES_1967__32__361_0}
}

\end{biblist}

\end{document}